\documentclass[12pt,a4paper]{article}

\usepackage{latexsym}
\usepackage{a4wide}
\usepackage{amscd}
\usepackage{graphics}
\usepackage{amsmath}
\usepackage{amssymb}
\usepackage{soul,xcolor}
\usepackage{esint}
\usepackage{mathrsfs}
\usepackage{amsthm}
\usepackage{bbm}
\usepackage{color}
\usepackage{enumerate}
\usepackage{hyperref}
\usepackage[normalem]{ulem}
\input xy
\xyoption{all}

\newcommand{\N}{\mathbb{N}}                     
\newcommand{\Z}{\mathbb{Z}}                     
\newcommand{\R}{\mathbb{R}}                     
\newcommand{\C}{\mathbb{C}}                     

\newcommand{\di}{\mathrm{d}}                     
\newcommand{\Graph}{\mathrm{graph}\,}	
\newcommand{\diff}{\mathrm{d}}		

\newtheorem{mainthm}{\sc Theorem}           
\newtheorem{thm}{\sc Theorem}[section]               
\newtheorem*{thm*}{\sc Theorem}               
\newtheorem{cor}[thm]{\sc Corollary}        
\newtheorem*{cor*}{\sc Corollary}        
\newtheorem{lem}[thm]{\sc Lemma}            
\newtheorem{prop}[thm]{\sc Proposition}     
\newtheorem{add}[thm]{\sc Addendum}         
\newtheorem{rem}[thm]{\sc Remark}           

\title{Gibbs measures on Cayley trees and homoclinic orbits}

\author{Alberto Abbondandolo and Pietro Majer}

\date{}

\begin{document}

\maketitle

\begin{abstract}
We study the question of existence and multiplicity of homogeneous Markov Gibbs measures on Cayley trees with spin space $\Z$. We prove that, under a natural summability assumption on the interaction function, there are always at least two homogeneous Markov Gibbs measures satisfying two natural symmetry conditions. Moreover, any two homogeneous Markov Gibbs measures satisfying a suitable non-degeneracy assumption can be superimposed and produce infinitely many other homogeneous Markov Gibbs measures. When the interaction function $Q$ has bounded support in $\Z$, there is a one-to-one correspondence between homogeneous Markov Gibbs measures and suitable homoclinic orbits of a symplectomorphism of $\R^{2n}$, where $n$ is the smallest natural number such that $Q(i)=0$ for every $i>n$. For $n=1$, results from symbolic dynamics and horseshoes imply the existence of homogeneous Markov Gibbs measures whose single-site marginal is concentrated near an arbitrary finite subset of $N\Z$, where $N$ is a sufficiently large integer. In the case of interaction functions with infinite support, there is a one-to-one correspondence between homogeneous Markov Gibbs measures and suitable homoclinic orbits of a Lagrangian correspondence on the cotangent bundle of a Banach space.
\end{abstract}

\tableofcontents

\section*{Introduction}
\addcontentsline{toc}{section}{Introduction}

\renewcommand\themainthm{\Alph{mainthm}}

\paragraph{Homogeneous Markov Gibbs measures on Cayley trees.} This paper addresses the questions of existence and multiplicity of homogeneous Markov Gibbs measures on Cayley trees, in the case of an infinite spin space. We recall that the Cayley tree of order $d\geq 2$ is the unique unrooted tree with countably many vertices, each of which is the origin of $d+1$ edges. We denote by $V$ and $E$ the sets of its vertices and edges. We choose as spin space the ring of integers $\Z$, thus considering $\Omega = \Z^V$ as  configuration space, endowed with its usual $\sigma$-algebra $\mathscr{A}$.

A probability measure on $(\Omega,\mathscr{A})$ is said to be homogeneous if it is invariant under all tree automorphisms, and a homogeneous
probability measure $\mu$ is said to be Markov if it is determined by its single-site marginal $\pi: \Z \rightarrow [0,1]$ and transition matrix $P: \Z \times \Z \rightarrow [0,1]$ by propagation along the tree, where
\[
\pi(i) := \mu( \sigma_x = i), \qquad P(i,j) := \mu ( \sigma_y=j \mid \sigma_x = i), \qquad  \forall i,j\in \Z.
\]
Here, $\sigma_x$ denotes the value of the configuration $\sigma\in \Omega$ at $x\in V$, $\{x,y\}$ is an edge, and the functions $\pi$ and $P$ do not depend on  $x$ and $\{x,y\}$ because of homogeneity.

In order to define Gibbs measures, we need an interaction function $Q: \Z \rightarrow \R$, which we assume to be nonnegative, even, not identically zero, and $\frac{d+1}{2}$-summable. Throughout the introduction, we shall also assume $Q$ to be irreducible, meaning that the abelian subgroup of $\Z$ generated by its support is the whole $\Z$. Typical and well-studied examples are the interaction functions
\begin{equation}
\label{e:examples}
Q(i) = e^{-\frac{|i|}{T}}, \quad T>0,  \qquad Q(i) = \Bigl( \frac{1}{1+|i|} \Bigr)^{\frac{1}{T}}, \quad 0 < T < \textstyle{\frac{d+1}{2}},  \qquad  \forall i\in \Z, 
\end{equation}
where in physical applications the positive parameter  $T$ is interpreted as temperature. The Gibbs property for a homogeneous Markov probability measure $\mu$ on $(\Omega,\mathscr{A})$ can be expressed in the following way: given a vertex $x\in V$ with neighbours $y_1,\dots,y_{d+1}$, we have
\[
\mu \bigl(  \sigma_x =i \mid \sigma_{y_1}=j_1, \dots, \sigma_{y_{d+1}} = j_{d+1}  \bigr) = \frac{ \prod_{h=1}^{d+1} Q(i - j_h)}{\sum_{k\in \Z} \prod_{h=1}^{d+1} Q(k-j_h) },
\]
for every $i\in \Z$ and $(j_1,\dots,j_{d+1})\in \Z^{d+1}$ for which the conditioning event has positive probability. See Section \ref{s:definitions} for more precise definitions and further discussion. See also the classical book of Georgii \cite{geo88} for the general theory of Gibbs measures on graphs and the book of Rozikov \cite{roz13}, focusing on models on Cayley trees. 

\paragraph{The convolution equation.} We denote by $\mathscr{G}(d,Q)$ the set of homogeneous Markov probability measures on the configuration space $\Omega$ associated to the Cayley tree of order $d$ which are Gibbs with respect to the interaction function $Q$. Here it is worth noticing that the Gibbs property depends only on the projective class of $Q$, meaning that if $c$ is any positive number then $\mathscr{G}(d,cQ) = \mathscr{G}(d,Q)$. Note also that $\Z$ acts by translation on the spin space and hence on $\mathscr{G}(d,Q)$. Our interest is in studying the number of elements of the quotient $\mathscr{G}(d,Q)/\Z$ by this $\Z$-action and their qualitative properties.

Due to the irreducibility of $Q$, there is a one-to-one correspondence between the Gibbs measures in $\mathscr{G}(d,Q)$ and the positive solutions $u\in \ell^p(\Z)$, $p:= \frac{d+1}{d}$, of the equation
\begin{equation}
\label{e:conv-eq}
u = (Q* u)^d,
\end{equation}
where the symbol $*$ denotes the convolution product of two functions on $\Z$ and $\ell^p(\Z)$ the space of $p$-summable bi-infinite real sequences. The single-site marginal and transition matrix of the Gibbs measure associated to the positive solution $u\in \ell^p(\Z)$ of \eqref{e:conv-eq} are given by the identities
\[
\pi(i) = \frac{u(i)^p}{\|u\|_p^p}, \qquad P(i,j) = \frac{Q(i-j)u(j)}{(Q *u)(i)},
\]
where $\|u\|_p$ denotes the $p$-norm of $u$. This fundamental fact, which was proven in a much more general form by Zachary in \cite{zac83}, is reproven for the reader's convenience in Section \ref{s:definitions}.  

We call the positive number
\begin{equation}
\label{e:weight-intro}
w:= \|Q\|_{\frac{d+1}{2}}^{\frac{d+1}{d-1}} \|u\|_p^p
\end{equation}
weight of the Gibbs measure $\mu$ associated to the solution $u$ to \eqref{e:conv-eq}. Here, the exponent in the norm of $Q$ is chosen in such a way that $w$ does not change if we replace $Q$ by $cQ$ for some $c>0$. Equation \eqref{e:conv-eq} and Young's inequality imply that $w\geq 1$. Weights play an important role in Theorem \ref{t:superpos-main} below. 

As far as we know, up to now equation \eqref{e:conv-eq} has been studied mainly in a perturbative regime, see \cite{hk21} and \cite{ahkm24}. If $Q=\mathbbm{1}_{\{0\}}$ is the characteristic function of the singleton $\{0\}$ then \eqref{e:conv-eq} reduces to the equation $u=u^d$, whose nonnegative solutions in $\ell^p(\Z)$ are all the characteristic functions of finite subsets of $\Z$. In the zero-temperature limit $T\rightarrow 0$, interaction functions $Q$ such as those defined in \eqref{e:examples} converge to $\mathbbm{1}_{\{0\}}$ and one expects to find positive solutions of \eqref{e:conv-eq} in $\ell^p(\Z)$ which are close to characteristic functions of finite subsets of $\Z$. Indeed, in \cite[Theorem 3.1]{ahkm24} it is proven that for every natural number $N$ there exists a positive threshold $\eta=\eta(d,N)$ such that if $Q$ is normalized by $Q(0)=1$ and $\|Q- \mathbbm{1}_{\{0\}} \|_{\frac{d+1}{2}} < \eta$ then $\mathscr{G}(d,Q)$ contains Gibbs measures whose single-site marginals are concentrated near arbitrary subsets of $\Z$ of cardinality at most $N$.

In this paper, we wish to derive existence and multiplicity results outside of the perturbative regime. In all our results, the starting point is equation \eqref{e:conv-eq}. Therefore, this paper can also be considered as an investigation of several techniques for studying this fixed point equation. From this point of view, the assumption that $d$ is an integer is not necessary and all our results concerning solutions of \eqref{e:conv-eq} would hold with minor changes for any real number $d>1$.

\paragraph{Existence of two symmetric Gibbs measures.} Consider the involutions $\iota, \jmath : \Z \rightarrow \Z$ defined by $\iota(i) = -i$ and $\jmath(i) = 1-i$, which lift to involutions on $\Omega=\Z^V$, and by the push-forward operation on measures to involutions $\iota_*$ and $\jmath_*$ on the space of probability measures on $(\Omega,\mathscr{A})$. 
The fact that $Q$ is even implies that the set of Gibbs measures $\mathscr{G}(d,Q)$ is invariant under both $\iota_*$ and $\jmath_*$. It then makes  sense to look for Gibbs measures in $\mathscr{G}(d,Q)$ which are fixed by either $\iota_*$ or $\jmath_*$. Equivalently, we are looking for Gibbs measures $\mu$ in $\mathscr{G}(d,Q)$ whose single-site marginal $\pi$ and transition matrix $P$ satisfy either
\[
\pi(-i) = \pi(i), \qquad P(-i,-j) = P(i,j) \qquad \forall (i,j)\in \Z\times \Z,
\]
or
 \[
\pi(1-i) = \pi(i), \qquad P(1-i,1-j) = P(i,j) \qquad \forall (i,j)\in \Z\times \Z.
\]
The first new result of this paper is the following theorem, stating that Gibbs measures which are symmetric in the above two senses always exist.

\begin{mainthm}
\label{t:ex2-main}
Assume that the interaction function $Q$ is irreducible and $\frac{d+1}{2}$-summable. Then there exist Gibbs measures $\mu_1$ and $\mu_2$ in $\mathscr{G}(d,Q)$ such that  $\iota_*(\mu_1) = \mu_1$ and $\jmath_*(\mu_2) = \mu_2$.
In particular, $\mu_1$ and $\mu_2$ define two distinct elements of $\mathscr{G}(d,Q)/\Z$.
\end{mainthm}

The proof of this theorem is variational: we find $p$-summable positive solutions $u$ of \eqref{e:conv-eq} by maximizing the functional
\[
\mathscr{Q}(u) := \langle Q* u, u \rangle = \sum_{(i,j)\in \Z^2} Q(j-i) u(i) u(j)
\]
on the unit sphere in $\ell^p(\Z)$. The symmetry conditions are achieved by constraining this maximization to bi-infinite sequences $u$ which satisfy either $u(-i)=u(i)$ or $u(1-i)=u(i)$, for every $i\in \Z$. The difficulty here is the lack of compactness, which we overcome by a discrete version of Lions's concentration compactness principle from \cite{lio84}. 

\paragraph{Superposition of Gibbs measures.} Our second result concerns the possibility of finding new Gibbs measures in $\mathscr{G}(d,Q)$ as superpositions of two non-degenerate Gibbs measures $\mu,\mu'\in \mathscr{G}(d,Q)$. Here, a Gibbs measure $\mu\in \mathscr{G}(d,Q)$ is said to be non-degenerate if the corresponding positive solution $u\in \ell^p(\Z)$ of \eqref{e:conv-eq} is non-degenerate, meaning that the linearized equation
\[
v = d (Q*u)^{d-1} Q *v
\]
has no non-trivial solutions in $\ell^p(\Z)$. The homogeneous probability measures which are obtained as convex combinations of two measures $\mu$ and $\mu'$ in $\mathscr{G}(d,Q)$ are Gibbs for $Q$, but in general they are not Markov, and hence not in $\mathscr{G}(d,Q)$. However, if we assume that $\mu$ and $\mu'$ are non-degenerate and $n\in \Z$ has a large absolute value, then we can guarantee the existence of an element of $\mathscr{G}(d,Q)$ all of whose finite-dimensional marginals are close in total variation to the corresponding marginals of a suitable convex combination of $\mu$ and  $\tau_*^n \mu'$, where $\tau_*$ denotes push-forward with respect to the translation operator $\sigma\mapsto \sigma+1$ on $\Omega$. The precise convex combination is determined by the weights of the Gibbs measures $\mu$ and $\mu'$, as defined in \eqref{e:weight-intro}. More precisely, 
denoting by $\mu|_{\Lambda}$ the marginal of $\mu$ on $\Z^{\Lambda}$, that is, the measure
\[
\mu|_{\Lambda}(\eta) := \mu(\sigma|_{\Lambda}= \eta) \qquad \forall \eta\in \Z^{\Lambda}, 
\]
and by $\|\cdot\|_1$ the total variation norm
\[
\|\nu\|_1 := \sum_{\eta\in \Z^{\Lambda}} |\nu(\eta)|
\]
of a signed measure $\nu$ on $\Z^{\Lambda}$,  we have the following theorem, which is proven in Section \ref{s:superposition}. 

\begin{mainthm}
\label{t:superpos-main}
Assume that the interaction function $Q$ is irreducible and $\frac{d+1}{2}$-summable.
Let $\mu$ and $\mu'$ be not necessarily distinct non-degenerate elements of $\mathscr{G}(d,Q)$ with weights $w$ and $w'$.  Then there exists $N\in \N$  such that for every $n\in \Z$ with $|n|\geq N$ there is a non-degenerate Gibbs measure $\mu \#_n \mu'\in \mathscr{G}(d,Q)$ with weight $w_n$ such that 
\[
\lim_{n\rightarrow \pm\infty} w_n = w + w' \qquad \mbox{and} \qquad
\lim_{n\rightarrow \pm \infty} \bigl\| (\mu \#_n \mu')|_{\Lambda} - \bigl( \textstyle{\frac{w}{w+w'}} \, \mu|_{\Lambda} + \textstyle{\frac{w'}{w+w'}}\, (\tau_*^n \mu')|_{\Lambda} \bigr) \bigr\|_1 = 0,
\]
for every finite subset $\Lambda\subset V$.
\end{mainthm}

Note that the superposition $\mu \#_n \mu'$ favours the measure among $\mu$ and $\mu'$ having larger weight. Denote by $\pi$, $\pi'$, and $\pi_n$ the single-site marginals and by $P$, $P'$, and $P_n$ the transition matrices of the measures $\mu$, $\mu'$, and $\mu \#_n \mu'$.
By taking $\Lambda$ in the above theorem to be either a singleton or the pair of vertices of an edge, we obtain the following asymptotics for the single-site marginal and the one-step joint law of $\mu \#_n \mu'$:
\[
\begin{split}
\lim_{n\rightarrow \pm \infty} &\sum_{i\in \Z} \left| \pi_n(i) - \bigl( {\textstyle \frac{w}{w+w'}} \,\pi (i) + {\textstyle \frac{w'}{w+w'}} \, \pi'(i-n) \bigr) \right| = 0, \\
\lim_{n\rightarrow \pm \infty} &\sum_{(i,j)\in \Z^2} \left| \pi_n(i) P_n(i,j) - \bigl( {\textstyle \frac{w}{w+w'}} \,\pi (i)P(i,j) + {\textstyle \frac{w'}{w+w'}} \, \pi'(i-n)P'(i-n,j-n) \bigr) \right| = 0. 
\end{split}
\]
By applying this theorem iteratively, we deduce that the existence of one non-degenerate Gibbs measure in $\mathscr{G}(d,Q)$ implies the existence of infinitely many non-degenerate Gibbs measures which are pairwise distinct also after projection to the quotient $\mathscr{G}(d,Q)/\Z$ and have diverging weights. It is conceivable that Theorem \ref{t:superpos-main} can be proven variationally. Here, we find it more convenient to prove it by Newton's method.

One may conjecture that the existence of at least one non-degenerate Gibbs measure in $\mathscr{G}(d,Q)$ is a generic property of $Q$, meaning that the set of $Q$'s for which it holds is a residual set in the space of all even nonnegative $\frac{d+1}{2}$-summable functions,
but we do not have a proof of this. Consequently, we do not know whether $\mathscr{G}(d,Q)/\Z$ is an infinite set for $Q$ generic.

Theorems \ref{t:ex2-main} and \ref{t:superpos-main} have natural generalizations to the case in which we replace the spin space
$\Z$ by $\Z^k$. The proofs of the next theorems, instead, strongly rely on the fact that the spin space is 1-dimensional. 

\paragraph{Gibbs measures as homoclinic orbits.} In our next result, we focus on the special class of interaction functions $Q$ which are supported in $\{-1,0,1\}$. By irreducibility, $Q(1)=Q(-1)$ is positive, and hence we can use the normalization $Q(1)=Q(-1)=1$ and get an interaction function of the form
\begin{equation}
\label{e:smallQ}
Q = q \mathbbm{1}_{\{0\}} +  \mathbbm{1}_{\{-1,1\}},
\end{equation}
where $q$ is an arbitrary nonnegative number. The Gibbs property now gives positive probability only to spin differences with values $-1$, $0$, or $1$ on every edge. The number $q$ should be interpreted as a laziness parameter: when $q\rightarrow +\infty$ the system is lazy and favours constant spin values, whereas for $q=0$ the system is maximally active and forces the spin to jump. The cases $q=0$ and $q=1$ model uniformly random graph homomorphisms and uniformly random 1-Lipschitz functions, respectively. See \cite{lt24}, \cite{bkrs26}, and \cite{hs26} for results about localization of Gibbs measures in these models.

When $Q$ has this form, equation \eqref{e:conv-eq} can be rewritten as the second order scalar recursion
\[
u(i+2) = u(i+1)^{\frac{1}{d}} - q\, u(i+1) - u(i) \qquad \forall i \in \Z,
\]
and is hence equivalent to a first order recursion on $\R^2$. From this observation, we deduce that the elements in $\mathscr{G}(d,Q)$ are in one-to-one correspondence with the orbits $(z_i=(x_i,y_i))_{i\in \Z}$ of the discrete time dynamical system
\[
\varphi: \R^2 \rightarrow \R^2, \qquad \varphi(x,y) = \bigl(y, |y|^{\frac{1}{d}-1}y - q y -x\bigr),
\]
which are contained in the first quadrant $(0,+\infty)\times (0,+\infty)$ and are homoclinic to the fixed point at the origin, meaning that $z_i= \varphi^i(z_0) \rightarrow 0$ for $i\rightarrow \pm \infty$. Therefore, the quest for Gibbs measures translates into the quest for homoclinic orbits of $\varphi$.

Notice that $\varphi$ is a homeomorphism of $\R^2$ and restricts to an analytic diffeomorphism from  the complement of the line $\{y=0\}$ to the complement of the line $\{x=0\}$, where it preserves the symplectic form $\mathrm{d} x \wedge \mathrm{d} y$. Moreover, it is reversible with respect to the involution $S(x,y) = (y,x)$, meaning that $\varphi^{-1} = S \circ \varphi \circ S$. 

Homoclinic orbits are given by the intersections of the stable and unstable sets of $0$, namely the sets
\[
\begin{split}
W^s(0) &:= \{z\in \R^2 \mid \varphi^n(z) \rightarrow 0 \mbox{ for } n\rightarrow \infty\}, \\ 
W^u(0) &:= \{z\in \R^2 \mid \varphi^{n}(z) \rightarrow 0 \mbox{ for } n\rightarrow -\infty\}.
\end{split}
\]
Even if $\varphi$ is not differentiable at the origin, this fixed point behaves as a hyperbolic fixed point of saddle type, and we will show that $W^s(0)$ and $W^u(0)$ are 1-dimensional objects. Using the reversibility of $\varphi$ and its analyticity away from the line $\{y=0\}$, in Sections \ref{s:ifss}, \ref{s:psmt}, \ref{s:duesol}, and \ref{s:homoconn} we will gather a good amount of information on $W^s(0)$ and $W^u(0)$, some of which is summarized in the following theorem.

\begin{mainthm}
\label{t:omocline-main}
For every $q\geq 0$ there is $r\in (0,1)$ and an odd $C^1$ function $f: (-r,r) \rightarrow \R$ which is analytic on $[0,r)$, positive on $(0,r)$, and such that
\[
W^s(0) = \bigcup_{n\geq 0} \varphi^{-n} ( \mathrm{graph}\, f ), \qquad W^u(0) = S (W^s(0)).
\]
Moreover, the intersection $W^s(0) \cap W^u(0)$ contains at least one point $z$ belonging to the fixed point line $\{y=x\}$ of $S$ and at least one point $z'$ such that $\varphi(z')=S(z')$, with the further property that the homoclinic orbits $(\varphi^k(z))_{k\in \Z}$ and $(\varphi^{k}(z'))_{k\in \Z}$ are fully contained in the first quadrant $(0,+\infty)\times (0,+\infty)$. If moreover $W^s(0)\neq W^u(0)$, then $W^s(0)$ and $W^u(0)$ have topologically transverse intersections at $z$ and $z'$. Finally, there are at most finitely many values of $q$, all contained in the interval $[0,\frac{d+2}{d-1})$, for which $W^s(0)=W^u(0)$.
\end{mainthm}

\begin{figure}[h]
\centerline{\scalebox{.4}{\includegraphics{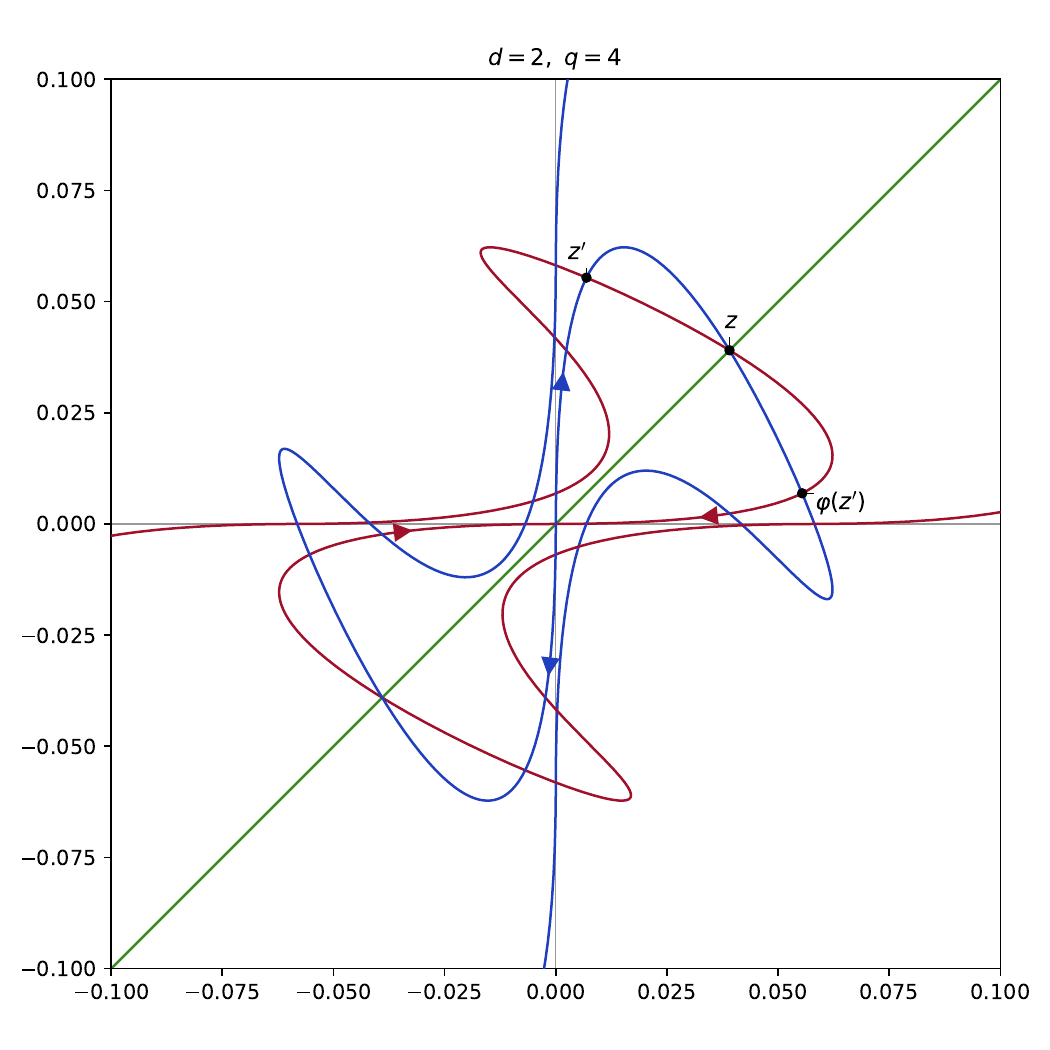} \includegraphics{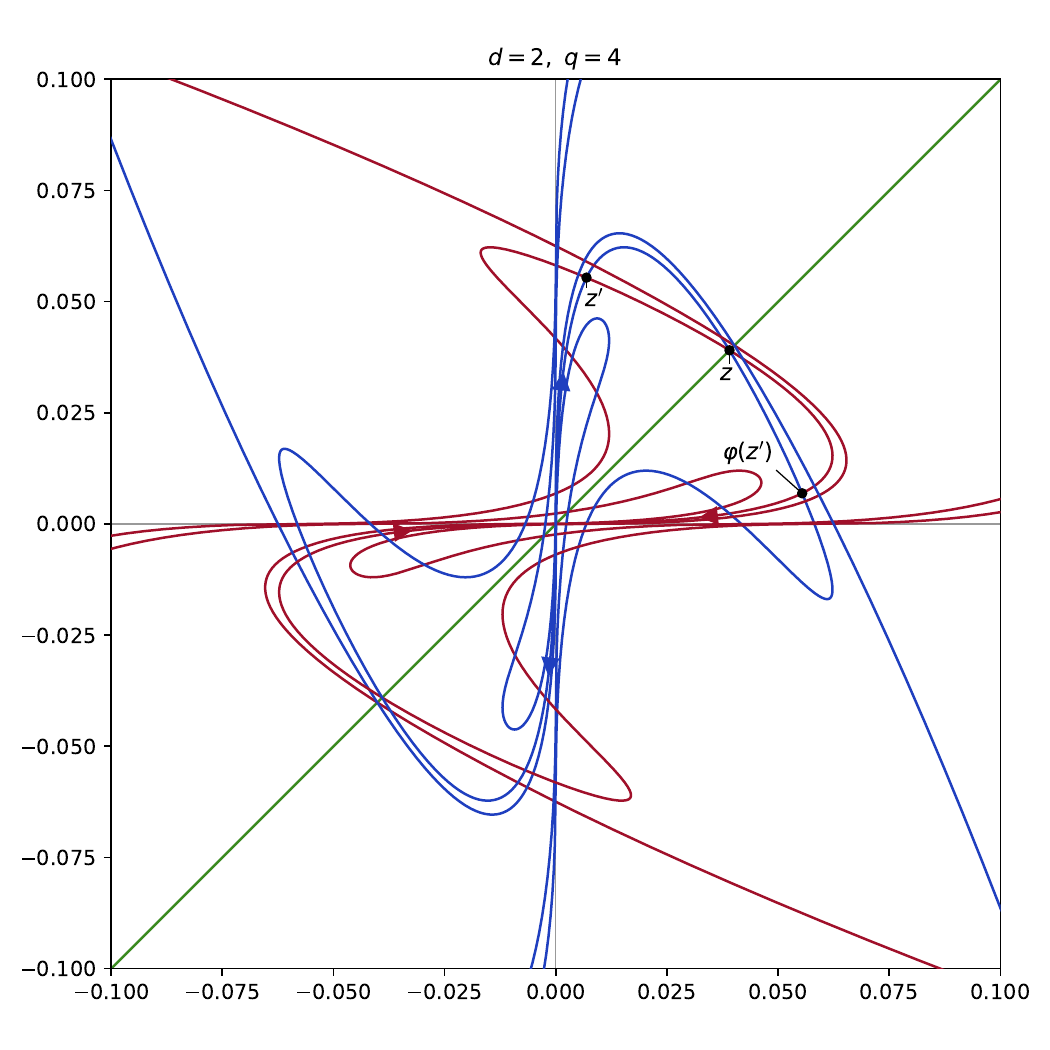}}}
\caption{\small{Subarcs of $W^s(0)$ and $W^u(0)$ and their intersection points $z$ and $z'$.}}
\label{f:figintro}
\end{figure} 

The two pictures in Figure \ref{f:figintro} represent subarcs of $W^s(0)$ and $W^u(0)$ and their intersections at $z$ and $z'$. The stable subarc is $\varphi^{-2}(\mathrm{graph}\, f)$ on the left, and $\varphi^{-3}(\mathrm{graph}\, f)$ on the right.

See also \cite{jast14}, where a planar homeomorphism with similar properties appears in the study of the discrete $p$-Schr\"odinger equation, and two symmetric homoclinic orbits are found in a similar way.

The two Gibbs measures $\mu$ and $\mu'$ which are associated to the homoclinic points $z$ and $z'$ of the above theorem satisfy $\jmath_*(\mu)=\mu$ and $\iota_*(\mu') = \mu'$, that is, the
 two symmetry properties which we saw in general in Theorem \ref{t:ex2-main}. Moreover, their single-site marginals $\pi,\pi': \Z \rightarrow [0,1]$ are unimodal functions and decay to zero with order $d$. See Corollary \ref{c:info+} for the precise statement. 

The equality $W^s(0)=W^u(0)$ would give us a homoclinic connection contained in the first quadrant, and hence a continuum of Gibbs measures.
Numerical simulations suggest that $W^s(0)=W^u(0)$ never happens (see the brief discussion at the end of Section \ref{s:homoconn}), but we do not have a proof of this.

\paragraph{Coding of Gibbs measures.} The existence of topologically transverse intersections of the stable and unstable manifold of $\varphi$ is known to imply the presence of a horseshoe. More precisely, it implies the existence of a positive integer $N$ and of compact $\varphi^N$-invariant subset $\Lambda\subset \R^2$ together with a continuous surjective map
\[
H: \Lambda \rightarrow \{0,1\}^{\Z} 
\]
such that $H\circ \varphi^N = \sigma \circ H$, where $\sigma: \{0,1\}^{\Z} \rightarrow \{0,1\}^{\Z}$ denotes the shift $(\sigma \omega)(i)=\omega(i+1)$. The standard proofs of this fact require $C^1$ regularity. In Appendix \ref{a:th}, building on an argument which is due to Burns and Weiss \cite{bw95}, we prove a general result about existence of a horseshoe for homeomorphisms of surfaces, which can be applied to the map $\varphi$. See Theorem \ref{t:horseshoe}. When applied to the topologically transverse intersection point $z' \in W^s(0)\cap W^u(0)$ from Theorem \ref{t:omocline-main}, this theorem implies that every point $w\in \Lambda$ which is mapped by $H$ to a bi-infinite sequence $\omega\in \{0,1\}^{\Z}$ such that $\omega(i)=0$ for $|i|$ large enough is a homoclinic point for $\varphi$ with orbit $\{\varphi^k(w)\}_{k\in \Z}$ fully contained in the first quadrant. 
This leads to the following coding result for Gibbs measures, which is proven in Section \ref{s:coding}.

\begin{mainthm}
\label{t:coding-main}
Let $d\geq 2$ be an integer and assume that $q\in [0,+\infty)$ does not belong to the (possibly empty) finite set of values for which $W^u(0)=W^s(0)$. Set $Q:= q\mathbbm{1}_{\{0\}} + \mathbbm{1}_{\{-1,1\}}$. Then there exists a positive integer $N$ and numbers $\theta>1$ and $\eta>0$ such that for every non-empty finite set $A\subset N\Z$ there exists a Gibbs measure $\mu_A$ in $\mathscr{G}(d,Q)$ whose single-site marginal $\pi_A$ satisfies
\begin{equation}
\label{e:favorisce-main}
\min_{i\in A} \pi_A(i) \geq \theta \, \max_{i\in \Z\setminus A} \pi_A(i),
\end{equation}
and whose weight $w_A$  satisfies $w_A \geq \eta |A|$, where $|A|$ denotes the cardinality of $A$.  If $A'$ is not a translated copy of $A$, then the Gibbs measures $\mu_A$ and $\mu_{A'}$ represent distinct elements in $\mathscr{G}(d,Q)/\Z$, which is then an infinite set.
\end{mainthm}

Condition \eqref{e:favorisce-main} tells us that the Gibbs measure $\mu_A$ favours the states in $A$ over the states which are not in $A$. Note that in this setting of small support, without any perturbative assumption we obtain Gibbs measures with single-site marginals concentrated near arbitrary finite subsets of $N\Z$. This should be compared to the perturbative results from \cite{ahkm24}, where we have single-site marginals concentrated near arbitrary subsets of $\Z$ with an upper bound on the cardinality. 

\medskip

Theorems \ref{t:superpos-main} and \ref{t:coding-main} suggest that $\mathscr{G}(d,Q)/\Z$ is an infinite set for most if not all $\frac{d+1}{2}$-summable interaction functions, without any need to go to a zero-temperature limit. This is due to the fact that we are considering an infinite spin space. One may then wonder whether some version of the paradigm ``a lower temperature forces the existence of more Gibbs measures'' survives in this setting, in which an infinite set of Gibbs measures seems to be the rule. We do not have a precise answer to this question, but the mechanism behind Theorem \ref{t:coding-main} suggests that in the case of the interaction function \eqref{e:smallQ}
there are qualitative differences in the Gibbs measures we find as we vary the parameter $q$. The number $N$ from this theorem, telling us how much a finite subset of $\Z$ should be sparse in order to find Gibbs measures with single-site marginal concentrated near it, depends on the angle between the stable and unstable manifolds at the intersection point $z'$ producing the horseshoe, and $N$ gets larger if this angle is small. When $q$ is large, $W^s(0)$ and $W^u(0)$ intersect at $z'$ transversally with an angle which is bounded away from zero. Instead, for values of $q$ in the range $[0,1]$ this angle becomes small, and $W^s(0)$ and $W^u(0)$ may even become tangent at $z'$. Therefore, $N$ is large when $q$ is small, and small when $q$ is large. If we have access only to the restriction of the single-site marginals of Gibbs measures to some bounded interval in $\Z$, we would then see a larger variety of behaviours for large values of $q$, corresponding to the zero-temperature limit. It would be interesting to turn these vague considerations into some precise statement.

\paragraph{Interaction functions with larger support.} We conclude this introduction by describing the contents of the last two Sections \ref{s:ifbs} and \ref{s:holc}. In the first one, we show that the one-to-one correspondence between Gibbs measures in $\mathscr{G}(d,Q)$ and homoclinic orbits of a (non-smooth) symplectomorphism holds whenever $Q$ has bounded support. More precisely: if $Q$ is supported in $\{-n,\dots,n\}$ and $Q(n)=Q(-n)\neq 0$, then there is a one-to-one correspondence between Gibbs measures in $\mathscr{G}(d,Q)$ and the homoclinic points of the (non-smooth) symplectomorphism
\[
\varphi : \R^{2n} \rightarrow \R^{2n}
\]
which is associated to the discrete Lagrangian system which is induced by the Lagrangian
\[
L : \R^{n+1} \rightarrow \R, \qquad L(x_0,x_1,\dots,x_n) := \frac{1}{2} Q(0) x_0^2 + \sum_{i=1}^n Q(i) x_0 x_i - \frac{d}{d+1} |x_0|^{\frac{d+1}{d}}.
\]
As in the case $n=1$, this symplectomorphism is reversible with respect to an anti-symplectic involution $S: \R^{2n} \rightarrow \R^{2n}$.
See Section \ref{s:ifbs} for more about the formalism of discrete Lagrangian systems and their associated symplectomorphisms.

In Section \ref{s:holc}, we show how this bijection needs to be modified when $Q$ has unbounded support. In this case, the Lagrangian
\[
L(x):= \frac{1}{2} Q(0) x_0^2 + \sum_{i\geq 1} Q(i) x_0 x_i - \frac{d}{d+1} |x_0|^{\frac{d+1}{d}}
\]
on the Banach space $E:=\ell^p(\Z_{\geq 0})$, where as before $p:= \frac{d+1}{d}$, and the left shift operator $\sigma:E \rightarrow E$, $(\sigma x)(k) = x(k+1)$, generate the Lagrangian correspondence 
\[
\mathcal{R} :=\bigl\{ \bigl( (x,y), (x',y') \bigr) \in T^* E \times T^* E \mid x' = \sigma x, \; y = \sigma^* y' - \di L(x) \bigr\}
\]
on the cotangent bundle $T^* E = E \times E^*$ equipped with its canonical symplectic structure. The orbits of this Lagrangian correspondence, that is, the $T^* E$-valued bi-infinite sequences $(x_k,y_k)_{k\in \Z}$ such that $((x_k,y_k), (x_{k+1},y_{k+1}))\in \mathcal{R}$, with the further property of being homoclinic to the origin in a suitable weak sense and such that $x_k(i)>0$ for every $k\in \Z$ and $i\geq 0$, are in one-to-one correspondence with the Gibbs measures in $\mathscr{G}(d,Q)$.

In Sections \ref{s:ifbs} and \ref{s:holc}, we content ourselves with establishing these one-to-one correspondences, but we believe that they should be the starting point for extending the detailed analysis we developed for $Q$ supported in $\{-1,0,1\}$ to  interaction functions with arbitrary bounded support, and possibly even to arbitrary $\frac{d+1}{2}$-summable interaction functions.

\paragraph{Acknowledgments.} We wish to thank Christof K\"ulske and Florian Henning for introducing us to Gibbs measures on trees and for stressing the relevance of equation \eqref{e:main-eq} in this context. 

This paper was completed while the first author was in residence at the Institute for Advanced Study during a Special Year on Conformally Symplectic Dynamics and Geometry, supported by an anonymous donor.

We were assisted by Claude Opus 5.5 and OpenAI GPT-5.4 in drafting the proofs of some known results which we wished to include for the reader's convenience (Propositions \ref{p:finite-partition-function}, \ref{p:zachary}, and \ref{p:conc-comp}), in polishing some of the combinatorial arguments  of Section \ref{s:superposition}, in making the pictures, and in the numerical simulations mentioned at the end of Section \ref{s:homoconn}. Claude was extremely helpful for grammatically and mathematically proofreading the manuscript.

\numberwithin{equation}{section}

\section{Homogeneous Markov Gibbs measures on Cayley trees}
\label{s:definitions}

In this section, we present a self-contained introduction to the special class of Gibbs measures on the special family of infinite graphs which are the subject of this paper. Definitions are adapted to the setting we are considering here but are equivalent to the standard ones, for which we refer to the classical book of Georgii \cite{geo88}. 

 Let $d\geq 2$ be an integer and denote by $\Gamma_d=(V,E)$ the \textit{Cayley tree of order} $d$: $\Gamma_d$ is the unrooted tree with countably many vertices, each of which has degree $d+1$. 
 
 As \textit{spin space} we choose the ring of integers $\Z$, thus obtaining the \textit{configuration space} $\Omega:=\Z^V$. For $\sigma\in \Omega$ and $x\in V$, we write $\sigma_x$ for the spin at $x$ of the configuration $\sigma$. We endow $\Omega$ with the $\sigma$-algebra $\mathscr{A}$ generated by cylindrical subsets of $\Omega$, namely sets of the form $\{\sigma\in \Omega \mid \sigma|_{\Lambda} = \xi\}$, where $\Lambda$ is a finite subset of $V$ and $\xi\in \Z^{\Lambda}$.
 
 Tree automorphisms of $\Gamma_d$ act on the set of probability measures on $(\Omega,\mathscr{A})$, and a probability measure on $(\Omega,\mathscr{A})$ is said to be \textit{homogeneous} if it is invariant under all tree automorphisms. 
 
Let $\mu$ be a homogeneous probability measure on $(\Omega,\mathscr{A})$. Since the automorphism group of $\Gamma_d$ acts transitively on vertices, the quantity
\[
\pi(i) := \mu ( \sigma_x = i ), \qquad\forall  i\in \Z,
\]
does not depend on the choice of the vertex $x\in V$. It defines a probability measure $\pi$ on $\Z$ which is called \textit{single-site marginal} of $\mu$. Since automorphisms of $\Gamma_d$ act transitively on oriented edges, the quantity
\[
P(i,j) := \left\{ \begin{array}{cl} \mu (\sigma_y=j \mid \sigma_x=i) & \mbox{if } \pi(i)>0, \\ \delta_{ij} & \mbox{if } \pi(i)=0, \end{array} \right. \qquad \forall (i,j)\in \Z\times \Z,
\]
does not depend on the edge $\{x,y\}\in E$ either. It is a stochastic matrix, meaning that the sum of $P(i,j)$ over all $j\in \Z$ is 1, and is called \textit{transition matrix} of the homogeneous measure $\mu$. By the definition of conditional probability, we have, for any choice of neighbouring vertices $x$ and $y$,
\begin{equation}
\label{e:step1}
\mu(\sigma_x = i, \; \sigma_y=j) = \pi(i) P(i,j), \qquad \forall (i,j)\in \Z\times \Z.
\end{equation}
By considering an automorphism of $\Gamma_d$ which interchanges $x$ and $y$, we obtain the \textit{reversibility identity}
\begin{equation}
\label{e:reversibility}
\pi(i) P(i,j) = \pi(j) P(j,i), \qquad \forall (i,j)\in \Z \times \Z.
\end{equation} 

A homogeneous probability measure $\mu$ on $(\Omega,\mathscr{A})$ is called  \textit{Markov} if it is determined by its single-site marginal and transition matrix. More precisely, fix a vertex $o\in V$. Every finite connected set $\Lambda\subset V$ containing $o$ is the set of vertices of a finite rooted tree $(\Lambda,E_{\Lambda})$ with root $o$, where $E_{\Lambda}$ is the set of edges in $E$ connecting two vertices in $\Lambda$ and oriented away from $o$. The \textit{Markov property} for the 
homogeneous probability measure $\mu$ on $(\Omega,\mathscr{A})$ is the condition
 \begin{equation}
 \label{e:markov}
 \mu\bigl( \sigma|_{\Lambda} = \eta \bigr) = \pi(\eta_o) \prod_{(x,y)\in E_{\Lambda}} P(\eta_x,\eta_y),
 \end{equation}
 for every function $\eta: \Lambda \rightarrow \Z$. This condition generalizes \eqref{e:step1}. Kolmogorov extension theorem implies that every probability measure $\pi$ on $\Z$ and stochastic matrix $P$ indexed by $\Z^2$ satisfying \eqref{e:reversibility} determine a unique homogeneous probability measure $\mu$ on $(\Omega,\mathscr{A})$ having single-site marginal $\pi$, transition matrix $P$, and satisfying the Markov property \eqref{e:markov}. 

By \textit{interaction function} we here mean a function $Q:\Z \rightarrow \R$ which is nonnegative, not identically zero, and even. 
The value $Q(k)$ should be thought of as \textit{statistical weight} of the configurations $\sigma\in \Omega$ such that $\sigma_x-\sigma_y=k$, where $\{x,y\}$ is any edge of $\Gamma_d$.  A probability measure $\mu$ on $(\Omega,\mathscr{A})$ is called a \textit{Gibbs measure} for the interaction function $Q$ if for every finite set $\Lambda\subset V$, denoting by $\partial \Lambda\subset  V$ the set of $y$ in $V\setminus \Lambda$ such that $\{x,y\}\in E$ for some $x\in \Lambda$,  we have
\begin{equation}
\label{e:gibbs-gen}
\mu(\sigma|_{\Lambda} =\xi \mid \sigma|_{\partial \Lambda} = \eta) = \frac{w_{\Lambda}(\xi, \eta)}{Z_{\Lambda}(\eta)}, 
\end{equation}
where
\[
w_{\Lambda}(\xi, \eta) := \prod_{\substack{ \{x,y\}\in E \\ \mbox{\footnotesize{with }} x,y\in \Lambda}} Q(\xi_x - \xi_y)  \prod_{\substack{x \in \Lambda, \; y\in \partial \Lambda\\ \{x,y\}\in E}} Q(\xi_x - \eta_y), \qquad Z_{\Lambda}(\eta) := \sum_{\xi \in \Z^{\Lambda}} w_{\Lambda}(\xi,\eta),
\]
for every $\xi\in \Z^{\Lambda}$ and $\eta\in \Z^{\partial \Lambda}$ such that $Z_{\Lambda}(\eta)>0$. In order for the above \textit{partition function} $Z_{\Lambda}$ to be finite, a suitable summability assumption on $Q$ is needed. Setting as usual
\[
\ell^q(\Z) := \Bigl\{ u\in \R^{\Z} \mid \|u\|_q := \Bigl( \sum_{j\in \Z} |u(j)|^q \Bigr)^{\frac{1}{q}} < \infty \Bigr\},
\]
we have the following result, which for the reader's convenience is proven in Appendix \ref{a:gibbs-measures}.

\begin{prop}
\label{p:finite-partition-function}
Assume that the interaction function $Q$ is in $\ell^d(\Z)$. Then for every finite set $\Lambda\subset V$ and every $\eta\in \Z^{\partial \Lambda}$ we have
\[
Z_{\Lambda}(\eta) \leq \|Q\|_d^{e(\Lambda)} < +\infty,
\]
where $e(\Lambda)$ denotes the number of edges of $\Gamma_d$ having at least one end-point in $\Lambda$.
\end{prop}

Therefore, when the interaction function $Q$ is $d$-summable, Gibbs measures are well-defined. Starting from the next section, we shall make the stronger assumption that $Q$ is $\frac{d+1}{2}$-summable. Note also that if $Q' = c\, Q$ for some positive number $c$, then $Q'$ and $Q$ define the same Gibbs measures. 

When a probability measure $\mu$ on $(\Omega,\mathscr{A})$ has the Markov property, the validity of \eqref{e:gibbs-gen} for arbitrary finite sets $\Lambda\subset V$ is equivalent to its validity for all singletons $\Lambda = \{x\}$, $x\in V$, see \cite[Chapter 12]{geo88}. In particular, a homogeneous Markov measure $\mu$ on $(\Omega,\mathscr{A})$ is Gibbs for $Q$ if, and only if, fixing a vertex $x\in V$ with neighbours $y_1,\dots,y_{d+1}$, for every $(j_1,\dots,j_{d+1})\in \Z^{d+1}$ such that
\[
\mu(\sigma_{y_1}=j_1, \dots, \sigma_{y_{d+1}}=j_{d+1} ) >0
\]
there exists at least one $k\in \Z$ for which $\prod_{h=1}^{d+1} Q(k-j_h)>0$ and for every $i\in \Z$  we have
\begin{equation}
\label{e:gibbs}
\mu \bigl(  \sigma_x =i \mid \sigma_{y_1}=j_1, \dots, \sigma_{y_{d+1}} = j_{d+1}  \bigr) = \frac{ \prod_{h=1}^{d+1} Q(i - j_h)}{\sum_{k\in \Z} \prod_{h=1}^{d+1} Q(k-j_h) }.
\end{equation}
Here, the H\"older inequality immediately implies the finiteness of the denominator:
\[
Z_{\{x\}}(\eta) = \sum_{k\in \Z} \prod_{h=1}^{d+1} Q(k-j_h) \leq \prod_{h=1}^{d+1}  \left( \sum_{k\in \Z} Q(k-j_h)^{d+1} \right)^{\frac{1}{d+1}} = \|Q\|_{d+1}^{d+1} \leq \|Q\|_d^{d+1} < +\infty,
\]
where $\eta\in \Z^{\partial \{x\}}$ is defined by $\eta_{y_1}=j_1, \dots, \eta_{y_{d+1}} =j_{d+1}$. 

The set of all homogeneous Markov Gibbs probability measures for $Q$ on the configuration space associated to the Cayley tree $\Gamma_d$ will be denoted by $\mathscr{G}(d,Q)$. 

We consider the action of $\Z$ by translations on $\Z$ and $\Omega=\Z^V$ which is generated by the map
\[
i \mapsto \tau(i) = i + 1, \qquad \sigma \mapsto \tau(\sigma) = \sigma + 1, \qquad \forall i\in \Z, \; \sigma\in \Omega,
\]
and, given measures $\pi$ and $\mu$ on $\Z$ and $\Omega$, the symbols $\tau_* (\pi)$ and $\tau_*(\mu)$ indicate the pushforward measures
\[
\begin{split}
\tau_*(\pi) (a) &= \pi(\tau^{-1}(a)) = \pi(a-1), \qquad \forall a\subset \Z, \\
\tau_*(\mu) (A) &= \mu(\tau^{-1}(A)) = \mu(A - 1), \qquad \forall A\in \mathscr{A}.
\end{split}
\]
Clearly, if $\mu\in \mathscr{G}(d,Q)$ has single-site marginal $\pi$ and transition matrix $P$, then $\tau_*(\mu)$ is also in $\mathscr{G}(d,Q)$ and has single-site marginal and transition matrix $\tau_*(\pi)$ and $P'$, where 
\[
P'(i,j) = (\tau\times \tau)_*(P)(i,j) = P(i-1,j-1).
\]
The action of $\Z$ by translations on $\mathscr{G}(d,Q)$ is free, and we denote by $\mathscr{G}(d,Q)/\Z$ the quotient by this action. We are interested in studying the set $\mathscr{G}(d,Q)/\Z$ of homogeneous Markov Gibbs measures modulo translations.

Another interesting group action on $\mathscr{G}(d,Q)$ is the $\Z_2$-action induced by the involution 
\begin{equation}
\label{e:defniota}
\iota: \Z \rightarrow \Z, \qquad \iota(i):= -i,
\end{equation}
which lifts to an involution $\iota(\sigma):= - \sigma$ on $\Omega$, and hence to an involution $\mu \mapsto \iota_*(\mu)$ on the space of probability measures on $(\Omega,\mathscr{A})$. The involution $\iota_*$ maps the homogeneous Markov measure with single-site marginal $\pi$ and transition matrix $P$ to the  homogeneous Markov measure with single-site marginal $\iota_*(\pi)$ and  transition matrix $(i,j) \mapsto P(-i,-j)$. Since $Q$ is even, $\mathscr{G}(d,Q)$ is invariant under $\iota_*$. 

Similarly, we have the  $\Z_2$-action induced by the involution 
\begin{equation}
\label{e:defnjmath}
\jmath:  \Z \rightarrow \Z, \qquad \jmath(i) := \tau\circ \iota (i) = 1-i,
\end{equation}
which lifts to an involution $\jmath$ on $\Omega$ and hence to an involution $\jmath_*$ on the space of probability measures on $(\Omega,\mathscr{A})$. Again, $\mathscr{G}(d,Q)$ is invariant under $\jmath_*$. 

In the trivial case in which $Q$ is supported in $\{0\}$, the elements of $\mathscr{G}(d,Q)$ are precisely the homogeneous Markov probability measures 
with arbitrary single-site marginal $\pi$ and transition matrix $P(i,j)=\delta_{i,j}$. 

The interaction function $Q$ is said to be \textit{irreducible} if the abelian subgroup of $\Z$ generated by its support is the whole $\Z$ or, equivalently, if the greatest common divisor of the set of positive integers $k$ with $Q(k)>0$ is 1. If $Q$ is not irreducible and nontrivial, then the subgroup generated by the support of $Q$ has the form $m\Z$ for some integer $m>1$, and the study of $\mathscr{G}(d,Q)$ can be reduced to the study of $\mathscr{G}(d,Q')$, where $Q'$ is the irreducible interaction function given by $Q'(i)=Q(mi)$. Indeed, the elements of $\mathscr{G}(d,Q)$ are the convex combinations 
$\sum_{j=0}^{m-1} t_j \mu_j$, where $\mu_j\in \mathscr{G}(d,Q)$ is supported in $(m\Z+j)^V$ and has the form
\[
\mu_j(A) = \mu_j'(A_j)  \;\mbox{ where} \; A_j := \Bigl\{x\mapsto \frac{\sigma(x)-j}{m} \mid \sigma\in A\Bigr\}, \quad A \in \mathscr{A}, \; A\subset   (m\Z+j)^V,
\]
with $\mu_j'\in \mathscr{G}(d,Q')$. Therefore, throughout this paper we assume $Q$ to be irreducible. 

We denote by $*$ the convolution product 
\[
(u*v)(i) := \sum_{j\in \Z} u(j) v(i-j),
\]
of real valued bi-infinite sequences $u$ and $v$ for which the family $\{u(j) v(i-j)\}_{j\in \Z}$ is summable for every $i\in \Z$. Note that by the $d$-summability assumption on $Q$, the H\"older and Young inequalityies ensure that $Q*u$ is well-defined and in $\ell^{d(d+1)}(\Z)$ whenever $u$ is in $\ell^p(\Z)$, where
\[
p:= \frac{d+1}{d} \in (1,2),
\]
is the conjugate exponent to $d+1$. Note also that the irreducibility of $Q$ implies that all nonnegative and not identically vanishing functions $u\in \ell^p(\Z)$ solving the equation
\begin{equation}
\label{e:main-eq}
u = (Q*u)^d
\end{equation}
are positive. We denote by $\mathscr{U}(d,Q)$ the set of positive solutions $u\in \ell^p(\Z)$ of \eqref{e:main-eq} and notice that this set is translation invariant, because the equation \eqref{e:main-eq} is autonomous, and invariant under the involutions $\iota^*$ and $\jmath^*$, where $\iota^*(u) := u\circ \iota$ and $\jmath^*(u) = u\circ \jmath$, because $Q$ is even.
 
\begin{prop}
\label{p:zachary}
Assume that the interaction function $Q\in \ell^d(\Z)$ is irreducible. Then there is a one-to-one correspondence between the sets $\mathscr{G}(d,Q)$ and $\mathscr{U}(d,Q)$. Indeed, every $u\in \mathscr{U}(d,Q)$ defines a unique $\mu\in \mathscr{G}(d,Q)$ whose single-site marginal and transition matrix are given by the formulas
\begin{equation}
\label{e:pi&P}
\pi(i) = \frac{u(i)^p}{\|u\|_p^p}, \qquad P(i,j) = \frac{Q(i-j)u(j)}{(Q*u)(i)},
\end{equation}
and, conversely, the single-site marginal and transition matrix of every $\mu\in \mathscr{G}(d,Q)$ have the form \eqref{e:pi&P} for a unique $u\in \mathscr{U}(d,Q)$.
Moreover, this correspondence is equivariant with respect to the actions generated by $\tau$, $\iota$, and $\jmath$ on $\mathscr{G}(d,Q)$ and on $\mathscr{U}(d,Q)$.
\end{prop}

The above proposition can be seen as a particular case of a more general result of Zachary \cite[Theorem 4.1]{zac83}. Solutions of \eqref{e:main-eq} are known under the name of \textit{boundary laws}. For the sake of completeness and in order to clarify the role of the irreducibility assumption and of the summability of the boundary law $u$, we provide a self-contained proof.

\begin{proof}
Assume first that $u$ is in $\mathscr{U}(d,Q)$ and define $\pi$ and $P$ by \eqref{e:pi&P}. Then $\pi$ is a probability measure on $\Z$, $P$ is a stochastic matrix indexed by $\Z^2$, and the reversibility identity \eqref{e:reversibility} holds. Denote by $\mu=M(u)$ the unique homogeneous Markov probability measure on $(\Omega,\mathscr{A})$ with single-site marginal $\pi$ and transition matrix $P$. We need to check that $\mu$ satisfies the Gibbs property \eqref{e:gibbs}.

Let $x\in V$, and let $y_1,\dots,y_{d+1}$ be its neighbouring vertices. We fix a tuple $(j_1,\dots,j_{d+1})$ in $\Z^{d+1}$ such that 
\[
\mu(\sigma_{y_1}=j_1, \dots, \sigma_{y_{d+1}} = j_{d+1})>0.
\]
By the Markov condition \eqref{e:markov} applied to the finite subtree determined by $\{x,y_1,\dots,y_{d+1}\}$ and rooted at $x$, we have for every $k\in \Z$
\[
\begin{split}
\mu(\sigma_x=k, \; \sigma_{y_1}=j_1, \dots, \sigma_{y_{d+1}} = j_{d+1}) &= 
\pi(k) \prod_{h=1}^{d+1} P(k,j_h) \stackrel{\eqref{e:pi&P}}{=} \frac{u(k)^p}{\|u\|_p^p} \prod_{h=1}^{d+1} \frac{Q(k-j_h) u(j_h)}{(Q*u)(k)} \\
& \stackrel{\eqref{e:main-eq}}{=} \frac{1}{\|u\|_p^p} \prod_{h=1}^{d+1} Q(k-j_h) u(j_h).
\end{split}
\] 
Adding over $k\in \Z$ we find
\[
\mu(\sigma_{y_1}=j_1, \dots, \sigma_{y_{d+1}} = j_{d+1})  = \frac{1}{\|u\|_p^p} \sum_{k\in \Z} \prod_{h=1}^{d+1} Q(k-j_h) u(j_h).
\]
This identity implies that $\prod_{h=1}^{d+1} Q(k-j_h)>0$ for at least one $k\in \Z$ and, together with the previous one, gives us
\[
\mu(\sigma_x=i \mid \sigma_{y_1}=j_1, \dots, \sigma_{y_{d+1}} = j_{d+1})  = \frac{ \prod_{h=1}^{d+1} Q(i-j_h) u(j_h) }{\sum_{k\in \Z} \prod_{h=1}^{d+1} Q(k-j_h) u(j_h)}, \qquad \forall i\in \Z,
\]
proving the Gibbs property \eqref{e:gibbs}.

\medskip

Conversely, let $\mu$ be in $\mathscr{G}(d,Q)$ and denote by $\pi$ and $P$ its single-site marginal and transition matrix. We wish to prove that $\pi$ and $P$ are determined through \eqref{e:pi&P} by some $u\in \mathscr{U}(d,Q)$. Here is where the irreducibility of $Q$ will be used.

Fix again some $x\in V$ and denote by $y_1,\dots,y_{d+1}$ its neighbouring vertices. Fix also some $(j_1,\dots,j_{d+1})\in \Z^{d+1}$ such that
\begin{equation}
\label{e:pequA}
\mu(\sigma_{y_1}=j_1, \dots, \sigma_{y_{d+1}} = j_{d+1})>0.
\end{equation}
By the Markov property we have for every $i\in \Z$ 
\[
\mu(\sigma_x = i, \; \sigma_{y_1}=j_1, \dots, \sigma_{y_{d+1}} = j_{d+1} ) = \pi(i) \prod_{h=1}^{d+1} P(i,j_h).
\]
By the Gibbs property, the above quantity is also equal to
\[
\begin{split}
& \mu(\sigma_x = i \mid  \sigma_{y_1}=j_1, \dots, \sigma_{y_{d+1}}=j_{d+1})  \mu( \sigma_{y_1}=j_1, \dots, \sigma_{y_{d+1}}=j_{d+1})  \\ &=  \mu(  \sigma_{y_1}=j_1, \dots, \sigma_{y_{d+1}}=j_{d+1}) \frac{ \prod_{h=1}^{d+1} Q(i-j_h) }{\sum_{k\in \Z} \prod_{h=1}^{d+1} Q(k-j_h) }.
\end{split}
\]
This shows that for every $(j_1,\dots,j_{d+1})\in \Z^{d+1}$ such that \eqref{e:pequA} holds and every $i\in \Z$ we have
\begin{equation}
\label{e:pequB}
\pi(i) \prod_{h=1}^{d+1} P(i,j_h) = C(j_1,\dots,j_{d+1}) \prod_{h=1}^{d+1} Q(i-j_h),
\end{equation}
for some positive function $C$.

\medskip

\noindent \textsc{Claim. } The support of $\pi$ is the whole $\Z$ and there exists a positive function $v: \Z \rightarrow (0,+\infty)$ such that 
\begin{equation}
\label{e:pequC}
\pi(i)^{\frac{1}{d+1}} P(i,j) = Q(i-j) v(j), \qquad \forall (i,j)\in \Z^2.
\end{equation}

\medskip

Let $j$ be in the support of $\pi$. Let $x$ and $y$ be neighbouring vertices. From the Markov property we find
\[
\pi(j) = \mu(\sigma_y = j) = \sum_{k\in \Z} \mu(\sigma_x = k, \; \sigma_y = j) = \sum_{k\in \Z} \pi(k) P(k,j),
\]
and hence there is some $k\in \Z$ such that $\pi(k) P(k,j)>0$. Denoting by $y_1,\dots,y_{d+1}$ the neighbours of $x$, we then have
\[
\mu(\sigma_x = k, \; \sigma_{y_1} = \dots = \sigma_{y_{d+1}} = j ) = \pi(k) P(k,j)^{d+1} >0,
\]
and we deduce that \eqref{e:pequA} holds for the tuple $(j_1,\dots,j_{d+1}) = (j,\dots,j)$. Then \eqref{e:pequB} gives us
\begin{equation}
\label{e:pequD}
\pi(i) P(i,j)^{d+1} = C(j) Q(i-j)^{d+1}, \qquad \forall i\in \Z,
\end{equation}
for a suitable positive number $C(j)$. This identity implies that $\pi(i)>0$ for all $i\in \Z$ such that $i-j$ belongs to the support of $Q$. Since this support generates $\Z$, the graph with vertex set $\Z$ and edges given by unordered pairs $\{i,j\}$ with $Q(i-j)>0$ is connected. By iterating the above argument, we conclude that the support of $\pi$ is the whole $\Z$ and \eqref{e:pequD} holds for every pair $(i,j)\in \Z^2$. Therefore, \eqref{e:pequC} holds with $v(j) := C(j)^{\frac{1}{d+1}}$ and the claim is proven.

\medskip

Since $P$ is stochastic, summing over $j\in \Z$ in \eqref{e:pequC} gives us
\begin{equation}
\label{e:pequE}
(Q*v)(i) = \pi(i)^{\frac{1}{d+1}}, \qquad \forall i\in \Z.
\end{equation}
From \eqref{e:pequC} and the reversibility condition \eqref{e:reversibility}, we obtain for every $(i,j)\in \Z^2$
\[
\pi(i)^{\frac{d}{d+1}} Q(i-j) v(j) \stackrel{\eqref{e:pequC}}{=} \pi(i) P(i,j) \stackrel{\eqref{e:reversibility}}{=} \pi(j) P(j,i) \stackrel{\eqref{e:pequC}}{=}\pi(j)^{\frac{d}{d+1}} Q(j-i) v(i).
\]
Since $Q$ is even, whenever $Q(i-j)>0$ this simplifies to
\[
\pi(j)^{-\frac{d}{d+1}} v(j) = \pi(i)^{-\frac{d}{d+1}} v(i).
\]
Using again the fact that the graph defined by the support of $Q$ is connected, we deduce that there is a positive number $c$ such that
\begin{equation}
\label{e:pequF}
v(i) = c\,  \pi(i)^{\frac{d}{d+1}}, \qquad \forall i\in \Z.
\end{equation}
Together with \eqref{e:pequE}, this implies
\begin{equation}
\label{e:pequG}
v = c (Q*v)^d.
\end{equation}
Since $\pi$ is a probability measure, \eqref{e:pequE} implies that $Q*v$ is in $\ell^{d+1}(\Z)$ and from \eqref{e:pequG} we deduce that $v$ is in $\ell^{\frac{d+1}{d}}(\Z) = \ell^p(\Z)$. Setting
\[
u:= c^{\frac{1}{d-1}} v,
\]
we obtain a positive function in $\ell^p(\Z)$ solving
\begin{equation}
\label{e:pequH}
u = (Q*u)^d.
\end{equation}
This is the element $u=U(\mu)$ of $\mathscr{U}(d,Q)$ we were looking for. From \eqref{e:pequF} we find
\begin{equation}
\label{e:pequI}
\pi(i) = \left( \frac{v(i)}{c} \right)^p = \frac{u(i)^p}{c^{\frac{d+1}{d-1}}}.
\end{equation}
Since $\pi$ is a probability measure, 
\[
c^{\frac{d+1}{d-1}} = \sum_{i\in \Z} u(i)^p = \|u\|_p^p,
\]
and the first identity in \eqref{e:pi&P} holds. Moreover
\[
P(i,j) \stackrel{\eqref{e:pequC}}{=} \frac{Q(i-j)u(j)}{c^{\frac{1}{d-1}} \pi(i)^{\frac{1}{d+1}}}  \stackrel{\eqref{e:pequI}}{=} \frac{Q(i-j)u(j)}{u(i)^{\frac{1}{d}}} 
\stackrel{\eqref{e:pequH}}{=} \frac{Q(i-j)u(j)}{(Q*u)(i)},
\]
proving the second identity in \eqref{e:pi&P}.

Summarizing: we have defined maps 
\[
M: \mathscr{U}(d,Q) \rightarrow \mathscr{G}(d,Q), \qquad U: \mathscr{G}(d,Q)\rightarrow \mathscr{U}(d,Q),
\]
such that $M(u)$ is determined by $u$ through \eqref{e:pi&P} and $U(\mu)$ satisfies \eqref{e:pi&P}. The composition $M\circ U$ is the identity on $\mathscr{G}(d,Q)$ because for every $\mu\in \mathscr{G}(d,Q)$  the Markov probability measures $\mu$ and $M(U(\mu))$ have the same single-site marginals and transition matrices, and hence coincide. Given $u\in \mathscr{U}(d,Q)$, set  $u'=U(M(u))$. By the first identity in \eqref{e:pi&P}, the single-site marginal $\pi$ of $M(u)$ has the expressions
\[
\frac{u(i)^p}{\|u\|_p^p} = \pi(i) =  \frac{u'(i)^p}{\|u'\|_p^p}, \qquad \forall i \in \Z,
\]
and hence $u'=a u$ for some $a>0$. The fact that both $u$ and $u'$ solve \eqref{e:main-eq} implies that $a=1$ and proves that the composition $U\circ M$ is the identity on $\mathscr{U}(d,Q)$. Therefore, $M$ and $U$ are inverses of each other.

The equivariance of $M$ and $U$ with respect to the $\Z$-action generated by $\tau$ and the $\Z_2$-actions generated by $\iota$ and $\jmath$
follows from \eqref{e:pi&P} by reading these actions on $\mathscr{G}(d,Q)$ in terms of the single-site marginal and the transition matrix. 
\end{proof}

\section{Existence of two Gibbs measures}
\label{s:ex2sol}

In the previous section, we assumed that the interaction function $Q$ is in $\ell^d(\Z)$. Starting with this section, we  require the stronger assumption that $Q$ is in $\ell^{\frac{d+1}{2}}(\Z)$, so Young's inequality gives us
\begin{equation}
\label{e:young}
\|Q*u\|_{d+1} \leq \|Q\|_{\frac{d+1}{2}} \|u\|_p, \qquad \forall u\in \ell^p(\Z),
\end{equation}
where as before $p=\frac{d+1}{d}$. Since $p'=d+1$ is the conjugate exponent to $p$, \eqref{e:young} implies 
\begin{equation}
\label{e:main-ineq}
\langle Q* u, v \rangle \leq \|Q\|_{\frac{d+1}{2}} \|u\|_p \|v\|_p \qquad \forall u,v \in \ell^p(\Z),
\end{equation}
where  
\[
\langle w,v \rangle := \sum_{i\in \Z} w(i) v(i)
\]
denotes the standard duality pairing of $w\in \ell^{p'}(\Z)$ and $v\in \ell^p(\Z)$. Note that the assumption that $Q$ is even implies the symmetry of the convolution by $Q$:
\begin{equation}
\label{e:symmetry}
\langle Q* u, v \rangle = \langle Q* v, u \rangle \qquad \forall u,v \in \ell^p(\Z).
\end{equation}

If $Q'=c\, Q$ with $c>0$, the Gibbs measure $\mu\in \mathscr{G}(d,Q)=\mathscr{G}(d,Q')$ corresponds, via Proposition \ref{p:zachary}, to some
$u\in \mathscr{U}(d,Q)$ and $u'\in \mathscr{U}(d,Q')$ which are related by $u':=c^{- \frac{d}{d-1}} u$. Since
\[
\|Q\|_{\frac{d+1}{2}}^{\frac{d+1}{d-1}} \|u\|_p^p = \|Q'\|_{\frac{d+1}{2}}^{\frac{d+1}{d-1}} \|u'\|_p^p,
\]
the positive quantity
\begin{equation}
\label{e:defnw}
w := \|Q\|_{\frac{d+1}{2}}^{\frac{d+1}{d-1}}  \|u\|_p^p
\end{equation}
does not depend on the normalization of $Q$. This quantity turns out to be significant in the construction of superpositions of Gibbs measures in Section \ref{s:superposition} below, and we call it the \textit{weight} of the Gibbs measure $\mu\in \mathscr{G}(d,Q)$ determined by $u\in \mathscr{U}(d,Q)$.

If $u\in \mathscr{U}(d,Q)$ is the function associated to a Gibbs measure $\mu$, then
\begin{equation}
\label{e:convo}
\|u\|_p^p = \langle u^{\frac{1}{d}}, u \rangle = \langle Q*u, u \rangle \leq \|Q\|_{\frac{d+1}{2}} \|u\|_p^2
\end{equation}
and hence
\[
\|u\|_p^p \geq  \|Q\|_{\frac{d+1}{2}}^{- \frac{d+1}{d-1}},
\]
which implies that the weight of any $\mu\in \mathscr{G}(d,Q)$ satisfies $w\geq 1$. 
Actually
\[
w>1
\]
if $Q$ is irreducible, because in this case $u$ is positive and the Young inequality \eqref{e:young}, on which the inequality in \eqref{e:convo} is based, is strict for positive functions.

After these preliminaries, we can state our first existence result for Gibbs measures, in which $\iota$ and $\jmath$ are the involutions from \eqref{e:defniota} and \eqref{e:defnjmath}.

\begin{thm}
\label{t:2gibbs}
Assume that the interaction function $Q$ is irreducible and $\frac{d+1}{2}$-summable. Then there exist Gibbs measures $\mu_1$ and $\mu_2$ in $\mathscr{G}(d,Q)$ such that  $\iota_*(\mu_1) = \mu_1$ and $\jmath_*(\mu_2) = \mu_2$.
In particular, $\mu_1$ and $\mu_2$ define two distinct elements of $\mathscr{G}(d,Q)/\Z$.
\end{thm}

The conclusion that $\mu_1$ and $\mu_2$ define two distinct elements of $\mathscr{G}(d,Q)/\Z$ follows from the fact that the single-site marginals $\pi_1$ and $\pi_2$ of $\mu_1$ and $\mu_2$ satisfy
\[
\pi_1(-i) = \pi_1(i), \qquad \pi_2(1-i) = \pi_2(i) \qquad \forall i\in \Z,
\]
so the identity $\mu_1 = \tau^n_* \mu_2$ for some $n\in \Z$ would imply 
\[
\pi_1(i) = \pi_1(-i) = \pi_2(-i-n)= \pi_2(i+n+1)= \pi_1(i+2n+1) \qquad \forall i\in \Z,
\] 
which is impossible because the function $\pi_1\in \ell^1(\Z)\setminus \{0\}$ cannot be invariant under the translation $\tau^{2n+1}$. 

In the Introduction, this theorem is labeled as Theorem \ref{t:ex2-main}. In the remaining part of this section, we show how it can be reduced to the maximization of a suitable functional. 

By \eqref{e:main-ineq}, the quadratic and translation-invariant functional
\[
\mathscr{Q}(u) := \langle Q* u, u \rangle = \sum_{(i,j)\in \Z^2} Q(j-i) u(i) u(j)
\]
is continuously differentiable on $\ell^p(\Z)$ with
\[
\diff \mathscr{Q}(u)[v] = 2 \langle Q* u, v \rangle \qquad \forall u,v\in \ell^p(\Z).
\]
Therefore, the translation-invariant functional
\[
\mathscr{E}(u) := \frac{1}{p} \|u\|_p^p - \frac{1}{2} \mathscr{Q}(u) = \frac{d}{d+1} \sum_{i\in \Z} | u(i) |^{\frac{d+1}{d}} - \frac{1}{2} \sum_{(i,j)\in \Z^2} Q(j-i) u(i) u(j)
\]
is continuously differentiable on $\ell^p(\Z)$ with
\[
\diff \mathscr{E}(u)[v] = \bigl\langle |u|^{\frac{1}{d}-1} u - Q* u, v \bigr\rangle\qquad \forall u,v \in \ell^p(\Z).
\]
The above identity shows that the critical points of $\mathscr{E}$ are precisely the $u\in \ell^p(\Z)$ such that 
\[
|u|^{\frac{1}{d}-1} u = Q * u.
\]
Every nonnegative solution $u\in \ell^p(\Z)$ of the above equation satisfies
\[
u = (Q*u)^d,
\]
and, by the irreducibility of $Q$, $u$ is either positive or identically zero. Therefore, the nonnegative critical points of $\mathscr{E}$ in $\ell^p(\Z)\setminus \{0\}$ are precisely the elements of $\mathscr{U}(d,Q)$.

Denote by $\mathscr{S}$ the unit sphere
\[
\mathscr{S}:= \{u\in \ell^p(\Z) \mid \|u\|_p = 1\}.
\]
By the Lagrange multiplier theorem, $u\in \mathscr{S}$ is a critical point of the restriction of $\mathscr{Q}$ to $\mathscr{S}$ if, and only if, there exists $\lambda\in \R$ such that
\[
\diff \mathscr{Q}(u)[v] = \lambda \bigl\langle |u|^{\frac{1}{d}-1} u, v \bigr\rangle \qquad \forall v\in \ell^p(\Z),
\]
which is equivalent to the equation 
\begin{equation}
\label{e:lagrange}
\lambda |u|^{\frac{1}{d}-1} u = 2 \, Q * u.
\end{equation}
Since 
\[
\langle |u|^{\frac{1}{d}-1} u,  u \rangle = \|u\|_p^p = 1, \qquad \forall u\in \mathscr{S},
\]
\eqref{e:lagrange} gives us
\[
\lambda = 2 \langle Q * u , u \rangle.
\]
Now assume that the critical point $u$ of $\mathscr{Q}|_{\mathscr{S}}$ is nonnegative. Since $Q*u$ is nonnegative and not identically zero, the number $\lambda$ in  \eqref{e:lagrange} is positive and this identity reads
\[
u = \Bigl(\frac{2}{\lambda}\Bigr)^d (Q*u)^d.
\]
Therefore,
\[
v:= \Bigl(\frac{2}{\lambda}\Bigr)^{\frac{d}{d-1}} u
\]
is a nonnegative solution of $v = (Q*v)^d$ in $\ell^p(\Z) \setminus \{0\}$, and hence an element of $\mathscr{U}(d,Q)$.

Conversely, if $u$ is an element of $\mathscr{U}(d,Q)$ then $u/\|u\|_p$ is a critical point of the restriction of $\mathscr{Q}$ to $\mathscr{S}$. Note also that the weight of the Gibbs measure $\mu$ determined by $u\in \mathscr{U}(d,Q)$ is
\begin{equation}
\label{e:weight}
w = \|Q\|_d^{\frac{d+1}{d-1}} \|u\|_p^p = \left( \frac{1}{\|Q\|_d} \mathscr{Q} \Bigl( \frac{u}{\|u\|_p} \Bigr) \right)^{-\frac{d+1}{d-1}}.
\end{equation}

Moreover, the functionals $\mathscr{Q}$, $\mathscr{E}$, and the set $\mathscr{S}$, are invariant under the involutions
\[
\iota^*, \jmath^* : \ell^p(\Z) \rightarrow \ell^p(\Z), \qquad (\iota^* u)(i) = u(-i), \qquad (\jmath^* u)(i) = u(1-i), \qquad \forall i\in \Z.
\]
By the principle of symmetric criticality, see e.g.\ \cite{pal79}, maximizers of $\mathscr{Q}$ on the submanifolds
\[
\begin{split}
\mathscr{S}_{\iota} &:=\mathscr{S} \cap \mathrm{Fix} (\iota^*) = \{ u\in \mathscr{S} \mid u(-i) = u(i)\; \forall i\in \Z\}, \\
\mathscr{S}_{\jmath} &:= \mathscr{S} \cap \mathrm{Fix} (\jmath^*) = \{ u\in \mathscr{S} \mid u(1-i) = u(i)\; \forall i\in \Z\},
\end{split}
\]
are critical points of the restriction of $\mathscr{Q}$ to $\mathscr{S}$. 

By Proposition \ref{p:zachary}, Theorem \ref{t:2gibbs}/Theorem \ref{t:ex2-main} is an immediate consequence of the following result, whose proof is discussed in the next section.

\begin{thm}
\label{t:maximizers}
The functional $\mathscr{Q}$ admits nonnegative maximizers on $\mathscr{S}$, $\mathscr{S}_{\iota}$, and $\mathscr{S}_{\jmath}$.
\end{thm}

Thanks to \eqref{e:weight}, the Gibbs measures corresponding to the above maximizers minimize the weight among all elements of $\mathscr{G}(d,Q)$, in the case of free maximization on $\mathscr{S}$, or among all $\iota_*$-symmetric or $\jmath_*$-symmetric 
elements of $\mathscr{G}(d,Q)$, in the case of maximization on $\mathscr{S}_{\iota}$ and $\mathscr{S}_{\jmath}$. 

\section{Existence of maximizers}
\label{s:exmax}

The proof of Theorem \ref{t:maximizers} is an application of the concentration-compactness principle, see \cite{lio84}. The following discrete version of the concentration-compactness lemma can be deduced from the more general \cite[Lemma I.4.3]{str08}. The proof in the discrete setting is actually simpler, and we provide it for the reader's convenience.

\begin{prop}
\label{p:conc-comp}
Let $(u_n)$ be a sequence in $\ell^1(\Z)$ with $u_n\geq 0$ and $\|u_n\|_1=1$ for every $n\in \N$. Then there exists a subsequence $(u_n')$ of $(u_n)$ and a sequence $(k_n)$ in $\Z$ such that, setting
\[
v_n (i) := u_n'(i-k_n), \qquad \forall i\in \Z,
\] 
one of the following three conditions holds:
\begin{enumerate}[1.]
\item (Compactness) For every $\epsilon>0$ there exists $r\in \N$ such that
\[
\sum_{|i| \geq r} v_n(i) < \epsilon, \qquad \forall n\in \N.
\]
\item (Vanishing) ${\displaystyle \lim_{n \rightarrow \infty} \|v_n\|_{\infty} = 0}$.
\item (Dychothomy) There exists $\lambda\in (0,1)$ and positive integers $r_n<s_n$ with $s_n - r_n\rightarrow +\infty$ for $n\rightarrow \infty$ such that:
\[
\lim_{n\rightarrow \infty} \sum_{|i|\leq r_n} v_n(i) = \lambda, \qquad \lim_{n\rightarrow \infty} \sum_{|i|\geq s_n} v_n(i) = 1-\lambda.
\]
\end{enumerate}
Moreover, we can take $k_n=0$, and hence $v_n=u_n'$, provided that in condition 3 the value $\lambda=0$ is also allowed. In this case, condition 2 is not needed: either 1 holds, or 3 holds with $\lambda\in [0,1)$.
\end{prop}

\begin{proof}
We start by proving the last statement, in which no translations are involved. For every $n\in \N$, L\'evy's concentration function
\[
F_n: \N \rightarrow [0,1], \qquad F_n(r) := \sum_{|i|\leq r} u_n(i),
\]
is non-decreasing and converges to $1$ for $r\rightarrow \infty$. By a diagonal argument, up to replacing $(u_n)$ by a subsequence we may assume that $(F_n)$ converges pointwise to some function $F: \N \rightarrow [0,1]$, which is then non-decreasing. We set
\[
\lambda := \lim_{r\rightarrow \infty} F(r) = \sup_{r\in \N} F(r) \in [0,1].
\]
Assume that $\lambda=1$ and let $\epsilon>0$. Then there exist $r_0\in \N$ such that $F(r_0)>1-\epsilon$ and $n_0\in \N$ such that $F_n(r_0)>1-\epsilon$ for every $n\geq n_0$. Since each $F_n$ converges to $1$ at infinity, we can find $r> r_0$ such that $F_n(r-1)>1-\epsilon$ also for every $n<n_0$. Since $r-1\geq r_0$ and each $F_n$ is non-decreasing, the same inequality holds for every $n\geq n_0$. Therefore,
\[
\sum_{|i|\geq r} u_n(i) = 1 - F_n(r-1) < \epsilon \qquad \forall n\in \N,
\]
and condition 1 holds with $v_n=u_n$.

Assume now that $\lambda<1$. For every positive integer $k$ we can find an integer $R_k\geq k$ such that $F(R_k)> \lambda - \frac{1}{k}$, and hence
\[
\lambda - \frac{1}{k} < F(R_k) \leq F(R_k+k) \leq \lambda.
\]
By the pointwise convergence of $(F_n)$ to $F$, we can find a strictly increasing sequence of natural numbers $(n_k)$ such that
\begin{equation}
\label{e:Fnk}
\lambda - \frac{1}{k} < F_n(R_k) \leq F_n(R_k+k) < \lambda + \frac{1}{k} \qquad \forall n\geq n_k.
\end{equation}
We set $r_n:= R_k$ and $s_n:= R_k + k+1$ for $n_k \leq n < n_{k+1}$, and $r_n:=1$, $s_n:=2$ for $n<n_1$. Then $s_n-r_n\rightarrow +\infty$ and by \eqref{e:Fnk}
\[
\lim_{n\rightarrow \infty} \sum_{|i|\leq r_n} u_n(i) = \lim_{n\rightarrow \infty} F_n(r_n) = \lambda, \qquad \lim_{n\rightarrow \infty} \sum_{|i|\geq s_n} u_n(i) = \lim_{n\rightarrow \infty} \bigl( 1 - F_n(s_n-1) \bigr) = 1-\lambda,
\]
so condition 3 holds with $v_n=u_n$ and $\lambda\in [0,1)$. This proves the last statement.

We now deduce the first statement from the last one. If the limit inferior of $\|u_n\|_{\infty}$ is zero, then condition 2 holds for a subsequence of $(u_n)$, with $k_n=0$. 
Otherwise, the number $\delta:= \inf_{n\in \N} \|u_n\|_{\infty}$ is positive,
because each $\|u_n\|_{\infty}$ is positive. Since $u_n(i)$ tends to zero for $|i|\rightarrow \infty$, there exists $j_n\in \Z$ such that $u_n(j_n) = \|u_n\|_{\infty}$. By the last statement, applied to the translated sequence $\tilde{u}_n := u_n(\cdot + j_n)$, we find a subsequence of $(\tilde{u}_n)$, which we denote by $(v_n)$, satisfying either condition 1, or condition 3 for some $\lambda\in [0,1)$. Each $v_n$ has the form $u_n'(\cdot - k_n)$, where $(u_n')$ is a subsequence of $(u_n)$ and $(k_n)$ is the corresponding subsequence of $(-j_n)$. Moreover, $v_n(0)= \|v_n\|_{\infty} \geq \delta$ for every $n$. If condition 3 holds, then
\[
\lambda = \lim_{n\rightarrow \infty} \sum_{|i|\leq r_n} v_n(i) \geq \liminf_{n\rightarrow \infty} v_n(0) \geq \delta > 0,
\]
and hence $\lambda$ belongs to $(0,1)$. Therefore, either Compactness or Dichotomy holds for $(v_n)$.
\end{proof}

We now show how the above result leads to the existence of a maximizer of $\mathscr{Q}$ on $\mathscr{S}$. Set
\[
M:= \sup_{u\in \mathscr{S}} \mathscr{Q}(u).
\]
The fact that $Q$ is non-negative and somewhere positive implies that $M>0$, and by  \eqref{e:main-ineq} we have $M \leq \|Q\|_{\frac{d+1}{2}}$. Moreover, the $2$-homogeneity of $\mathscr{Q}$ implies the inequality
\begin{equation}
\label{e:better}
\mathscr{Q}(u) \leq M \|u\|_p^2 \qquad \forall u\in \ell^p(\Z).
\end{equation}
Let $(u_n)\subset \mathscr{S}$ be a maximizing sequence for $\mathscr{Q}$:
\[
\lim_{n\rightarrow \infty} \mathscr{Q}(u_n) = M.
\]
Since
\begin{equation}
\label{e:meglioass}
\mathscr{Q}(u) \le \mathscr{Q}(|u|) \qquad \forall u\in \mathscr{S},
\end{equation}
we may assume that $u_n\geq 0$  for every $n$. We apply the above proposition to the sequence $(u_n^p)$. Up to replacing $(u_n)$ with a subsequence of translations of the original sequence, we can therefore assume that one of the following three facts holds:
\begin{enumerate}[1.]
\item (Compactness) For every $\epsilon>0$ there exists $r\in \N$ such that
\[
\sum_{|i| \geq r} u_n(i)^p < \epsilon \qquad \forall n\in \N.
\]
\item (Vanishing) ${\displaystyle \lim_{n \rightarrow \infty} \|u_n\|_{\infty} = 0}$.
\item (Dychothomy) There exists $\lambda\in (0,1)$ and positive integers $r_n<s_n$ with $s_n - r_n\rightarrow +\infty$ for $n\rightarrow \infty$ such that:
\[
\lim_{n\rightarrow \infty} \sum_{|i|\leq r_n} u_n(i)^p = \lambda, \qquad \lim_{n\rightarrow \infty} \sum_{|i|\geq s_n} u_n(i)^p = 1-\lambda.
\]
\end{enumerate}

Our aim is to prove that neither Vanishing nor Dychothomy hold. Once this is proven, the remaining alternative is Compactness, which by the Riesz-Kolmogorov characterization of compact subsets of $\ell^p(\Z)$ implies that $(u_n)$ has a subsequence which converges to some $u$ in the norm topology of $\ell^p$. Since $\mathscr{S}$ is closed and $\mathscr{Q}$ is continuous in the norm topology of $\ell^p$, $u$ is the required nonnegative maximizer of $\mathscr{Q}$ on $\mathscr{S}$.

\medskip

Assume by contradiction that Vanishing holds. Fix some $\epsilon>0$ and let $q\in \N$ be large enough, so that the truncated function $\tilde{Q}:= Q \mathbbm{1}_{[-q,q]}$ satisfies
\[
\|Q-\tilde{Q} \|_{\frac{d+1}{2}} < \epsilon.
\]
Then we find the chain of inequalities
\[
\begin{split}
\mathscr{Q}(u_n) &= \sum_{k \in \Z} Q(k) (u_n * \iota^*u_n) (k) = \sum_{|k|\leq q} Q(k) (u_n* \iota^*u_n) (k) + \sum_{k\in \Z} (Q-\tilde{Q})(k) (u_n*\iota^* u_n)  (k) \\
&\leq \|Q\|_{\infty} \sum_{|k|\leq q} (u_n*\iota^*u_n)(k) + \|Q - \tilde{Q} \|_{\frac{d+1}{2}} \|u_n\|_p^2 \leq \|Q\|_{\infty} (2q+1) \|u_n * \iota^*u_n\|_{\infty} + \epsilon \\ &\leq 
 \|Q\|_{\infty} (2q+1) \|u_n\|_p \|u_n\|_{p'} + \epsilon =  \|Q\|_{\infty} (2q+1)  \|u_n\|_{p'} + \epsilon
\end{split}
\]
where we used \eqref{e:main-ineq} and the fact that $p'=d+1$ is the conjugate exponent of $p$. The Vanishing assumption tells us that $\|u_n\|_{\infty}$ tends to zero. Together with the boundedness of $\|u_n\|_p=1$ and the strict inequality $p<p'$, interpolation implies that $\|u_n\|_{p'}$ tends to zero. Then the above inequalities and the fact that $\epsilon>0$ is arbitrary imply that that the limes superior of $\mathscr{Q}(u_n)$ is non-positive, contradicting the fact that this sequence converges to $M>0$.

\medskip

Now assume by contradiction that Dychothomy holds. Write each $u_n$ as the sum
\[
u_n = v_n + w_n + z_n
\]
of three sequences with mutually disjoint supports
\[
v_n:= u_n \mathbbm{1}_{[-r_n,r_n]}, \quad w_n:= u_n \mathbbm{1}_{[-s_n,s_n] \setminus [-r_n,r_n]}, \quad z_n:= u_n \mathbbm{1}_{\Z \setminus [-s_n,s_n] },
\]
so that
\begin{equation}
\label{e:dico}
\lim_{n\rightarrow \infty} \|v_n\|_p = \lambda^{\frac{1}{p}}, \quad \lim_{n\rightarrow \infty} \|w_n\|_p = 0, \quad \lim_{n\rightarrow \infty} \|z_n\|_p =(1-\lambda)^{\frac{1}{p}}.
\end{equation}
By the quadraticity of $\mathscr{Q}$ and \eqref{e:symmetry} we have
\begin{equation}
\label{e:ide}
\mathscr{Q}(u_n) = \mathscr{Q}(v_n) + \mathscr{Q}(w_n) + \mathscr{Q}(z_n) + 2 \langle Q * v_n, w_n \rangle + 2 \langle Q * z_n, w_n \rangle + 2 \langle Q * v_n, z_n \rangle.
\end{equation}
From \eqref{e:better} and \eqref{e:dico} we deduce
\[
\begin{split}
\limsup_{n\rightarrow \infty} \mathscr{Q}(v_n) &\leq \limsup_{n\rightarrow \infty}  M \|v_n\|_p^2 = M \lambda^{\frac{2}{p}}, \\
\limsup_{n\rightarrow \infty} \mathscr{Q}(w_n) &\leq \limsup_{n\rightarrow \infty} M \|w_n\|_p^2 =0, \\
\limsup_{n\rightarrow \infty} \mathscr{Q}(z_n) &\leq \limsup_{n\rightarrow \infty} M \|z_n\|_p^2 = M (1-\lambda)^{\frac{2}{p}} .
\end{split}
\]
Moreover, from \eqref{e:main-ineq} and \eqref{e:dico}  we find
\[
\begin{split}
\limsup_{n\rightarrow \infty} \langle Q * v_n, w_n \rangle &\leq \limsup_{n\rightarrow \infty} \|Q\|_{\frac{d+1}{2}} \|v_n\|_p \|w_n\|_p  = 0, \\
\limsup_{n\rightarrow \infty} \langle Q * z_n, w_n \rangle &\leq \limsup_{n\rightarrow \infty} \|Q\|_{\frac{d+1}{2}} \|z_n\|_p \|w_n\|_p  = 0.
\end{split}
\]
There remains to bound the last term in \eqref{e:ide}.
Fix $\epsilon>0$, and let $q\in \N$ be so large that the truncated function $\tilde{Q}:= Q \mathbbm{1}_{[-q,q]}$ satisfies
\[
\|Q-\tilde{Q} \|_{\frac{d+1}{2}} \leq \epsilon.
\]
Write the last term in \eqref{e:ide} as 
\[
\langle Q * v_n, z_n \rangle = \langle (Q-\tilde{Q}) * v_n, z_n \rangle + \langle \tilde{Q} * v_n, z_n \rangle.
\]
If $n$ is large enough, then $\tilde{Q} * v_n$ and $z_n$ have disjoint support because $\tilde{Q}$ has bounded support and $s_n-r_n\rightarrow +\infty$. Thus, and the last term vanishes for $n$ large enough. Therefore
\[
\limsup_{n\rightarrow \infty} \langle Q * v_n, z_n \rangle \leq \limsup_{n\rightarrow \infty} \|Q-\tilde{Q}\|_{\frac{d+1}{2}} \|v_n\|_p \|z_n\|_p \leq \epsilon.
\]
By plugging the above bound into \eqref{e:ide}, we find
\[
\limsup_{n\rightarrow \infty} \mathscr{Q}(u_n) \leq M \lambda^{\frac{2}{p}} + M (1-\lambda)^{\frac{2}{p}} + 2\epsilon,
\]
and since $\epsilon>0$ was arbitrary we conclude that
\[
\limsup_{n\rightarrow \infty} \mathscr{Q}(u_n) \leq M \bigl( \lambda^{\frac{2}{p}} + (1-\lambda)^{\frac{2}{p}} \bigr).
\]
Since $\frac{2}{p}>1$ and $\lambda\in (0,1)$, we have
\[
 \lambda^{\frac{2}{p}} + (1-\lambda)^{\frac{2}{p}} < 1,
\]
and hence
\[
\limsup_{n \rightarrow \infty} \mathscr{Q}(u_n) < M,
\]
contradicting the fact that $(u_n)$ is a maximizing sequence. This contradiction excludes Dychothomy and concludes the proof of the existence of a nonnegative maximizer of $\mathscr{Q}$ on $\mathscr{S}$.

\medskip

We now turn to the existence of nonnegative maximizers of $\mathscr{Q}$ on $\mathscr{S}_{\iota}$ and $\mathscr{S}_{\jmath}$. In order to deal with these two cases at the same time, in what follows $\varrho$ denotes one of the two involutions $\iota^*$ and $\jmath^*$ of $\ell^p(\Z)$, and we set
\[
\mathscr{S}_{\varrho} := \mathscr{S} \cap \mathrm{Fix}(\varrho), \qquad M_{\varrho}:= \sup_{u\in \mathscr{S}_{\varrho}} \mathscr{Q}(u) \leq M.
\]
Note that $M_{\varrho}$ is positive, because $\mathscr{Q}(\mathbbm{1}_I)>0$ for every sufficiently large finite interval $I\subset \Z$, and $I$ can be chosen so that $\mathbbm{1}_I$ is fixed by $\varrho$. The $2$-homogeneity of $\mathscr{Q}$ implies the inequality
\begin{equation}
\label{e:better-sym}
\mathscr{Q}(u) \leq M_{\varrho} \|u\|_p^2 \qquad \forall u\in \mathrm{Fix}(\varrho).
\end{equation}
Let $(u_n)\subset \mathscr{S}_{\varrho}$ be a maximizing sequence for the restriction of $\mathscr{Q}$ to $\mathscr{S}_{\varrho}$. As before, we may assume that $u_n\geq 0$ for every $n$, by \eqref{e:meglioass} and because $|u|$ belongs to $\mathscr{S}_{\varrho}$ whenever $u$ does.

Here, we cannot translate the elements of a maximizing sequence without losing their symmetry, so we use the last statement of Proposition \ref{p:conc-comp}. By applying it to the sequence $(u_n^p)$ and passing to a subsequence, we may assume that either Compactness holds, or there exist $\lambda\in [0,1)$ and positive integers $r_n<s_n$ with $s_n - r_n\rightarrow +\infty$ such that
\begin{equation}
\label{e:dico-sym}
\lim_{n\rightarrow \infty} \sum_{|i|\leq r_n} u_n(i)^p = \lambda, \qquad \lim_{n\rightarrow \infty} \sum_{|i|\geq s_n} u_n(i)^p = 1-\lambda.
\end{equation}
If Compactness holds, then $(u_n)$ has a subsequence which converges in $\ell^p(\Z)$ to a nonnegative maximizer of $\mathscr{Q}$ on $\mathscr{S}_{\varrho}$, as in the case of $\mathscr{S}$. If \eqref{e:dico-sym} holds with $\lambda\in (0,1)$, we obtain a contradiction by the same argument which excluded Dichotomy in the case of $\mathscr{S}$. Indeed, if $\varrho=\iota^*$ then the functions $v_n$, $w_n$, and $z_n$ which are defined there are fixed by $\varrho$, so we can use \eqref{e:better-sym} instead of \eqref{e:better} and obtain
\[
M_{\varrho} = \lim_{n\rightarrow \infty} \mathscr{Q}(u_n) \leq M_{\varrho} \bigl( \lambda^{\frac{2}{p}} + (1-\lambda)^{\frac{2}{p}} \bigr) < M_{\varrho}.
\]
If $\varrho=\jmath^*$, the same is true after replacing the sets $\{|i|\leq r_n\}$ and $\{|i|\geq s_n\}$ in the definition of $v_n$ and $z_n$ by the $\jmath$-invariant sets $\{|i|\leq r_n\}\setminus \{-r_n\}$ and $\{|i|\geq s_n\} \cup \{1-s_n\}$. This does not affect \eqref{e:dico}, because $u_n(-r_n) = u_n(r_n+1)$ and $u_n(1-s_n)$ are values of $u_n$ at points $i$ with $r_n<|i|<s_n$, and by \eqref{e:dico-sym} the sum of $u_n(i)^p$ over these points tends to zero.

Therefore, it is enough to exclude the case in which \eqref{e:dico-sym} holds with $\lambda=0$. Unlike Vanishing, this case cannot be excluded by the above arguments: for instance, the symmetric sequence $\frac{1}{2} ( \mathbbm{1}_{\{n\}} + \mathbbm{1}_{\{-n\}})$ satisfies \eqref{e:dico-sym} with $\lambda=0$ but does not satisfy Vanishing. What we need is the following strict inequality.

\begin{lem}
\label{l:strict}
We have
\[
M_{\varrho} > 2^{1-\frac{2}{p}} M.
\]
\end{lem}

\begin{proof}
Let $u$ be a nonnegative maximizer of $\mathscr{Q}$ on $\mathscr{S}$, whose existence has been proven above. As discussed in Section \ref{s:ex2sol}, a positive multiple of $u$ belongs to $\mathscr{U}(d,Q)$, and hence $u$ is positive. The function $v:= \varrho\, u$ is also a positive maximizer of $\mathscr{Q}$ on $\mathscr{S}$, because both $\mathscr{Q}$ and the $p$-norm are $\varrho$-invariant. We denote by $u\vee v$ and $u\wedge v$ the pointwise maximum and minimum of $u$ and $v$. These functions are fixed by $\varrho$, and since $(u\vee v)^p + (u\wedge v)^p = u^p + v^p$ they satisfy
\begin{equation}
\label{e:maxmin}
\|u\vee v\|_p^p = 2 - \mu, \qquad \mbox{where} \quad \mu:= \|u\wedge v\|_p^p \in (0,1].
\end{equation}
Here, $\mu$ is positive because both $u$ and $v$ are positive.

We claim that
\begin{equation}
\label{e:lattice}
\mathscr{Q}(u\vee v) + \mathscr{Q}(u\wedge v) \geq \mathscr{Q}(u) + \mathscr{Q}(v).
\end{equation}
Since $Q$ is nonnegative, it is enough to check that
\[
(u\vee v)(i) \, (u\vee v)(j) + (u\wedge v)(i) \, (u\wedge v)(j) \geq u(i) u(j) + v(i) v(j) \qquad \forall i,j\in \Z,
\]
which is true because of the elementary inequality
\[
(a\vee b)\, (a'\vee b') + (a\wedge b)\, (a' \wedge b') \geq aa' + bb' \qquad \forall a,b,a',b'\in \R.
\]

Set $\theta:= \frac{2}{p} \in (1,2)$. By \eqref{e:lattice}, \eqref{e:better}, and \eqref{e:maxmin} we have
\begin{equation}
\label{e:lower-vee}
\mathscr{Q}(u\vee v) \geq 2M - \mathscr{Q}(u\wedge v) \geq 2M - M \|u\wedge v\|_p^2 = M \bigl( 2 - \mu^{\theta} \bigr).
\end{equation}
The function
\[
g:[0,1] \rightarrow \R, \qquad g(t) := 2 - t^{\theta} - 2 \Bigl( 1 - \frac{t}{2} \Bigr)^{\theta},
\]
has derivative
\[
g'(t) = \theta \Bigl( \Bigl( 1 - \frac{t}{2} \Bigr)^{\theta-1} - t^{\theta-1} \Bigr),
\]
which is positive on $(0,\frac{2}{3})$ and negative on $(\frac{2}{3},1]$. Since $g(0)=0$ and $g(1) = 1 - 2^{1-\theta}>0$, we deduce that $g$ is positive on $(0,1]$. Therefore, \eqref{e:lower-vee} and \eqref{e:maxmin} imply
\[
\mathscr{Q}(u\vee v) > 2 M  \Bigl( 1 - \frac{\mu}{2} \Bigr)^{\theta} = 2^{1-\theta} M \, (2-\mu)^{\theta} =  2^{1-\theta} M \, \|a\vee b\|_p^2.
\]
Since $u\vee v$ belongs to $\mathrm{Fix}(\varrho)$, the above inequality and \eqref{e:better-sym} give us $M_{\varrho} > 2^{1-\theta} M$.
\end{proof}

We can now exclude the case in which \eqref{e:dico-sym} holds with $\lambda=0$. Assume by contradiction that this is the case, and write each $u_n$ as the sum
\[
u_n = w_n + z_n^+ + z_n^-
\]
of the three functions with mutually disjoint supports
\[
w_n := u_n \mathbbm{1}_{(-s_n,s_n)}, \qquad z_n^+ := u_n \mathbbm{1}_{[s_n,+\infty)}, \qquad z_n^- := u_n \mathbbm{1}_{(-\infty,-s_n]}.
\]
By \eqref{e:dico-sym} with $\lambda=0$, $\|w_n\|_p$ tends to zero. By the symmetry of $u_n$, the numbers $\|z_n^+\|_p^p$ and $\|z_n^-\|_p^p$ coincide if $\varrho=\iota^*$, and their difference is $u_n(1-s_n)^p\leq \|w_n\|_p^p$ if $\varrho=\jmath^*$. Therefore,
\begin{equation}
\label{e:half}
\lim_{n\rightarrow \infty} \|w_n\|_p = 0, \qquad \lim_{n\rightarrow \infty} \|z_n^+\|_p^p =  \lim_{n\rightarrow \infty} \|z_n^-\|_p^p = \frac{1}{2}.
\end{equation}
By the quadraticity of $\mathscr{Q}$ and \eqref{e:symmetry} we have
\begin{equation}
\label{e:ide-sym}
\mathscr{Q}(u_n) = \mathscr{Q}(w_n) + 2 \langle Q*w_n, z_n^+ + z_n^- \rangle + \mathscr{Q}(z_n^+) + \mathscr{Q}(z_n^-) + 2 \langle Q* z_n^+, z_n^- \rangle.
\end{equation}
The first two terms on the right-hand side are infinitesimal by \eqref{e:main-ineq} and \eqref{e:half}. The functions $z_n^{\pm}$ are not fixed by $\varrho$, so for the next two terms we use \eqref{e:better} and \eqref{e:half}, obtaining
\[
\limsup_{n\rightarrow \infty} \bigl( \mathscr{Q}(z_n^+) + \mathscr{Q}(z_n^-) \bigr) \leq \lim_{n\rightarrow \infty} M \bigl( \|z_n^+\|_p^2 + \|z_n^-\|_p^2 \bigr) = 2^{1-\frac{2}{p}} M.
\]
Finally, the supports of $z_n^+$ and $z_n^-$ have distance at least $2s_n$, which diverges, so the last term in \eqref{e:ide-sym} is infinitesimal by the argument which we used above for the term $\langle Q*v_n,z_n \rangle$. We conclude that
\[
M_{\varrho} = \lim_{n\rightarrow \infty} \mathscr{Q}(u_n) \leq 2^{1-\frac{2}{p}} M,
\]
contradicting Lemma \ref{l:strict}. This contradiction concludes the proof of the existence of nonnegative maximizers of $\mathscr{Q}$ on $\mathscr{S}_{\iota}$ and $\mathscr{S}_{\jmath}$, and hence of Theorem \ref{t:maximizers}.

\section{Superposition of non-degenerate Gibbs measures}
\label{s:superposition}

Throughout this section, $Q\in \ell^{\frac{d+1}{2}}(\Z)$ is an even irreducible interaction function. 
Let $u\in \mathscr{U}(d,Q)$. By linearizing equation \eqref{e:main-eq} we obtain the equation
\begin{equation}
\label{e:lin-eq}  
v = d (Q*u)^{d-1} Q*v
\end{equation}
for $v\in \ell^p(\Z)$. We shall say that $u\in  \mathscr{U}(d,Q)$ is \textit{non-degenerate} if the above linear equation has no non-trivial solutions in $\ell^p(\Z)$. The Gibbs measure $\mu\in \mathscr{G}(d,Q)$ is said to be \textit{non-degenerate} if the corresponding $u\in \mathscr{U}(d,Q)$ is non-degenerate.

In this section, we set up the proof of the following result, which is precisely Theorem \ref{t:superpos-main} from the Introduction.

\begin{thm}
\label{t:superpos}
Assume that the interaction function $Q$ is irreducible and $\frac{d+1}{2}$-summable.
Let $\mu$ and $\mu'$ be not necessarily distinct non-degenerate elements of $\mathscr{G}(d,Q)$ with weights $w$ and $w'$.  Then there exists $N\in \N$  such that for every $n\in \Z$ with $|n|\geq N$ there is a non-degenerate Gibbs measure $\mu \#_n \mu'\in \mathscr{G}(d,Q)$ with weight $w_n$ such that 
\[
\lim_{n\rightarrow \pm\infty} w_n = w + w' \qquad \mbox{and} \qquad
\lim_{n\rightarrow \pm \infty} \bigl\| (\mu \#_n \mu')|_{\Lambda} - \bigl( \textstyle{\frac{w}{w+w'}} \, \mu|_{\Lambda} + \textstyle{\frac{w'}{w+w'}}\, (\tau_*^n \mu')|_{\Lambda} \bigr) \bigr\|_1 = 0,
\]
for every finite subset $\Lambda\subset V$.
\end{thm} 

In other words, all finite-dimensional marginals of $\mu\#_n \mu'$ are close in total variation to those of a convex combination of $\mu$ and of the translated measure $\tau_*^n \mu'$, with coefficients given by the weights. 

\begin{rem}
{\rm By applying the above theorem iteratively, we deduce that  the existence of one non-degenerate element in $\mathscr{G}(d,Q)$ implies the existence of infinitely many non-degenerate elements in $\mathscr{G}(d,Q)$ which also after projection to the quotient space $\mathscr{G}(d,Q)/\Z$ are pairwise distinct and have diverging weights.}
\end{rem}

Theorem \ref{t:superpos} will be deduced from the following more precise proposition.

\begin{prop}
\label{p:superpos}
Under the assumptions of Theorem \ref{t:superpos}, let $\pi$ and $\pi'$ be the single-site marginals, $P$ and $P'$ the transition matrices, and $w$ and $w'$ the weights of $\mu$ and $\mu'$. 
Then there exists $N\in \N$  such that for every $n\in \Z$ with $|n|\geq N$ there exists a non-degenerate Gibbs measure $\mu \#_n \mu'\in \mathscr{G}(d,Q)$ with  single-site marginal $\pi_n$, transition matrix $P_n$, and weight $w_n$  such that
\begin{eqnarray}
\label{e:asy0}
& {\displaystyle \lim_{n\rightarrow \pm \infty} w_n = w + w',} & \\
\label{e:asy1}
& {\displaystyle \lim_{n\rightarrow \pm \infty} \sum_{i\in \Z} \left| \pi_n(i) - \bigl( {\textstyle \frac{w}{w+w'}} \,\pi (i) + {\textstyle \frac{w'}{w+w'}} \, \pi'(i-n) \bigr) \right| = 0,} &\\
\label{e:asy2}
& {\displaystyle \lim_{n\rightarrow \pm\infty} \sum_{(i,j)\in \Z^2} \pi_n(i) \left| P_n(i,j) - \tilde{P}_n(i,j) \right| = 0,} &
\end{eqnarray}
where
\begin{equation}
\label{e:P'}
\tilde{P}_n(i,j) := \lambda_n(i) P(i,j) + (1-\lambda_n(i)) P'(i-n,j-n),
\end{equation}
with
\begin{equation}
\label{e:lambda}
\lambda_n(i) := \frac{(w \pi(i))^{\frac{1}{d+1}}}{(w \pi(i))^{\frac{1}{d+1}} + (w' \pi'(i-n))^{\frac{1}{d+1}}}\in (0,1).
\end{equation}
\end{prop}

This proposition  will be proven by showing how from two non-degenerate elements $u_1$ and $u_2$ of $\mathscr{U}(d,Q)$ we can construct a new element of this set which is suitably close to $u_1 + \tau^n_* (u_2)$. Instead of working in the variational setting on $\ell^p(\Z)$ as in the previous two sections, here it is convenient to identify elements of $\mathscr{U}(d,Q)$ with zeroes of a suitable map on $\ell^{d+1}(\Z)$. More precisely, we observe that $u$ is in $\mathscr{U}(d,Q)$ if, and only if, the function $x:=u^{\frac{1}{d}}\in \ell^{d+1}(\Z)$ is a not identically vanishing zero of the map
\[
\mathscr{F}= \mathscr{F}_Q: \ell^{d+1}(\Z) \rightarrow \ell^{d+1}(\Z), \quad \mathscr{F}(x) := x - \mathscr{K}(x), \qquad  \mathscr{K}(x) = \mathscr{K}_Q(x):= Q * |x|^d.
\]
The map $\mathscr{K}$ is easily seen to be of class $C^2$ with 
\begin{eqnarray}
\label{e:diff1}
& \diff \mathscr{K}(x)[y] = d\, Q * (|x|^{d-2} x y),&  \qquad \forall x,y\in \ell^{d+1}(\Z), \\
\nonumber
& \diff^2 \mathscr{K}(x)[y,z] = d(d-1) Q * (|x|^{d-2} yz), & \qquad \forall x,y,z\in \ell^{d+1}(\Z),
\end{eqnarray}
and the Young and H\"older inequalities give us the bounds
\begin{eqnarray}
\label{e:bd1}
\|\diff \mathscr{K}(x)[y]\|_{d+1} &\leq& d \|Q\|_{\frac{d+1}{2}} \|x\|_{d+1}^{d-1} \|y\|_{d+1}, \\
\label{e:bd2}
\|\diff^2 \mathscr{K}(x)[y,z]\|_{d+1} &\leq& d(d-1)  \|Q\|_{\frac{d+1}{2}} \|x\|_{d+1}^{d-2} \|y\|_{d+1} \|z\|_{d+1},
\end{eqnarray}
for all $x,y,z\in \ell^{d+1}(\Z)$.

\begin{lem}
For every $x\in \ell^{d+1}(\Z)$, the linear operator $\diff \mathscr{K}(x): \ell^{d+1}(\Z) \rightarrow \ell^{d+1}(\Z)$ is compact.
\end{lem}

\begin{proof}
Since $\diff \mathscr{K}$ is continuous from $\ell^{d+1}(\Z)$ to the space of bounded operators on $\ell^{d+1}(\Z)$, $\diff \mathscr{K}(x)$ is close to $\diff \mathscr{K}(x')$ in the operator norm topology when $\|x'-x\|_{d+1}$ small. Since the space of compact operators is closed and bi-infinite sequences with compact support are dense in $\ell^{d+1}(\Z)$, we may assume that $x$ has finite support. In this case, the operator $\diff \mathscr{K}(x)$ has finite rank and hence is compact.
\end{proof}

From the above lemma, we deduce that for every $x\in \ell^{d+1}(\Z)$ the operator
\[
\diff \mathscr{F}(x) = I - \diff \mathscr{K}(x) : \ell^{d+1}(\Z) \rightarrow \ell^{d+1}(\Z)
\]
is Fredholm of index zero. In particular, it is invertible if, and only if, it is injective. Its kernel is given by the elements $y\in \ell^{d+1}(\Z)$ such that
\[
y = d  \, Q* (|x|^{d-2} xy).
\]
If now $x=u^{\frac{1}{d}}$ with $u \in \mathscr{U}(d,Q)$, a simple computation and Young's inequality show that $y\in \ell^{d+1}(\Z)$ solves the above equation if, and only if, the function $v= x^{d-1} y\in \ell^{p}(\Z)$ solves \eqref{e:lin-eq}. We can summarize the above facts into the following lemma.
\begin{lem}
The nonnegative function $u\in \ell^p(\Z)$ is in $\mathscr{U}(d,Q)$ if, and only if, $x= u^{\frac{1}{d}}\in \ell^{d+1}(\Z)$ is a not identically vanishing zero of the map $\mathscr{F}$. Moreover, $u\in \mathscr{U}(d,Q)$ is non-degenerate if, and only if, $x$ is a non-degenerate zero of $\mathscr{F}$, meaning that the operator  $\diff \mathscr{F}(x)$ is invertible.
\end{lem}

Proposition \ref{p:superpos} will be deduced from the next result about the existence of \textit{multibump solutions} of equation $\mathscr{F}(x)=0$, once two non-degenerate solutions are given. Here, the $\Z$-action on $\ell^{d+1}(\Z)$ is given by $\tau^n_*(x) (i) = x(i-n)$.

\begin{thm}
\label{t:multibump}
Let $x_1$ and $x_2$ be not necessarily distinct zeroes of $\mathscr{F}$ in $\ell^{d+1}(\Z) \setminus \{0\}$. If $x_1$ and $x_2$ are non-degenerate, then there exists $N\in \N$ such that for every $n\in \Z$ with $|n|\geq N$ there exists a non-degenerate zero $z_n \in \ell^{d+1}(\Z) \setminus \{0\}$ 
of $\mathscr{F}$ such that
\[
\lim_{n \rightarrow \pm \infty} \| z_n - (x_1 + \tau_*^n x_2) \|_{d+1} = 0.
\]
\end{thm} 

The above theorem is proven in the next section. Proposition \ref{p:superpos} and Theorem \ref{t:superpos} can be deduced from it by using the identities \eqref{e:pi&P} and \eqref{e:defnw} relating the single-site marginal, transition matrix, and weight of a Gibbs measure in $\mathscr{G}(d,Q)$ to the corresponding function $u=x^d\in \mathscr{U}(d,Q)$, where $x\in \ell^{d+1}(\Z) \setminus \{0\}$ is the corresponding zero of $\mathscr{F}$. For the reader's convenience, we conclude this section by presenting proofs of these derivations.

\begin{proof}[Proof of Proposition \ref{p:superpos}]
Denote by $u_1=x^d_1$ and $u_2=x^d_2$ the elements of $\mathscr{U}(d,Q)$ which correspond to the Gibbs measures $\mu$ and $\mu'$ in the sense of Proposition \ref{p:zachary}. Here, $x_1$ and $x_2$ are positive solutions of the equation
\begin{equation}
\label{e:laeq}
x = Q*x^d 
\end{equation}
in $\ell^{d+1}(\Z)$ and from  \eqref{e:pi&P} and \eqref{e:defnw} we find, setting $q:= \|Q\|_{\frac{d+1}{2}}^{\frac{d+1}{d-1}}$,
\[
\begin{split}
& w = q \|x_1\|_{d+1}^{d+1}, \quad w' = q \|x_2\|_{d+1}^{d+1},\quad 
\pi(i) = q \frac{x_1(i)^{d+1}}{w}, \quad  \pi'(i) = q \frac{x_2(i)^{d+1}}{w'},  \\
& P(i,j) = \frac{Q(i-j) x_1(j)^d}{x_1(i)}, \qquad P'(i,j) = \frac{Q(i-j) x_2(j)^d}{x_2(i)}. 
\end{split}
\]
Then $x_1$ and $x_2$ are non-degenerate zeroes of $\mathscr{F}$ in $\ell^{d+1}(\Z)\setminus \{0\}$ and Theorem \ref{t:multibump} gives us a number $N\in \N$ such that for every $n\in \N$ with $|n|\geq N$ the map $\mathscr{F}$ has a non-degenerate zero $z_n \in \ell^{d+1}(\Z) \setminus \{0\}$  such that
\begin{equation}
\label{e:conve}
\lim_{n \rightarrow \pm \infty} \| z_n - (x_1 + \tau_*^n x_2) \|_{d+1} = 0.
\end{equation}
For $|n|\geq N$, we denote by $\mu\#_n \mu'$ the element of $\mathscr{G}(d,Q)$ corresponding to $z_n$ and by $\pi_n$, $P_n$, and $w_n$ its single-site marginal, transition matrix, and weight:
\[
w_n = q \|z_n\|_{d+1}^{d+1}, \qquad \pi_n(i) = q \frac{z_n(i)^{d+1}}{w_n}, \qquad P_n(i,j) = \frac{Q(i-j) z_n(j)^d}{z_n(i)}.
\]
Since the map from $\ell^{d+1}(\Z)$ to $\ell^1(\Z)$ mapping $x$ to $x^{d+1}$ is Lipschitz continuous on bounded sets, \eqref{e:conve} implies
\begin{equation}
\label{e:conve1}
\lim_{n\rightarrow \pm \infty} \| z_n^{d+1} - (x_1 + \tau_*^n x_2)^{d+1} \|_1 = 0.
\end{equation}
By approximating $x_1$ and $x_2$ with functions with bounded support, we find
\[
\lim_{n\rightarrow \pm \infty} \| (x_1 + \tau_*^n x_2)^{d+1} - (x_1^{d+1} + ( \tau_*^n x_2)^{d+1}) \|_1=0,
\]
and hence \eqref{e:conve1} implies
\begin{equation}
\label{e:conve2}
\lim_{n\rightarrow \pm \infty} \| z_n^{d+1} - (x_1^{d+1} + (\tau_*^n x_2)^{d+1}) \|_1 = 0.
\end{equation}
Using also the inequality
\[
\bigl| \|z_n^{d+1}\|_1 - \|x_1^{d+1} + (\tau_*^n x_2)^{d+1}\|_1 \bigr| \leq \| z_n^{d+1} - (x_1^{d+1} + (\tau_*^n x_2)^{d+1}) \|_1,
\]
and 
\[
\begin{split}
\lim_{n\rightarrow \pm \infty} \|x_1^{d+1} + (\tau_*^n x_2)^{d+1}\|_1 &= \lim_{n\rightarrow \infty}  ( \|x_1^{d+1}\|_1 + \|(\tau_*^n x_2)^{d+1}\|_1 ) \\ &=   \|x_1^{d+1}\|_1 + \|x_2^{d+1}\|_1 = \frac{w + w'}{q},
\end{split}
\]
we deduce that
\begin{equation}
\label{e:convew}
\lim_{n\rightarrow \pm \infty} w_n =\lim_{n\rightarrow  \pm \infty} q  \|z_n^{d+1}\|_1 = w + w',
\end{equation}
proving \eqref{e:asy0}. By dividing \eqref{e:conve2} by $w+w'$, we find
\[
\lim_{n\rightarrow \pm \infty} \Bigl\| \frac{w_n}{w+w'} \frac{z_n^{d+1}}{w_n} - \Bigl( \frac{w}{w+w'}\frac{x_1^{d+1}}{w} + \frac{w'}{w+w'}  \frac{(\tau_*^n x_2)^{d+1}}{w'}\Bigr) \Bigr\|_1 = 0,
\]
which, together with \eqref{e:convew}, implies
\[
\lim_{n\rightarrow \pm \infty} \Bigl\| \pi_n -  \Bigl( \frac{w}{w+w'}\, \pi + \frac{w'}{w+w'} \, \tau_*^n \pi' \Bigr) \Bigr\|_1 = 0,
\]
proving \eqref{e:asy1}. 

There remains to prove the statement about the asymptotics of $P_n$ in \eqref{e:asy2}. Note that the coefficients $\lambda_n(i)$ defined in \eqref{e:lambda} have the form
\[
\lambda_n(i) = \frac{x_1(i)}{x_1(i) + \tau_*^n x_2 (i)},
\]
and hence the transition matrix $\tilde{P}_n$ defined in \eqref{e:P'} takes the form
\[
\tilde{P}_n(i,j) = \frac{Q(i-j) \bigl( x_1(j)^d + \tau_*^n x_2(j)^d\bigr)}{y_n (i)}, \qquad \mbox{where} \quad y_n(i) := x_1(i) + \tau_*^n x_2 (i).
\] 
From the identities
\[
\begin{split}
\pi_n(i) P_n(i,j) &= \frac{q}{w_n} z_n(i)^d Q(i-j) z_n(j)^d, \\ \pi_n(i) \tilde{P}_n(i,j) &=  \frac{q}{w_n} \frac{z_n(i)^{d+1}}{y_n(i)} Q(i-j)  \bigl( x_1(j)^d + \tau_*^n x_2(j)^d\bigr),
\end{split}
\]
we find
\[
\begin{split}
\sum_{(i,j)\in \Z^2} & \pi_n(i) | P_n(i,j) - \tilde{P}_n(i,j) | = \\
& \frac{q}{w_n} \sum_{(i,j)\in \Z^2} z_n(i)^d Q(i-j) \Bigl| z_n(j)^d -  \frac{z_n(i)}{y_n(i)} (x_1(j)^d + \tau_*^n x_2(j)^d) \Bigr| \leq \frac{q}{w_n} ( A_n  + B_n  ),
\end{split}
\]
where
\[
\begin{split}
A_n &:=  \sum_{(i,j)\in \Z^2} z_n(i)^d Q(i-j) \Bigl| z_n(j)^d -  (x_1(j)^d + \tau_*^n x_2(j)^d) \Bigr|, \\
B_n &:= \sum_{(i,j)\in \Z^2} z_n(i)^d Q(i-j) \Bigl| x_1(j)^d + \tau_*^n x_2(j)^d-  \frac{z_n(i)}{y_n(i)} (x_1(j)^d + \tau_*^n x_2(j)^d) \Bigr|,
\end{split}
\]
and it is enough to prove that both sequences $(A_n)$ and $(B_n)$ are infinitesimal. The first one can be rewritten as 
\[
A_n = \sum_{i\in \Z} z_n(i)^d (Q* v_n ) (i), \qquad \mbox{where} \quad v_n := |z_n^d - (x_1^d + (\tau_*^n x_2)^d)|,
\]
and by the H\"older and Young inequalities we find
\begin{equation}
\label{e:A}
A_n \leq \|z_n\|_{d+1}^d \, \|Q * v_n\|_{d+1} \leq \|z_n\|_{d+1}^d\, \|Q\|_{\frac{d+1}{2}} \,\|v_n\|_{\frac{d+1}{d}}
\end{equation}
By \eqref{e:conve} and the fact that the map from $\ell^{d+1}(\Z)$ to $\ell^{\frac{d+1}{d}}(\Z)$ mapping $x$ to $x^d$ is Lipschitz continuous on bounded sets we deduce that
\[
\lim_{n\rightarrow \pm \infty} \| z_n^d - (x_1 + \tau_*^n x_2)^d \|_{\frac{d+1}{d}} = 0.
\]
Together with
\[
\lim_{n\rightarrow \pm \infty} \|x_1^d + (\tau_*^n x_2)^d - (x_1 + \tau_*^n x_2)^d \|_{\frac{d+1}{d}} = 0,
\]
this implies that $(v_n)$ tends to zero in $\ell^{\frac{d+1}{d}}$, and hence \eqref{e:A} implies that $(A_n)$ is infinitesimal. The sequence $(B_n)$ can be rewritten as
\[
\begin{split}
B_n &= \sum_{i\in \Z} z_n(i)^d \Bigl| 1 - \frac{z_n(i)}{y_n(i)} \Bigr| (Q * (x_1^d + (\tau_*^n x_2)^d))(i) \\ &= \sum_{i\in \Z} z_n(i)^d \Bigl| 1 - \frac{z_n(i)}{y_n(i)} \Bigr|( x_1(i) + \tau_*^n x_2 (i)) = \sum_{i\in \Z} z_n(i)^d |y_n(i) - z_n(i)| ,
\end{split}
\]
where we used the fact that $x_1$ and $\tau_*^n x_2$ solve \eqref{e:laeq}. By H\"older's inequality
\[
B_n \leq \|z_n\|_{d+1}^d \|y_n - z_n\|_{d+1},
\] 
and from \eqref{e:conve} we conclude that $(B_n)$ is infinitesimal. This concludes the proof of Proposition \ref{p:superpos}.
\end{proof}

We now deduce Theorem \ref{t:superpos} from Proposition \ref{p:superpos}. We first need a lemma about Markov measures on finite trees. Let $(\Lambda,E_{\Lambda})$ be a finite directed tree with root $o$ and edges oriented away from the root. A probability measure $\rho$ on $\Z$ and a stochastic matrix $K$ indexed on $\Z^2$ determine the Markov measure $\nu_{\rho,K}$ on $\Z^{\Lambda}$ defined by
\begin{equation}
\label{e:defnu}
\nu_{\rho,K}(\eta) := \rho(\eta_o) \prod_{(x,y)\in E_{\Lambda}} K(\eta_x,\eta_y) \qquad \forall \eta\in \Z^{\Lambda}.
\end{equation}
If $\mu$ is a homogeneous Markov measure on $(\Omega,\mathscr{A})$ with single-site marginal $\pi$ and transition matrix $P$,
and $\Lambda$ is a connected finite subset of $V$, then $\nu_{\pi,P}$ is precisely the marginal $\mu|_{\Lambda}$. As usual, $\rho$ is said to be $K$-invariant if
\[
\sum_{i\in \Z} \rho(i) K(i,j) = \rho(j) \qquad \forall j\in \Z.
\]
The reversibility condition
\[
\rho(i) K(i,j) = \rho(j) K(j,i) \qquad \forall (i,j) \in \Z^2,
\]
implies that $\rho$ is $K$-invariant.

\begin{lem}
\label{l:finite-trees}
Let $(\Lambda,E_{\Lambda})$ be a finite directed tree with root $o$ and edges oriented away from the root. Let $\rho$ and $\rho'$ be probability measures on $\Z$, $K$ and $K'$ stochastic matrices indexed on $\Z^2$. Then
\begin{eqnarray}
\label{e:fitre1}
& \displaystyle{\| \nu_{\rho,K} - \nu_{\rho',K} \|_1 = \|\rho - \rho' \|_1.} &
\end{eqnarray}
Moreover, if $\rho$ is $K$-invariant then
 \begin{eqnarray}
\label{e:fitre2}
 & \displaystyle{\|\nu_{\rho,K} - \nu_{\rho,K'} \|_1 \leq (|\Lambda |-1) \sum_{(i,j)\in \Z^2} \rho(i) | K(i,j) - K'(i,j) |.} &
\end{eqnarray}
\end{lem}

\begin{proof}
Since \eqref{e:defnu} is linear in $\rho$, we have
\[
\| \nu_{\rho,K} - \nu_{\rho',K} \|_1 = \sum_{i\in \Z} |\rho(i) - \rho'(i)| \sum_{\substack{\eta\in \Z^{\Lambda} \\ \eta_o = i}}  \prod_{(x,y)\in E_{\Lambda}} K(\eta_x,\eta_y).
\]
Since $K$ is stochastic, summing successively over the spins at the leaves of the tree $(\Lambda,E_{\Lambda})$ and canceling them from it shows that the above rightmost  sum has the value 1, proving \eqref{e:fitre1}.

Recall that the tail and head of an edge $(x,y)\in E_{\Lambda}$  are the vertices $x$ and $y$, respectively. Label the edges in $E_{\Lambda}$ as $e_1,\dots,e_m$, where $m=|\Lambda|-1$, ordering them in such a way that the distance of the head of $e_j$ from $o$ is non-decreasing in $j$. For $h$ in $\{0,1,\dots,m\}$ let $\nu_h$ be the measure on $\Z^{\Lambda}$ defined as in \eqref{e:defnu} but with the term $K$ replaced by $K'$ on the edges $e_{h+1},\dots,e_m$, so that $\nu_m=\nu_{\rho,K}$ and $\nu_0=\nu_{\rho,K'}$. Fix some $h$ in $\{1,\dots,m\}$. The measures $\nu_h$ and $\nu_{h-1}$ differ only in the factor corresponding to the edge $e_h=(x,y)$. Therefore
\[
\bigl|\nu_h(\eta) - \nu_{h-1}(\eta)\bigr| = \rho(\eta_o) A(\eta) \bigl| K(\eta_x,\eta_y) -  K'(\eta_x,\eta_y) \bigr| B(\eta),
\]
where $A(\eta)$ is the product of the factors $K$ over $e_1,\dots,e_{h-1}$ and $B(\eta)$ is the product of the factors $K'$ over $e_{h+1},\dots,e_m$. Since the head of every $e_j$ is not the tail of any edge $e_k$ with $k\leq j$, summing iteratively over the values of the spin at the head of $e_m,\dots,e_{h+1}$ and using the fact that $K'$ is stochastic, the terms $B(\eta)$ add up to 1. After this, $\eta_y$ appears only in the term $|K(\eta_x,\eta_y) - K'(\eta_x,\eta_y)|$, and summing over  $\eta_y$ gives
\[
\rho(\eta_o) A(\eta) \sum_{j\in \Z} \bigl| K(\eta_x,j) - K'(\eta_x,j) \bigr|.
\]
Denoting by $(\Lambda_h,E_{\Lambda_h})$ the subtree of $(\Lambda,E_{\Lambda})$ whose set of vertices consists of $o$ and the heads of the edges $e_1,\dots,e_{h-1}$, we see that $\rho(\eta_o) A(\eta)=\nu_{\rho,K}|_{\Lambda_h}(\eta|_{\Lambda_h})$. Therefore,
\[
\|\nu_h - \nu_{h-1} \|_1 = \sum_{(i,j)\in \Z^2} f_h(i) \bigl| K(i,j) - K'(i,j) \bigr|,
\]
where 
\[
f_h(i) := \sum_{\substack{\eta\in \Z^{\Lambda_h} \\ \eta_x = i}} \nu_{\rho,K}|_{\Lambda_h} (\eta)
\]
is the law of $\eta\mapsto \eta_x$ under $\nu_{\rho,K}|_{\Lambda_h}$. Since $\rho$ is $K$-invariant, an induction argument on $h$ shows that $f_h=\rho$. Therefore,
\[
\|\nu_h - \nu_{h-1} \|_1 = \sum_{(i,j)\in \Z^2} \rho(i) \bigl| K(i,j) - K'(i,j) \bigr|,
\]
and the triangle inequality yields \eqref{e:fitre2}.
\end{proof}

\begin{proof}[Proof of Theorem \ref{t:superpos}] Let $\mu\#_n \mu'$ for $|n|\geq N$ be the Gibbs measure given by Proposition \ref{p:superpos}, and denote by $w_n$, $\pi_n$, and $P_n$ its weight, single-site marginal, and transition matrix.  The convergence of $(w_n)$ to $w+w'$ is given by \eqref{e:asy0}. We must show that \eqref{e:asy1} and \eqref{e:asy2} imply  
\begin{equation}
\label{e:totvar}
\lim_{n\rightarrow \pm \infty} \bigl\| (\mu \#_n \mu')|_{\Lambda} - \bigl( \textstyle{\frac{w}{w+w'}} \, \mu|_{\Lambda} + \textstyle{\frac{w'}{w+w'}}\, \tau_*^n \mu'|_{\Lambda} \bigr) \bigr\|_1 = 0,
\end{equation}
for every finite subset $\Lambda\subset V$. Up to replacing $\Lambda$ by a larger finite subset, we may assume that it is connected. Choose a vertex $o$ in $\Lambda$ and denote by $(\Lambda,E_{\Lambda})$ the directed tree with root $o$ and edges oriented away from $o$ which is induced by $(V,E)$. Set
\[
\alpha:= \frac{w}{w+w'}, \qquad \pi_n'(i) := \pi'(i-n), \qquad P_n'(i,j) := P'(i-n,j-n),
\]
so that, with the notation introduced in \eqref{e:defnu} and relative to the directed tree $(\Lambda,E_{\Lambda})$, we have
\[
\mu|_{\Lambda} = \nu_{\pi,P}, \qquad (\tau_*^n \mu')|_{\Lambda} = \nu_{\pi_n',P_n'}, \qquad (\mu \#_n \mu')|_{\Lambda} = \nu_{\pi_n,P_n}.
\]
By the reversibility identity \eqref{e:reversibility}, $\pi$, $\pi_n'$, and $\pi_n$ are invariant for $P$, $P_n'$, and $P_n$, respectively. Using the fact that $\nu_{\rho,K}$ is linear in $\rho$, we can bound the quantity in \eqref{e:totvar} by
\[
\scalebox{.94}{$
\| \nu_{\pi_n,P_n} - \nu_{\pi_n,\tilde{P}_n} \|_1 + \| \nu_{\pi_n,\tilde{P}_n} - \nu_{\alpha \pi + (1-\alpha) \pi_n',\tilde{P}_n} \|_1 
+ \alpha\, \| \nu_{\pi,\tilde{P}_n} - \nu_{\pi,P} \|_1 + (1-\alpha)\, \| \nu_{\pi_n',\tilde{P}_n} - \nu_{\pi_n',P_n'} \|_1 .$}
\]
By Lemma \ref{l:finite-trees}, the first term tends to $0$ by \eqref{e:asy2}, and the second one by \eqref{e:asy1}. By \eqref{e:P'} and Lemma \ref{l:finite-trees}, the third term does not exceed
\[
\alpha (|\Lambda|-1) \sum_{(i,j)\in \Z^2} \pi(i) (1-\lambda_n(i)) |P(i,j)-P_n'(i,j)| \leq 2 (|\Lambda|-1) \sum_{i\in \Z} \pi(i) (1-\lambda_n(i)),
\]
and the last sum tends to $0$ by dominated convergence, because 
\[
0\leq 1-\lambda_n(i) \leq \Bigl( \frac{w'\pi'(i-n)}{ w\,\pi(i)} \Bigr)^{\frac{1}{d+1}} \rightarrow 0 \qquad \forall i\in \Z.
\]
Similarly, the fourth term does not exceed 
\[
2(|\Lambda|-1) \sum_{i\in \Z} \pi'(i) \lambda_n(i+n),
\]
which tends to $0$ because 
\[
0\leq \lambda_n(i+n) \leq \Bigl( \frac{w\, \pi(i+n)}{w' \pi'(i)} \Bigr)^{\frac{1}{d+1}}.
\]
This concludes the proof of \eqref{e:totvar}.
\end{proof}

\section{Existence of multibump solutions}
\label{s:embs}

In this section, we prove Theorem \ref{t:multibump}. The proof is an application of Newton's method, which we use in the following form, see \cite[Theorem 15.6]{dei85}. Here, $\mathrm{L}(X,Y)$ denotes the space of bounded linear operators from the Banach space $X$ to the Banach space $Y$, and $\|\cdot\|_{\mathrm{op}}$ denotes the operator norm.

\begin{prop}
\label{p:newton}
Let $X$ and $Y$ be Banach spaces, let $B_r(x_0)\subset X$ be the open ball of radius $r>0$ centered at $x_0\in X$ and let
\[
F: B_r(x_0) \rightarrow Y
\]
be a continuously differentiable map. Assume that there are constants $\alpha,\beta,\kappa>0$ satisfying:
\begin{enumerate}[(i)]
\item $\diff F(x_0)$ is an isomorphism, $\|\diff F(x_0)^{-1} F(x_0)\|_X \leq \alpha$ and $\|\diff F(x_0)^{-1}\|_{\mathrm{op}} \leq \beta$;
\item $\diff F: B_r(x_0) \rightarrow \mathrm{L}(X,Y)$ is $\kappa$-Lipschitz continuous;
\item $2\kappa \alpha \beta<1$ and $2\alpha<r$.
\end{enumerate}
Then $F$ has a unique zero $\overline{x}$ in $B_{\min\{\frac{1}{\kappa \beta},r\}}(x_0)$. Moreover, this zero is non-degenerate and
\[
\|\overline{x}-x_0\|_X \leq 2\alpha.
\]
\end{prop}

We shall also use the Neumann series criterion for invertibility in the following form: if the bounded linear operator $T$ on a Banach space $X$ is invertible with $\|T^{-1}\|_{\mathrm{op}} \leq c$, then any bounded linear operator $S$ on $X$ with $\|S-T\|_{\mathrm{op}} \leq \frac{1}{2c}$ is invertible with $\|S^{-1}\|_{\mathrm{op}} \leq 2c$. 

\medskip

We start the proof of Theorem  \ref{t:multibump} by collecting some useful bounds.
From \eqref{e:bd1} we obtain that the map $\diff \mathscr{K}_Q$ is bounded on bounded sets with
\begin{equation}
\label{e:bd-K}
\|\diff \mathscr{K}_Q(x) \|_{\mathrm{op}} \leq d \|Q\|_{\frac{d+1}{2}} r^{d-1}, \qquad \mbox{if } \|x\|_{d+1} \leq r.
\end{equation}
By the mean value theorem, the above inequality and \eqref{e:bd2} imply that $\mathscr{F}_Q$ and $\diff \mathscr{F}_Q$ are Lipschitz-continuous on bounded sets with
\begin{eqnarray}
\label{e:Flip}
\|\mathscr{F}_Q(x) -  \mathscr{F}_Q(y)\|_{d+1} \leq \bigl( 1 + d \|Q\|_{\frac{d+1}{2}} r^{d-1} \bigr) \|x-y\|_{d+1}, \\
\label{e:dFlip}
\|\diff \mathscr{F}_Q(x) - \diff \mathscr{F}_Q(y) \|_{\mathrm{op}} \leq d \bigl(d-1) \|Q\|_{\frac{d+1}{2}} r^{d-2} \bigr) \|x-y\|_{d+1},
\end{eqnarray}
for all $x,y\in \ell^{d+1}(\Z)$ with $\|x\|_{d+1} \leq r$, $\|y\|_{d+1}\leq r$. 

Let $x_1$ and $x_2$ be non-degenerate zeroes of $\mathscr{F}_Q$ in $\ell^{d+1}(\Z)\setminus \{0\}$. Recall that every zero of $\mathscr{F}_Q$ is nonnegative, because $Q\geq 0$. Fix numbers $R>0$ and $C\geq \frac{1}{2}$ such that
\[
\|x_1\|_{d+1} \leq R, \qquad \|x_2\|_{d+1} \leq R, \qquad \| \diff \mathscr{F}_Q(x_1)^{-1} \|_{\mathrm{op}} \leq C, \qquad \| \diff \mathscr{F}_Q(x_2)^{-1} \|_{\mathrm{op}} \leq C.
\]
Let $q\in \N$ be so large that the function $\tilde{Q}:= Q \mathbbm{1}_{[-q,q]}$ satisfies 
\begin{equation}
\label{e:bdQ}
\|Q - \tilde{Q} \|_{\frac{d+1}{2}} \leq \delta := \frac{1}{2^{d+1} d C R^{d-1}}.
\end{equation}
Let $\epsilon$ be a positive number. Let $m=m(\epsilon)\in \N$ be such that, setting $\tilde{x}_1:= x_1 \mathbbm{1}_{[-m,m]}$ and $\tilde{x}_2:= x_2 \mathbbm{1}_{[-m,m]}$, we have
\begin{equation}
\label{e:bdx12}
\|x_1 - \tilde{x}_1 \|_{d+1} \leq \epsilon, \qquad \|x_2 - \tilde{x}_2 \|_{d+1} \leq \epsilon.
\end{equation}
By \eqref{e:bd-K} and \eqref{e:dFlip} we obtain for $j\in \{1,2\}$
\[
\begin{split}
\| \diff \mathscr{F}_{\tilde{Q}} (\tilde{x}_j) - \diff \mathscr{F}_Q(x_j) \|_{\mathrm{op}} &\leq \| \diff \mathscr{F}_{\tilde{Q}} (\tilde{x}_j) - \diff \mathscr{F}_Q (\tilde{x}_j) \|_{\mathrm{op}} +  \| \diff \mathscr{F}_{Q} (\tilde{x}_j) - \diff \mathscr{F}_Q (x_j) \|_{\mathrm{op}} \\ 
&= \| \diff \mathscr{K}_{Q-\tilde{Q}} (\tilde{x}_j) 
\|_{\mathrm{op}} + \| \diff \mathscr{F}_{Q} (\tilde{x}_j) - \diff \mathscr{F}_Q(x_j) \|_{\mathrm{op}} \\
&\leq d \|Q- \tilde{Q}\|_{\frac{d+1}{2}} R^{d-1} + d(d-1) \| Q\|_{\frac{d+1}{2}} R^{d-2} \|x_j - \tilde{x}_j \|_{d+1} \\ 
&\leq d \delta R^{d-1} + d (d-1) \|Q\|_{\frac{d+1}{2}} R^{d-2} \epsilon \\ 
&= \frac{1}{2^{d+1} C} + d (d-1) \|Q\|_{\frac{d+1}{2}} R^{d-2} \epsilon,
\end{split}
\]
where we also used \eqref{e:bdQ}, \eqref{e:bdx12}, and the definition of $\delta$. Assuming that
\begin{equation}
\label{e:epsilon}
\epsilon \leq \frac{2^d-1}{2^{d+1} d (d-1) C \|Q\|_{\frac{d+1}{2}} R^{d-2}},
\end{equation}
we then have
\[
\| \diff \mathscr{F}_{\tilde{Q}} (\tilde{x}_j) - \diff \mathscr{F}_Q(x_j) \|_{\mathrm{op}} \leq \frac{1}{2C},
\]
and by the Neumann series criterion for invertibility we deduce that the operators $\diff \mathscr{F}_{\tilde{Q}}(\tilde{x}_1)$ and $\diff \mathscr{F}_{\tilde{Q}}(\tilde{x}_2)$ are invertible with
\begin{equation}
\label{e:inversa}
\| \diff \mathscr{F}_{\tilde{Q}}(\tilde{x}_1)^{-1} \|_{\mathrm{op}} \leq 2C, \qquad \| \diff \mathscr{F}_{\tilde{Q}}(\tilde{x}_2)^{-1} \|_{\mathrm{op}} \leq 2C.
\end{equation}
The next lemma builds on \eqref{e:inversa}, together with the fact that $\tilde{Q}$ is supported in $[-q,q]$ while $\tilde{x}_1$ and $\tilde{x}_2$ are supported in $[-m,m]$.

\begin{lem}
If $|n|> 2q + 2m$ then the operator $\diff \mathscr{F}_{\tilde{Q}} (\tilde{x}_1 + \tau_*^n \tilde{x}_2)$ is invertible with
\begin{equation}
\label{e:new-inve}
\| \diff \mathscr{F}_{\tilde{Q}} (\tilde{x}_1 + \tau_*^n \tilde{x}_2)^{-1} \|_{\mathrm{op}} \leq 2C.
\end{equation}
\end{lem}

\begin{proof}
We deal with the case $n>2q +2m$, the case of $n<-2q-2m$ being analogous. Fix some $z\in \ell^{d+1}(\Z)$. We wish to find $y\in \ell^{d+1}(\Z)$ such that $\diff \mathscr{F}_{\tilde{Q}} ( \tilde{x}_1 + \tau_*^n \tilde{x}_2 )[y] = z$. By our assumption on $n$, the functions 
\[
z_1:= z \mathbbm{1}_{(-\infty,q+m]}, \qquad z_2:= z \mathbbm{1}_{[n-q-m,+\infty)},
\]
have disjoint support and we can write
\[
z = z_0 + z_1 + z_2,
\]
where $z_0$ is supported in $[q+m+1,n-q-m-1]$. From \eqref{e:diff1} we find the expression
\[
\di \mathscr{F}_{\tilde{Q}} (x)[y] = y - d \, \tilde{Q} * (x^{d-1} y), \qquad  \forall y\in \ell^{d+1}(\Z),
\]
for every nonnegative $x\in \ell^{d+1}(\Z)$. Since $\diff \mathscr{F}_{\tilde{Q}}(\tilde{x}_1)$ is invertible and its inverse has the bound \eqref{e:inversa}, we can find $y_1\in \ell^{d+1}(\Z)$ such that
\[
y_1 - d\, \tilde{Q} * ( \tilde{x}_1^{d-1} y_1) = z_1, \qquad \|y_1\|_{d+1} \leq 2C\, \|z_1\|_{d+1}.
\]
Since the supports of $\tilde{Q} *( \tilde{x}_1^{d-1} y_1)$ and $z_1$ are contained in $(-\infty,q+m]$, the same is true for the support of $y_1$.

From \eqref{e:inversa} and the $\Z$-equivariance of $\mathscr{F}_{\tilde{Q}}$, we deduce that $\diff \mathscr{F}_{\tilde{Q}} (\tau_*^n \tilde{x}_2)$ is invertible and the operator norm of its inverse does not exceed $2C$. Therefore, we can find $y_2\in \ell^{d+1}(\Z)$ such that
\[
y_2 - d\, \tilde{Q} * \bigl( (\tau_*^n \tilde{x}_2)^{d-1} y_2 \bigr) = z_2, \qquad \|y_2\|_{d+1} \leq 2C\, \|z_2\|_{d+1}.
\]
Since the supports of $\tilde{Q} * ( (\tau_*^n \tilde{x}_2)^{d-1} y_2)$ and $z_2$ are contained in $[n-q-m,+\infty)$, the same is true for the support of $y_2$. Now set
\[
y:= y_1 + y_2 + z_0.
\]
The fact that $\tilde{x}_1$ and $\tau_*^n \tilde{x}_2$ have disjoint supports implies the identity
\[
(\tilde{x}_1+ \tau_*^n \tilde{x}_2)^{d-1} = \tilde{x}_1^{d-1} + (\tau_*^n \tilde{x}_2)^{d-1}.
\]
Since $n>2q+2m$, the support of $y_2+z_0$ is contained in $[q+m+1,\infty)$ and is then disjoint from the support of 
$\tilde{x}_1$. Similarly, the support of $y_1+z_0$ is disjoint from the support of $\tau_*^n\tilde{x}_2$. Therefore, we have the identities
\[
\tilde{x}_1^{d-1} y = \tilde{x}_1^{d-1} y_1, \qquad  (\tau_*^n \tilde{x}_2)^{d-1} y =  (\tau_*^n \tilde{x}_2)^{d-1} y_2.
\]
The above three identities imply
\[
\begin{split}
\diff \mathscr{F}_{\tilde{Q}} ( \tilde{x}_1 + \tau_*^n \tilde{x}_2 ) [y] &= y - d\, \tilde{Q}* \bigl( ( \tilde{x}_1 + \tau_*^n \tilde{x}_2 )^{d-1} y \bigr) \\ &= y_1 + y_2 + z_0 - d \, \tilde{Q}* ( \tilde{x}_1^{d-1} y_1) - d \, \tilde{Q}* \bigl( (\tau_*^n \tilde{x}_2)^{d-1} y_2\bigr) \\ &= z_1 + z_2 + z_0 = z.
\end{split}
\]
This shows that $\diff \mathscr{F}_{\tilde{Q}} ( \tilde{x}_1 + \tau_*^n \tilde{x}_2 )$ is surjective and hence invertible, being Fredholm of index zero. Since the supports of $y_1$, $y_2$, and $z_0$ are pairwise disjoint, we have
\[
\begin{split}
\|y\|_{d+1}^{d+1} &= \|y_1\|_{d+1}^{d+1} + \|y_2\|_{d+1}^{d+1} + \|z_0\|_{d+1}^{d+1} \leq (2C)^{d+1} \|z_1\|_{d+1}^{d+1} + (2C)^{d+1} \|z_2\|_{d+1}^{d+1} + \|z_0\|_{d+1}^{d+1}\\ &\leq (2C)^{d+1} \|z\|_{d+1}^{d+1},
\end{split}
\]
where we have also used the fact that the supports of $z_0$, $z_1$, and $z_2$ are pairwise disjoint and the assumption $C\geq \frac{1}{2}$. This proves the bound
\[
\| \diff \mathscr{F}_{\tilde{Q}} ( \tilde{x}_1 + \tau_*^n \tilde{x}_2 )^{-1} \|_{\mathrm{op}} \leq 2C.
\]
\end{proof}

Assume that $|n|> 2q + 2m$, so that the conclusion of the above lemma holds. Since $\| \tilde{x}_1 + \tau_*^n \tilde{x}_2\|_{d+1} \leq 2R$, the bounds \eqref{e:bd-K}, \eqref{e:bdQ}, and our definition of $\delta$ imply
\[
\begin{split}
\| \diff \mathscr{F}_Q (\tilde{x}_1 + \tau_*^n \tilde{x}_2) - \diff \mathscr{F}_{\tilde{Q}} (\tilde{x}_1 + \tau_*^n \tilde{x}_2) \|_{\mathrm{op}} &= \| \diff \mathscr{K}_{Q-\tilde{Q}}  (\tilde{x}_1 + \tau_*^n \tilde{x}_2) \|_{\mathrm{op}} \\ &\leq d \|Q-\tilde{Q}\|_{\frac{d+1}{2}} (2R)^{d-1} \leq  d \delta (2R)^{d-1} = \frac{1}{4C}.
\end{split}
\]
By \eqref{e:new-inve}, the Neumann series criterion for invertibility implies that $\diff \mathscr{F}_Q (\tilde{x}_1 + \tau_*^n \tilde{x}_2)$ is invertible with
\begin{equation}
\label{e:last-inve}
\| \diff \mathscr{F}_Q (\tilde{x}_1 + \tau_*^n \tilde{x}_2)^{-1} \|_{\mathrm{op}} \leq 4C.
\end{equation}

Since $\tilde{x}_1$ and $\tau_*^n \tilde{x}_2$ are nonnegative and have disjoint supports, we have the identity
\[
\begin{split}
\mathscr{F}_Q(\tilde{x}_1 + \tau_*^n \tilde{x}_2) &= \tilde{x}_1 + \tau_*^n \tilde{x}_2 - Q * (\tilde{x}_1 + \tau_*^n \tilde{x}_2)^d =  \tilde{x}_1 + \tau_*^n \tilde{x}_2 - Q * \tilde{x}_1^d  - Q* ( \tau_*^n \tilde{x}_2)^d \\ &= \mathscr{F}_Q(\tilde{x}_1) +  \mathscr{F}_Q(\tau_*^n \tilde{x}_2) =  \mathscr{F}_Q(\tilde{x}_1) +  \tau_*^n \mathscr{F}_Q(\tilde{x}_2) \\ &=  \mathscr{F}_Q(\tilde{x}_1) -  \mathscr{F}_Q(x_1) + \tau_*^n \bigl( \mathscr{F}_Q(\tilde{x}_2) -  \mathscr{F}_Q(x_2)\bigr),
\end{split}
\]
where we used the $\Z$-equivariance of $\mathscr{F}_Q$ and the fact that $x_1$ and $x_2$ are zeroes of this map. By \eqref{e:Flip} and \eqref{e:bdx12} we then have
\[
\begin{split}
\| \mathscr{F}_Q(\tilde{x}_1 + \tau_*^n \tilde{x}_2) \|_{d+1} &\leq \bigl(1+ d \|Q\|_{\frac{d+1}{2}} R^{d-1} \bigr) \bigl( \|x_1-\tilde{x}_1\|_{d+1} + \|x_2-\tilde{x}_2\|_{d+1} \bigr) \\ &\leq 2 \bigl(1+ d \|Q\|_{\frac{d+1}{2}} R^{d-1} \bigr) \epsilon,
\end{split}
\]
and hence by \eqref{e:last-inve}
\begin{equation}
\label{e:alphabd}
\| \diff \mathscr{F}_Q (\tilde{x}_1 + \tau_*^n \tilde{x}_2)^{-1}  \mathscr{F}_Q(\tilde{x}_1 + \tau_*^n \tilde{x}_2) \|_{d+1} \leq 8C\bigl(1+ d \|Q\|_{\frac{d+1}{2}} R^{d-1} \bigr) \epsilon.
\end{equation}

We now wish to check the assumptions of Proposition \ref{p:newton} with $X=Y=\ell^{d+1}(\Z)$, $F=\mathscr{F}_Q$,
$x_0:= \tilde{x}_1 + \tau_*^n \tilde{x}_2$, and $r:=1$. By \eqref{e:last-inve} and \eqref{e:alphabd}, assumption (i) holds with
\[
\alpha:=  \gamma \epsilon, \qquad \beta := 4C, \qquad \mbox{where} \quad \gamma:= 8C \bigl(1+ d \|Q\|_{\frac{d+1}{2}} R^{d-1} \bigr).
\]
By \eqref{e:dFlip}, assumption (ii) holds with
\[
\kappa := d (d-1) \|Q\|_{\frac{d+1}{2}} (2R+1)^{d-2}.
\]
The above expressions for $\alpha$, $\beta$, and $\kappa$ show that assumption (iii) holds provided that $\epsilon$ is small enough, and more precisely if
\begin{equation}
\label{e:epsilon2} 
\begin{split}
16 C( 1 + d \|Q\|_{\frac{d+1}{2}} R^{d-1})& \epsilon < 1, \\
64 C^2 d(d-1) \|Q\|_{\frac{d+1}{2}} (2R+1)^{d-2} ( 1 + d \|Q\|_{\frac{d+1}{2}} R^{d-1})&  \epsilon < 1.
\end{split}
\end{equation}
From Proposition \ref{p:newton} we conclude that if $\epsilon>0$ is so small that \eqref{e:epsilon} and \eqref{e:epsilon2} hold and $m=m(\epsilon)$ is such that \eqref{e:bdx12} holds, then for every $n\in \Z$ with $|n|>  2q + 2m$ the map $\mathscr{F}_Q$ has a non-degenerate zero $z_n \in \ell^{d+1}(\Z)$ such that
\[
\|z_n - (\tilde{x}_1 + \tau_*^n \tilde{x}_2) \|_{d+1} \leq 2 \gamma \epsilon,
\]
and hence, by the triangle inequality and \eqref{e:bdx12},
\[
\|z_n - (x_1 + \tau_*^n x_2) \|_{d+1} \leq 2 (\gamma +1) \epsilon.
\]
After possibly reducing $\epsilon$ even more, the above inequality implies that $z_n \neq 0$. We conclude that for $|n|$ large enough
$z_n$ is a non-degenerate zero of $\mathscr{F}_Q$ in $\ell^{d+1}(\Z) \setminus \{0\}$, and the fact that $\epsilon$ can be chosen to be arbitrarily small implies that
\[
\lim_{n\rightarrow \pm \infty} \|z_n  - (x_1 + \tau_*^n x_2) \|_{d+1} = 0.
\]
This concludes the proof of Theorem \ref{t:multibump}.

\section{Interaction functions with small support}
\label{s:ifss}

In this and in the next five sections, we consider interaction functions $Q$ supported in $\{-1,0,1\}$. In this case, the irreducibility condition forces $Q(1)>0$ and, up to multiplying $Q$ by a positive constant, an operation which does not affect the set of Gibbs measures, we may assume that $Q(1)=Q(-1)=1$. Therefore, $Q$ takes the form
\begin{equation}
\label{e:simpleQ}
Q = q \mathbbm{1}_{\{0\}} + \mathbbm{1}_{\{-1,1\}},
\end{equation}
and we have just one free parameter $q\in [0,+\infty)$. Equation \eqref{e:main-eq} now reads
\begin{equation}
\label{e:2ndrec}
u(i)  = \bigl( u(i-1) + q \,u(i) + u(i+1)\bigr)^d, \qquad \forall i\in \Z.
\end{equation}
Solving for $u(i+1)$ and setting $(x_i,y_i) = (u(i-1),u(i))$, the second order recursion \eqref{e:2ndrec} reduces to the first order recursion
\[
(x_{i+1},y_{i+1}) = \varphi(x_i,y_i), \qquad \forall i\in \Z,
\]
where $\varphi$ is the map
\[
\varphi (x,y) := (y,y^{\frac{1}{d}}-qy -x), \qquad (x,y)\in [0,+\infty)^2.
\]
It is useful to extend this map to an odd map on the whole $\R^2$ by setting
\begin{equation}
\label{e:mapphi}
\varphi (x,y) := (y,|y|^{\frac{1}{d}-1} y - q y - x).
\end{equation}
This map is a homeomorphism of the plane to itself with inverse
\begin{equation}
\label{e:inve}
\varphi^{-1}(x,y) = (|x|^{\frac{1}{d}-1} x - q x - y,x).
\end{equation}
The  restriction of $\varphi$ to the complement of the $x$-axis is an analytic diffeomorphism onto the complement of the $y$-axis. The origin is a fixed point of $\varphi$, and we define its \textit{stable} and \textit{unstable set} as
\[
\begin{split}
W^s(0) &:= \{ z\in \R^2 \mid \varphi^n(z) \rightarrow 0 \mbox{ for } n\rightarrow +\infty\}, \\
W^u(0) &:= \{ z\in \R^2 \mid \varphi^n(z) \rightarrow 0 \mbox{ for } n\rightarrow -\infty\}.
\end{split}
\]
Gibbs measures for $Q$ correspond to positive solutions $u\in \ell^{p}(\Z)$ of \eqref{e:2ndrec}. In particular, $u(i) \rightarrow 0$ for $i\rightarrow \pm \infty$, and hence the point $(u(0),u(1))$ belongs to the intersection $(W^s(0)\cap W^u(0)) \setminus \{0\}$. As usual, such an intersection point $z$ is called a \textit{homoclinic point} to the fixed point $0$ of $\varphi$, and the corresponding orbit $\{\varphi^n(z)\}_{n\in \Z}$ is called a \textit{homoclinic orbit}.  We will systematically omit to specify ``to the fixed point $0$'' and just talk about homoclinic points and homoclinic orbits of $\varphi$. Conversely, our analysis of the stable and unstable sets will show that all homoclinic orbits converge to the origin with order at least $d$, meaning that
\begin{equation}
\label{e:order-d}
|\varphi^n(z)| \leq C |\varphi^{n-1}(z)|^d, \qquad |\varphi^{-n}(z)| \leq C |\varphi^{-n+1}(z)|^d, \qquad \forall n\geq 1,
\end{equation}
for some $C$, and in particular their components define elements in $\ell^p(\Z)$, because convergence to the origin and \eqref{e:order-d} imply doubly exponential decay: there exist $\theta\in (0,1)$ and $C'>0$ such that
\[
|\varphi^n(z)| \leq C' \theta^{d^{|n|}} \qquad \forall n\in \Z.
\]
We can then summarize the above discussion into the following proposition.

\begin{prop}
\label{p:gibbs=homo}
Let $Q$ be given by \eqref{e:simpleQ}. Then
there is a one-to-one correspondence between $\mathscr{G}(d,Q)$ and the set consisting of the homoclinic points in $W^s(0)\cap W^u(0)$ whose orbit is fully contained in the positive quadrant $(0,+\infty) \times (0,+\infty)$. This descends to a one-to-one correspondence between $\mathscr{G}(d,Q)/\Z$ and the set of homoclinic orbits of $\varphi$ in $(0,+\infty) \times (0,+\infty)$.
\end{prop}

Before analysing the stable and unstable sets $W^s(0)$ and $W^u(0)$, we list some useful properties of the homeomorphism $\varphi$.

\begin{enumerate}[(a)]
\item $\varphi$ is equivariant with respect to the involution $z\mapsto -z$, i.e., $\varphi(-z) = - \varphi(z)$ for every $z\in \R^2$.
\item $\varphi$ is conjugate to $\varphi^{-1}$ by the involution 
\[
S(x,y) := (y,x),
\]
i.e., $\varphi^{-1} = S \circ \varphi \circ S$.
\item $\varphi$ is orientation and area preserving.
\end{enumerate}

Indeed, the differential of $\varphi$ at any $(x,y)$ with $y\neq 0$ has the form
\[
\diff \varphi(x,y) = \left( \begin{array}{cc} 0 & 1 \\ -1 & \frac{1}{d} |y|^{\frac{1}{d}-1} - q \end{array} \right).
\]

\begin{enumerate}[(a)]
\setcounter{enumi}{3}
\item The fixed points of $\varphi$ are $(0,0)$, $(c,c)$ and $(-c,-c)$, with
\begin{equation}
\label{c}
c:= (2+q)^{- \frac{d}{d-1}}< 1.
\end{equation}
\item The homeomorphism $\varphi$ is not differentiable at $(0,0)$, but for $(x,y)$ close to $(0,0)$ with $y\neq 0$ the linear map $\diff \varphi(x,y)$ has a very large positive trace and determinant 1. It has therefore real eigenvalues $\epsilon,\epsilon^{-1}$ with $\epsilon>0$ small, and corresponding eigendirections
\[
\R \left( \begin{array}{c} 1 \\ \epsilon \end{array} \right) \qquad \mbox{and} \qquad \R \left( \begin{array}{c} \epsilon \\ 1 \end{array} \right),
\]
which for $(x,y)\rightarrow (0,0)$ converge to the $x$-axis and $y$-axis, respectively. The origin should therefore be considered as a hyperbolic fixed point with stable direction given by the $x$-axis and unstable direction given by the $y$-axis. 
\end{enumerate}

\begin{rem}
{\rm The stability properties of the fixed points $(c,c)$ and $(-c,-c)$ depend on the value of $q$. Indeed, a simple computation shows the following. These fixed points are elliptic (eigenvalues of the differential of $\varphi$ at them with absolute value 1) for $0 \leq q\leq 2 \frac{d+1}{d-1}$. They are negatively hyperbolic (real negative eigenvalues $-\lambda,-\lambda^{-1}$ with $0<\lambda<1$) for $q> 2 \frac{d+1}{d-1}$.}
\end{rem}

The equivariance of $\varphi$ with respect to the involution $z\mapsto -z$ determines the symmetries
\[
W^s(0) = - W^s(0), \qquad W^u(0) = - W^u(0).
\]
The fact that $\varphi$ is conjugate to its inverse by the reflection $S$ implies the identity
\begin{equation}
\label{e:stab-unstab}
W^u(0)= S (W^s(0)), 
\end{equation}
so it is enough to study the stable set. Since $\varphi$ is not differentiable at the origin, the regularity properties of $W^s(0)$ do not follow from standard results.  In the next section, we investigate these properties and give a precise description of the local stable manifold of $0$ as graph of a function. More precisely, we shall prove the following result, where $c$ is the number defined in \eqref{c}.

\begin{thm}
\label{t:stable}
There exist numbers $r\in (c,1)$, $s\in (0,r)$, and a continuous odd function $f:[-r,r] \rightarrow [-s,s]$ such that:
\begin{enumerate}[(i)]
\item There exists a power series $F(x) = x^d + o(x^d)$ with nonnegative coefficients and radius of convergence $r$ such that $f|_{(-r,r)} = F$ if $d$ is odd and $f|_{[0,r)} = F$ if $d$ is even, and  
\[
f(\pm r) = \pm s, \qquad \lim_{x\rightarrow \pm r} f'(x)=+\infty, \qquad 0< f(x) < x \quad \forall x\in (0,r].
\]
\item The graph of $f$ is postively invariant with respect to $\varphi$. More precisely, 
\[
\varphi(x,f(x)) = (f(x),f(f(x))) \quad \forall x\in [-r,r] \quad \mbox{and} \quad \varphi(\Graph f) = \Graph f|_{[-s,s]}.
\]
In particular, $\Graph f\subset W^s(0)$.
\item If $\rho\in (0,c)$ is small enough, then 
\[
W^s_{\mathrm{loc},\rho}(0):= \{(x,y)\in \R^2 \mid \varphi^n(x,y)\in [-\rho,\rho]^2 \; \forall n\geq 0\} = \Graph f|_{[-\rho,\rho]},
\]
and 
\[
W^s(0) = \bigcup_{n\in \N} \varphi^{-n} \bigl( W^s_{\mathrm{loc},\rho}(0) \bigr) = \bigcup_{n\in \N} \varphi^{-n} \bigl( \mathrm{graph}\, f \bigr).
\]
\end{enumerate}
\end{thm}

The set $W^s_{\mathrm{loc},\rho}(0)$ appearing in (iii) is called \textit{local stable manifold} of size $\rho$ of the fixed point $0$. The proof of the above theorem is discussed in the next section. Here, we see how it implies that every homoclinic orbit converges to the origin with order at least $d$, thus concluding the proof of Proposition \ref{p:gibbs=homo}. Let $z\in W^s(0)\cap W^u(0)$ be a homoclinic point of $\varphi$.  Since $\varphi^n(z)$ converges to the origin for $n\rightarrow + \infty$, for every $\rho>0$ there exists $N>0$ such that $\varphi^n(z)\in [-\rho,\rho]^2$ for every $n\geq N$. By choosing $\rho$ small enough as in Theorem \ref{t:stable} (iii), we deduce that $\varphi^n(z)$ belongs to the graph of $f$ for every $n\geq N$. In particular, $\varphi^N(z) = (x,f(x))$ for some $x\in [-r,r]$, and from Theorem \ref{t:stable} (ii) we deduce that
\[
\varphi^{N+n}(z) = (f^n(x), f^{n+1}(x)) \qquad \forall n\geq 0.
\]
By Theorem \ref{t:stable} (i) we have $f(x) = O(x^d)$ for $x\rightarrow 0$, and the above identity implies that $\varphi^n(z)$ converges to the origin with order at least $d$ for $n\rightarrow +\infty$. By \eqref{e:stab-unstab}, the same holds true for the convergence for $n\rightarrow -\infty$.

While proving the above theorem, we shall also establish the following additional facts, which will be useful in Section \ref{s:duesol}.

\begin{add}
\label{a:more}
The image by $\varphi^{-1}$ of the graph of $f|_{[0,r]}$ is the set
\[
\varphi^{-1} ( \Graph f|_{[0,r]} ) = S( \mathrm{graph}\, g),
\]
where $g: [0,r] \rightarrow \R$ is the strictly concave continuous function
\[
g(y):=  y^{\frac{1}{d}} - q y - f(y),
\]
which satisfies $\varphi(g(x),x) = (x,f(x))$ for every $x\in [0,r]$. Moreover, $g|_{[0,s]}$ is the inverse of the increasing homeomorphism $f|_{[0,r]} : [0,r] \rightarrow [0,s]$, $g$ achieves its maximum at $s$ and is strictly decreasing on $[s,r]$, and $g(c)>c$.
\end{add}

\begin{figure}[h]
\centerline{\scalebox{.4}{\includegraphics{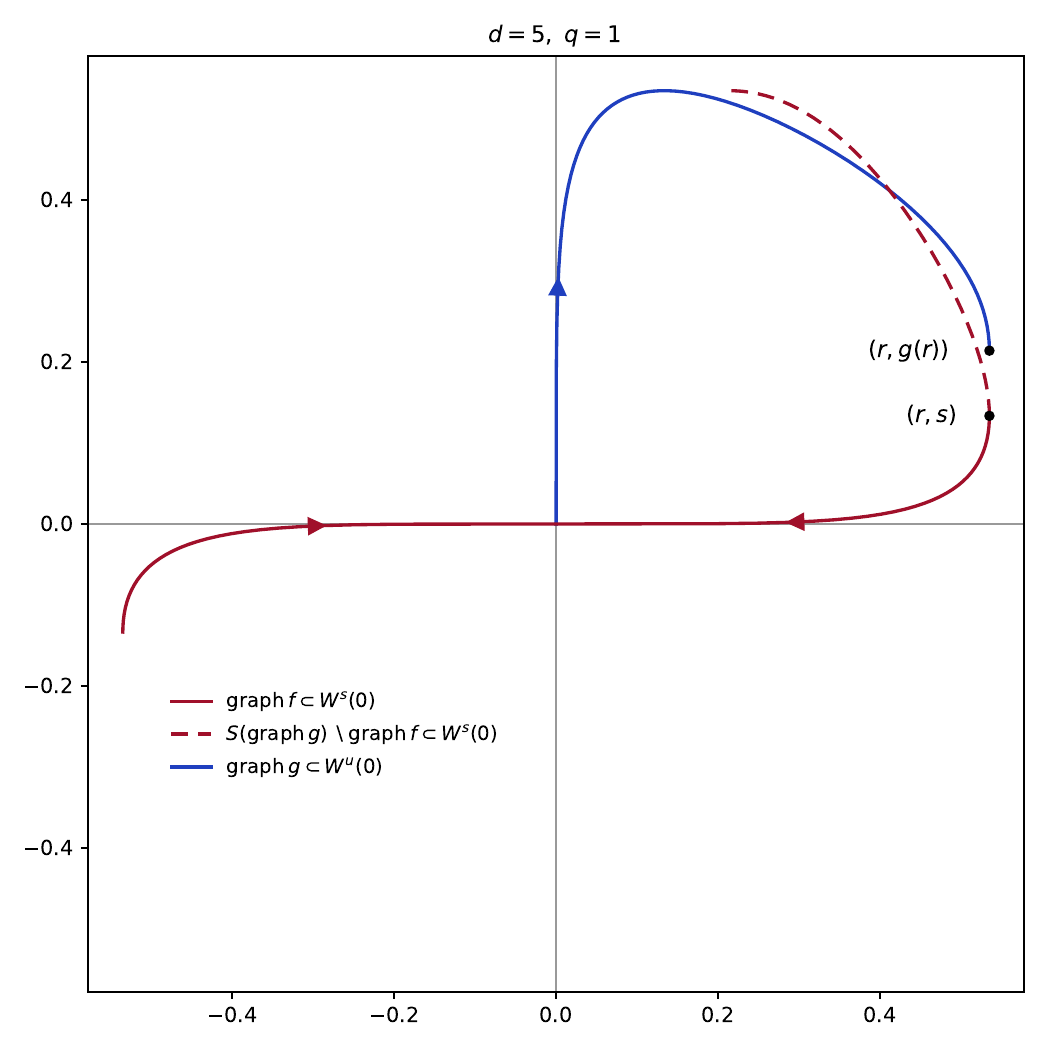} \includegraphics{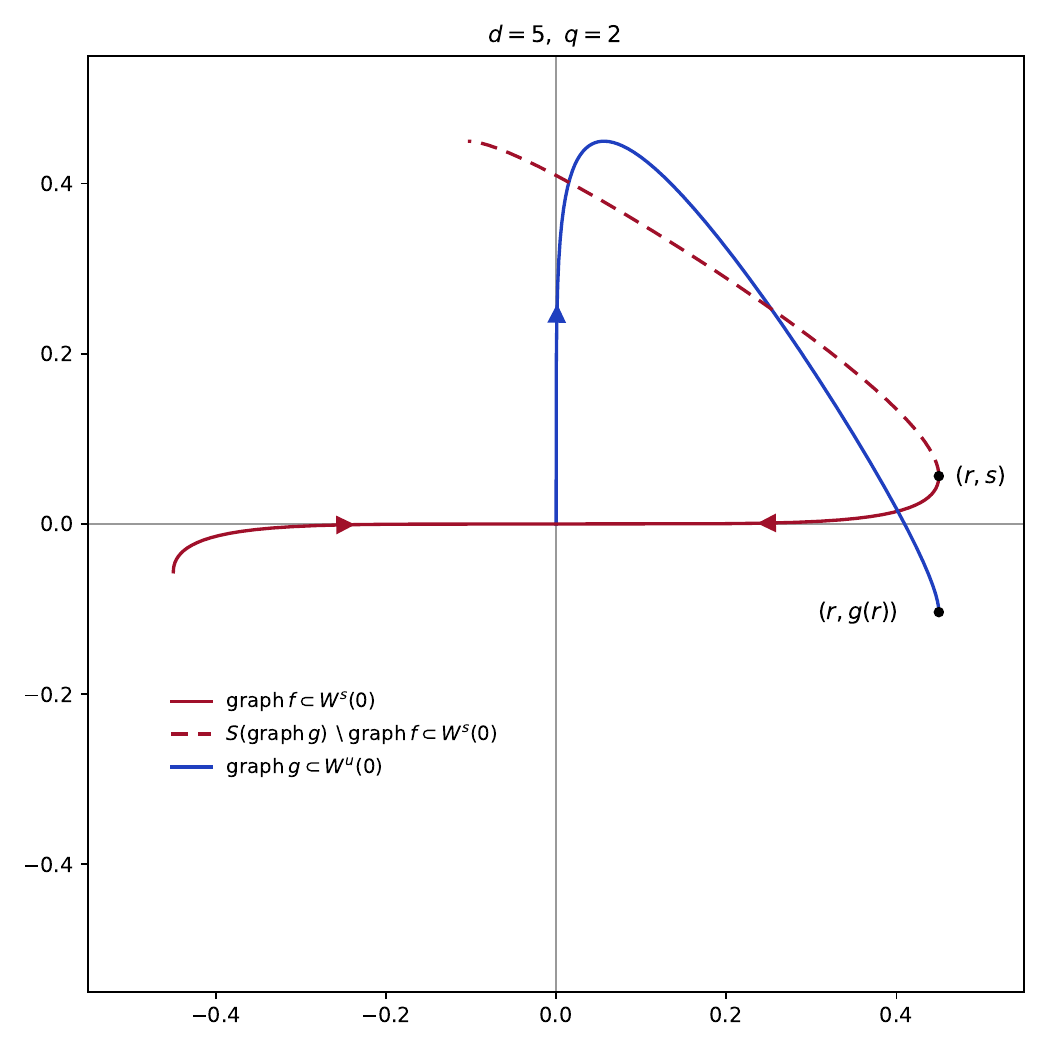}}}
\caption{\small{The graphs of $f$ and $g$ for $d=5$ and $q=1$ (left), and $q=2$ (right).}}
\label{f:grafici}
\end{figure} 

\section{Proof of the stable manifold theorem}
\label{s:psmt}

The aim of this section is to prove Theorem \ref{t:stable} and Addendum \ref{a:more}.

\paragraph{Invariant graphs.} Fix some $r>0$. We start by looking for a function $f: [0,r] \rightarrow [0,+\infty)$ whose graph is mapped to itself by $\varphi$. The identity in the condition
\[
\forall x\in [0,r] \;\exists x'\in [0,r] \mbox{ such that } \varphi(x,f(x)) = (x',f(x'))
\]
 is equivalent to
 \[
 f(x) = x', \quad f(x)^{\frac{1}{d}} - q f(x) - x = f(x'),
 \]
 and hence to
 \[
 f(x)^{\frac{1}{d}} - q f(x) - x = f(f(x)),
 \]
 i.e.\
 \[
 f(x) = \bigl( x + q f(x) + f(f(x)) \bigr)^d.
 \]
We conclude that the graph of $f: [0,r] \rightarrow [0,+\infty)$ is $\varphi$-invariant if and only if 
\[
f(x) \leq r \quad \mbox{and} \quad f(x) = \bigl( x + q f(x) + f(f(x)) \bigr)^d, \qquad \forall x\in [0,r].
\]
The next paragraphs are devoted to the study of the above equation.

\paragraph{Monotone iteration of formal power series.} Before studying the above equation in some space of functions, we study it in the algebra $\R[[x]]$ of formal power series with real coefficients. We recall that the composition $F\circ G$ is a well-defined element of $\R[[x]]$ for every $F\in \R[[x]]$ and $G$ in the ideal $\mathfrak{m} := x\cdot \R[[x]]$. Therefore, the operator
\[
\mathscr{H}: \mathfrak{m} \rightarrow \mathfrak{m}, \qquad \mathscr{H}(F) := \bigl( x + q F + F\circ F \bigr)^d,
\]
is well defined. On $\R[[x]]$, we consider the standard partial order ($F\leq G$ means that the $k$-th coefficient of $F$ is not larger than the $k$-th coefficient of $G$, for every $k\in \N\cup \{0\}$) and the standard topology given by convergence of all coefficients.

\begin{lem}
\label{l:formal}
Set $F_0=0$ and define $F_n\in \mathfrak{m}$ iteratively by $F_{n+1}=\mathscr{H}(F_n)$. Then $(F_n)$ is non-decreasing and converges to the unique fixed point of $\mathscr{H}$, that is, to the unique solution $F\in \mathfrak{m}$ of the equation
\begin{equation}
\label{e:formal}
F = \bigl( x + q F + F\circ F \bigr)^d.
\end{equation}
This solution has the form
\[
F(x) = x^d G(x^{d-1}),
\]
where $G\in \R[[y]]$ is a power series with non-negative coefficients and coefficient of $y^0$ equal to $1$. In the case $q=0$, the solution has the form
\[
F(x) = x^d H(x^{d^2-1}),
\]
where $H\in \R[[z]]$ is a power series with non-negative integer coefficients and coefficient of $z^0$ equal to $1$.
\end{lem}

\begin{proof}
The existence and uniqueness of the fixed point of $\mathscr{H}$ follow from the fact that for every integer $j\geq 1$ the coefficient of $x^j$ in the power series on the right-hand side of (\ref{e:formal}) depends only on the coefficients of $F$ of degree strictly less than $j$. 

Note that if $0\leq F \leq G$, then $0\leq \mathscr{H}(F) \leq \mathscr{H}(G)$. Since $F_1=x^d\geq 0=F_0$, an induction argument shows that $F_n\geq 0$ and $(F_n)$ is non-decreasing.

For each $k\in \N$, the sequence consisting of the $k$-th coefficients of $F_n$ is constant for $n$ large enough. Indeed, denoting by $k_n$ the degree of the first non-vanishing monomial in $F_{n+1}-F_n$, we easily see that $k_n \geq k_{n-1} + d -1$ for every $n\in \N$, and hence $(k_n)$ diverges.

In particular, $(F_n)$ converges to some $F$ in $\mathfrak{m}$ and by taking the limit in $F_{n+1}=\mathscr{H}(F_n)$ we deduce that $F$ is the fixed point of $\mathscr{H}$.

A simple computation involving the change of variable $y=x^{d-1}$ shows that a power series of the form $F(x)=x^d G(x^{d-1})$ satisfies (\ref{e:formal}) if,  and only if, $G$ satisfies
\[
G(y) = \bigl(1+ q y G(y) + y^{d+1} G(y)^d G(y^d G(y)^{d-1}) \bigr)^d.
\]
This equation has a unique solution $G\in \R[[y]]$, its coefficients are non-negative, and the coefficient of $y^0$ is 1. By uniqueness, the solution $F$ of (\ref{e:formal}) has the stated form. In the case $q=0$, we find that  a power series of the form $F(x)=x^d H(x^{d^2-1})$ satisfies (\ref{e:formal}) if, and only if, $H$ satisfies
\[
H(z) = \bigl(1+z H(z)^d H(z^d H(z)^{d^2-1}) \bigr)^d.
\]
This equation has a unique solution $H\in \R[[z]]$, its coefficients are non-negative integers, and the coefficient of $z^0$ is 1. We can then conclude as before.
\end{proof}

\begin{rem}
{\rm In the case $d=2$, $q=0$, the solution $F$ has the form $F(x)=x^2 H(x^3)$, where the first coefficients $h_0,h_1,h_2,\ldots$ of the power series
\[
H(x) = \sum_{n=0}^{\infty} h_n x^n
\]
are the integers
\[
1, 2, 9, 56, 400, 3096, 25256, 213832, 1861272, 16552320, 149737632, 1373597892,  \ldots.
\]
This sequence appears as A212392 in the ``On-Line Encyclopedia of Integer Sequences'' \cite{oeis}. In the case $d=3$, $q=0$, the solution $f$ has the form $F(x)=x^3 H(x^8)$, where the first coefficients of the power series $H$ are the integers
\[
1, 3, 30, 406, 6336, 107415, 1922310, 35739990, 683593902, 13364444808, 265869803598,
\ldots.
\]
This sequence appears as A212425 in \cite{oeis}.}
\end{rem}

\paragraph{Monotone iteration of functions.} The convex function 
\[
\phi:[0,+\infty) \rightarrow \R, \qquad \phi(x) := \frac{1+qx+x^2}{2+q} - x^{\frac{1}{d}}
\]
satisfies
\[
\phi(0)=\frac{1}{2+q} >0, \quad \phi(1)=0, \quad \phi'(1) = 1- \frac{1}{d} >0, 
\]
and hence has a unique zero $\eta$ in the open interval $(0,1)$, which satisfies
\begin{equation}
\label{e:eta}
(1+q\eta+\eta^2)^d = \eta (2+q)^d.
\end{equation}
Let $c=(2+q)^{-\frac{d}{d-1}}$ be the number in $(0,1)$ defined in \eqref{c} and consider the space of functions
\[
\mathscr{F} := \{ f : [0,c) \rightarrow \R \mid 0 \leq f(x) \leq \eta x \; \forall x\in [0,c) \},
\]
and the operator $\mathscr{H}$ which to every $f\in \mathscr{F}$ associates the function
\[
\mathscr{H}(f)(x) :=  \bigl( x + q f(x) + f(f(x)) \bigr)^d.
\]

\begin{lem}
\label{l:it-fun}
The operator $\mathscr{H}$ maps $\mathscr{F}$ to $\mathscr{F}$.  Set $f_0=0$ and define $f_n\in \mathscr{F}$ iteratively by $f_{n+1}=\mathscr{H}(f_n)$. Then $(f_n)$ is non-decreasing and converges pointwise to a function $f\in \mathscr{F}$ such that
\begin{equation}
\label{e:u-fixed-point}
f(x) = \bigl( x + q f(x) + f(f(x)) \bigr)^d \qquad \forall x\in [0,c).
\end{equation}
\end{lem}

\begin{proof}
For every $f\in \mathscr{F}$ and $x\in [0,c)$ we have by \eqref{e:eta}
\[
0\leq (\mathscr{H}(f))(x) \leq (x + q \eta x + \eta^2x)^d = x^d (1+ q\eta+ \eta^2)^d \leq c^{d-1} (2+q)^d \eta x = \eta x,
\]
and hence $\mathscr{H}(f)$ is in $\mathscr{F}$.

Now notice that if $f\in \mathscr{F}$ is non-decreasing then $\mathscr{H}(f)$ is non-decreasing as well. Indeed, if  $0\leq x \leq y < c$ then
\[
\mathscr{H}(f)(x) = (x+qf(x)+ f(f(x)))^d \leq (y+qf(y) + f(f(y)))^d = \mathscr{H}(f)(y).
\]

Next observe that if $f,g\in \mathscr{F}$ satisfy $f\leq g$ and $f$ is non-decreasing, then $\mathscr{H}(f) \leq \mathscr{H}(g)$. Indeed, for every $x\in [0,c)$ we have
\[
\begin{split}
\mathscr{H}(f)(x) &= (x+q f(x) + f(f(x)))^d \leq (x+ q g(x) + f(g(x)))^d \\ &\leq (x+ q g(x) + g(g(x)))^d = \mathscr{H}(g)(x).
\end{split}
\]

Since $f_1=x^d \geq 0 =f_0$, an induction argument shows that every function $f_n$ is non-decreasing and $\mathscr{H}(f_n)\geq f_n$. Moreover, the functions $f_n$ are restrictions of polynomials. Being bounded from above by the identity function, the non-decreasing sequence $(f_n)$ converges pointwise to a function $f:[0,c)\rightarrow \R$. The function $f$ belongs to $\mathscr{F}$, because this set is closed under pointwise convergence. 

In order to take the limit in the identity $f_{n+1}=\mathscr{H}(f_n)$, we must
check that $f_n\circ f_n$ converges to $f\circ f$ pointwise. Since
$f_n$ is non-decreasing and $f_n\leq f$, we have
\[
f_n\circ f_n \leq f_n \circ f \leq f\circ f.
\]
On the other hand, for every $m\leq n$ we have
$f_n \circ f_n \geq f_m\circ f_n$, and the latter sequence of functions converges to
$f_m\circ f$ pointwise for $n\rightarrow \infty$, because $f_m$ is continuous. Therefore,
\[
f_m\circ f \leq \liminf_{n\rightarrow \infty} f_n\circ f_n \leq
\limsup_{n\rightarrow \infty} f_n\circ f_n \leq f\circ f \qquad \forall m\in \N,
\]
and the claim follows by taking the limit for $m\rightarrow \infty$. We conclude
that $f$ is a fixed point of $\mathscr{H}$, meaning that
\eqref{e:u-fixed-point} holds.
\end{proof}

From now on, let $F\in \mathfrak{m} \subset \R[[x]]$ be the solution of \eqref{e:formal} and 
$f\in \mathscr{F}$ be the solution of \eqref{e:u-fixed-point} as in the above lemma. 

\begin{lem}
\label{l:u-versus-U}
The power series $F$ has convergence radius $r\geq c$ and
\[
F(x) = f(x), \qquad \forall x\in [0,c).
\]
\end{lem}

\begin{proof}
This is a standard consequence of monotone convergence. We include a proof for the reader's convenience. By construction, $f_n(x) = F_n(x)$ are polynomials with non-negative coefficients:
\[
f_n(x) = F_n(x) = \sum_{k \geq 0} a_{n,k} x^k,
\]
where $a_{n,k} \geq 0$ and for every $n$ the numbers $a_{n,k}$ are zero for $k$ large enough. Moreover, the sequences $(a_{n,k})_n$ are non-decreasing. The power series $F$ has the form
\[
F(x) = \sum_{k\geq 0} a_k x^k,
\]
where  
\[
a_k = \lim_{n\rightarrow \infty} a_{n,k}.
\]
Fix $x\in [0,c)$. For every $m$ and $n$ we have
\[
\sum_{k=0}^m a_{n,k} x^k \leq f_n(x) \leq f(x),
\]
and by taking a limit for $n\rightarrow \infty$ we find
\[
\sum_{k=0}^m a_k x^k \leq f(x).
\]
On the left, we see a partial sum of the series $F(x)$, which has non-negative terms, and the above uniform bound implies that this series converges. Therefore, the convergence radius $r$ of the power series $F$ is at least $x$. Since $x\in [0,c)$ was arbitrary, $r$ is not smaller than $c$.

Since $a_{n,k} x^k\leq a_k x^k$ and the latter sequence is summable for $x\in [0,c)$, the dominated convergence theorem implies that for such $x$ the sequence
\[
u_n(x) = \sum_{k \geq 0} a_{n,k} x^k
\]
converges to $F(x)$. Since the same sequence also converges to $f(x)$, we conclude that $F(x)=f(x)$ for every $x\in [0,c)$.
\end{proof}

\paragraph{Analytic continuation.} The power series $F$ with convergence radius $r$ defines a real analytic function $\tilde{f}: J \rightarrow \R$
on some maximal open interval $J\subset \R$ containing $(-r,r)$. By Lemma \ref{l:u-versus-U}, $r\geq c$ and the restriction of $\tilde{f}$ to $[0,c)$ is $f$.
We keep the same symbol $f$ to denote the restriction of $\tilde{f}$ to $J \cap [0,+\infty)$.

Since the coefficients of the power series $F$ are non-negative, the Vivanti--Pringsheim theorem implies that
\[
\sup J = r.
\]
Therefore, $f$ is an analytic, strictly increasing, and strictly convex function on $[0,r)$, and by Lemma \ref{l:formal} it satisfies 
\begin{equation}
\label{e:for-u}
f(x) = F(x) =  x^d + O(x^{2d-1}) \qquad \forall x\in [0,r),
\end{equation}
and in particular $f(0)=f'(0)=0$.

\begin{lem}
\label{l:u-glob}
The convergence radius $r$ of $F$ is strictly smaller than 1. For every $x\in (0,r)$ we have $f(x) < x$, and
\[
s:=\lim_{x\uparrow r} f(x) < r.
\]
Moreover,
\[
f(x) = (x + q f(x) + f(f(x)))^d, \qquad \forall x\in [0,r).
\]
\end{lem}

\begin{proof}
If $r=c$ then there is nothing to prove, because $c<1$, $f(x) \leq \eta x$ for every $x\in [0,c)$, with $\eta\in (0,1)$, and $f$ satisfies \eqref{e:u-fixed-point}. 

Therefore, we can assume that $r>c$, so that $c$ belongs to the domain of analyticity of $f$. Since $f(c) \leq \eta c < c$, the number
\[
t:= \sup\{ x\in (0,r) \mid f(x)<x \} 
\]
is strictly larger than $c$. Since $f(0)=0$ and $x\mapsto \frac{f(x)}{x}$ is non-decreasing, by convexity, on the  interval $[0,t)$ we have $0\leq f(x)< x< r$. Therefore, the composition $f(f(x))$ is well defined. The functions on the left- and right-hand sides of \eqref{e:u-fixed-point} are analytic on $(0,t)$ and coincide on $[0,c)$. Therefore, the identity
\begin{equation}
\label{e:equ}
f(x) = (x + q f(x) + f(f(x)))^d
\end{equation}
holds for every $x\in [0,t)$.

If by contradiction $t<r$, then $f(t)=t$ and by continuity \eqref{e:equ} holds also for $x=t$. This implies
\[
t = (t + q t + t)^d = t^d (2+q)^d = t^d c^{1-d},
\]
from which we deduce that $t=c$, contradicting the fact that $t>c$. This proves that $f(x)<x$ for every $x\in (0,r)$ and that \eqref{e:equ} holds for every $x\in [0,r)$.

For every $x\in [0,r)$ identity \eqref{e:equ} implies the lower bounds
\[
f(x) \geq x^d \quad \Rightarrow \quad f(x) \geq (x + f(f(x)))^d \geq  (x + x^{d^2})^d \geq x^d + d x^{d^2+d-1},
\]
which together with the upper bound $f(x)\leq x$ imply that $r<1$. Moreover,
\[
\lim_{x\uparrow r} f(x) \leq r,
\]
and if by contradiction the above inequality is not strict, then passing to the limit for $x\uparrow r$ we obtain
\[
r = (r + qr + r)^d = r^d (2+q)^d,
\]
which implies that $r=c$, contradicting our assumption $r>c$.
\end{proof}

\paragraph{Dynamical consequences.} We now go back to the dynamics of the homeomorphism $\varphi: \R^2 \rightarrow \R^2$.

\begin{lem}
\label{l:stab}
For every $x\in [0,r)$ we have
\begin{equation}
\label{e:inv-graph}
\varphi(x,f(x)) = \bigl(f(x),f(f(x))\bigr).
\end{equation}
Moreover, the graph of $f$ -- which by the above identity is positively invariant by $\varphi$ -- is contained in the stable set $W^s(0)$. 
\end{lem}

\begin{proof}
For every $x\in [0,r)$ we have by Lemma \ref{l:u-glob}
\[
\begin{split}
\varphi(x,f(x)) &= (f(x), f(x)^{\frac{1}{d}} - q f(x) - x) = (f(x), x + q f(x) + f(f(x)) - q f(x) - x) \\ &= (f(x),f(f(x))),
\end{split}
\]
proving \eqref{e:inv-graph}. By iterating this identity we find
\[
\varphi^n(x,f(x)) = (f^n(x),f^{n+1}(x)) \qquad \forall n\in \N.
\]
Now assume that $x\in (0,r)$. Since $f(0)=0$ and $0<f(x)<x$ for every $x\in (0,r)$, the strictly decreasing sequence $f^n(x)$ converges to 0, and hence $\varphi^n(x,f(x))$ converges to the origin.
\end{proof}

Recall that in Lemma \ref{l:u-glob} we defined $s\in (0,r)$ to be the number
\[
s:= \lim_{x\uparrow r} f(x).
\]
Extend $f$ to a continuous function on $[0,r]$ by setting $f(r):=s$. Then $f$ is an increasing homeomorphism from $[0,r]$ to $[0,s]$. By continuity, \eqref{e:inv-graph} remains true for every $x\in [0,r]$.

\begin{lem}
\label{l:vert}
The graph of $f$ has a vertical tangency at $r$, meaning that
\[
\lim_{x\uparrow r} f'(x) = +\infty.
\]
\end{lem}

\begin{proof}
Since $f'>0$ and $f''>0$ on $(0,r)$, $f'(x)$ has a limit $L\in (0,+\infty]$ for $x\uparrow r$, and hence the graph of $f$ has a well-defined tangent line at $(r,f(r))$. We must prove that $L=+\infty$ or, equivalently, that this tangent line is vertical.

The curve
\[
\gamma: (0,r) \rightarrow \R^2, \qquad \gamma(x) := (x,f(x)),
\]
is an analytic embedding onto the graph of $f|_{(0,r)}$. Since this graph is contained in the domain in which $\varphi^{-1}$ is an analytic diffeomorphism, the curve 
\[
\varphi^{-1}\circ \gamma: (0,r) \rightarrow \R^2
\]
is an analytic embedding as well. By the positive $\varphi$-invariance of the graph of $f$, $\gamma((0,r))$ is an open one-dimensional submanifold of $\varphi^{-1}\circ \gamma((0,r))$.

Since 
\[
\varphi(r,f(r)) = (f(r),f(f(r))) = (s,f(s)),
\]
the embedding $\varphi^{-1} \circ \gamma$ maps the interior point $s\in (0,r)$ to $(r,f(r))$. The image of its differential is the tangent line at the graph of $f$ at $(r,f(r))$. If by contradiction this tangent line is not vertical, then the implicit function theorem implies that $\varphi^{-1}\circ \gamma((s-\epsilon,s+\epsilon))$ is the graph of an analytic function $x \mapsto g(x)$ defined on an interval of the form $(r-\delta_1,r+\delta_2)$. Here, $\epsilon$, $\delta_1$, and $\delta_2$ are small positive numbers. Since 
\[
\varphi^{-1}\circ \gamma((s-\epsilon,s)) \subset \{ (x,f(x)) \mid x\in (0,r) \},
\]
we must have $g=f$ on $(r-\delta_1,r)$. Therefore, $f$ can be extended to an analytic function on $(0,r+\delta_2)$, contradicting the fact that $r$ is the supremum of the interval of analyticity of $f$.
\end{proof}

\paragraph{The function $g$.} We now consider the function
\[
g: [0,r] \rightarrow \R, \qquad g(y) := y^{\frac{1}{d}} - q y - f(y),
\]
which is continuous and strictly concave on $[0,r]$ and analytic on $(0,r)$. The relevance of this function is explained by the next lemma.

\begin{lem}
\label{l:v}
The function $g|_{(0,r)}$ is the maximal analytic extension of the inverse of $f|_{(0,r)}$. More precisely:
\begin{enumerate}[(i)]
\item The restriction of $g$ to $[0,s]$ is an increasing homeomorphism onto $[0,r]$, with 
\[
g|_{[0,s]} = f^{-1}: [0,s] \rightarrow [0,r].
\]
\item $g$ achieves its maximum at $s$ and is strictly decreasing on $[s,r]$.
\item The image by $\varphi^{-1}$ of the set $W:= \{(x,f(x)) \mid x\in [0,r]\}$ is
\[
\varphi^{-1}(W) = \{ (g(y),y) \mid y\in [0,r]\}.
\]
\item $\varphi(g(x),x) = (x,f(x))$ for every $x\in [0,r]$.
\item $c<g(c) \leq r$.
\end{enumerate}
\end{lem}  

\begin{proof}
The function $f$ is a homeomorphism from $[0,r]$ to $[0,s]$. Let $y\in [0,s]$. By applying \eqref{e:inv-graph} to $x=f^{-1}(y)$, we find
\[
\varphi(f^{-1}(y),y) = (y,f(y)).
\]
On the other hand, the definition of $\varphi$ gives us the identity
\[
\varphi(f^{-1}(y),y) = (y,y^{\frac{1}{d}} - q y - f^{-1}(y)).
\]
By comparing the second components of these vectors, we obtain
\[
f^{-1}(y) =  y^{\frac{1}{d}} - q y - f(y) = g(y),
\]
proving (i). From $f'(r)=+\infty$ we deduce $g'(s)=0$. Then (ii) follows from the fact that $g$ is strictly concave. From the expression \eqref{e:inve} for the inverse of $\varphi$ we find
\[
\varphi^{-1}(x,f(x)) = (g(x),x), \qquad \forall x\in [0,r],
\]
which implies (iii) and (iv). From the definition of $c$ we obtain
\[
q = c^{\frac{1}{d}-1} - 2,
\]
and hence 
\[
g(c) = c^{\frac{1}{d}} - qc - f(c) = 2c - f(c) = c + c - f(c) > c,
\]
where the inequality follows from Lemma \ref{l:u-glob}. Together with $\max g = g(s) =r$ by (i) and (ii), this proves (v).
\end{proof}

\paragraph{Proof of Theorem \ref{t:stable} and Addendum \ref{a:more}.} We can finally prove the stable manifold Theorem \ref{t:stable} and its Addendum \ref{a:more}. The numbers $r$ and $s$ satisfy the required conditions $r\in(c,1)$ and $s\in (0,r)$ by Lemma \ref{l:u-glob} and Lemma \ref{l:v} (v).

We extend $f: [0,r) \rightarrow \R$ to an odd function on $[-r,r]$, keeping the same name $f$ for this extension. Then $f$ maps $[-r,r]$ homeomorphically onto $[-s,s]$ and satisfies the properties listed in (i) because of \eqref{e:for-u} and Lemmata \ref{l:formal}, \ref{l:u-glob}, and \ref{l:vert}. Statement (ii) follows from Lemma \ref{l:stab}. Addendum \ref{a:more} follows from Lemma \ref{l:v}. Finally, statement (iii) of Theorem \ref{t:stable}  follows from the next lemma.

\begin{lem}
Assume that $\rho>0$ satisfies
\begin{equation}
\label{e:rho}
d (q+2)^d \rho^{d-1} <1.
\end{equation}
If $\varphi^n(x,y)\in [-\rho,\rho] \times [-\rho,\rho]$ for every $n\geq 0$, then $y=f(x)$.
\end{lem}

\begin{proof}
First notice that
\[
\rho < d^{-\frac{1}{d-1}} (q+2)^{-\frac{d}{d-1}} < c <r,
\]
and hence $[-\rho,\rho]$ is contained in the domain of $f$. Setting $(x_n,y_n) = \varphi^n(x,y)$, we have
\[
x_n = h ( x_{n-1} + q x_n + x_{n+1}),
\]
where $h(s):= |s|^{d-1}s$. In particular, the sequence $(x_n)_{n\geq 1}$ is a fixed point of the map 
\[
\mathscr{H}_x : \ell^{\infty}(\N) \rightarrow  \ell^{\infty}(\N), \qquad (\mathscr{H}_x(u))(n) := h ( u(n-1) + q u(n) + u(n+1)), \quad \forall n\geq 1,
\]
where we are setting $u(0):= x$. The restriction of $h$ to $[-s,s]$ has Lipschitz constant $d s^{d-1}$. For $u$ and $v$ in $\ell^{\infty}(\N)$ with $\|u\|_{\infty}$, $\|v\|_{\infty}\leq \rho$, we then find for every $n\geq 1$, setting $u(0)=v(0)=x$, 
\[
\begin{split}
|\mathscr{H}_x(u)(n) &- \mathscr{H}_x(v)(n)| \\ 
&= \bigl|  h ( u(n-1) + q u(n) + u(n+1)) -  h ( v(n-1) + q v(n) + v(n+1)) \bigr| \\ &\leq d ((q+2)\rho)^{d-1} \bigl| u(n-1) + q u(n) + u(n+1) -  v(n-1) - q v(n) - v(n+1) | \\ &\leq d ((q+2)\rho)^{d-1} (q+2) \|u-v\|_{\infty} = d (q+2)^d \rho^{d-1}  \|u-v\|_{\infty}.
\end{split}
\]
By \eqref{e:rho}, $\mathscr{H}_x$ has Lipschitz constant less than 1 on the closed ball of radius $\rho$ in $\ell^{\infty}(\N)$. Therefore, $(x_n)_{n\geq 1}$ is the unique fixed point of $\mathscr{H}_x$ in this ball. On the other hand, the sequence $(f^n(x))_{n\geq 1}$ is contained in this ball and is also a fixed point of $\mathscr{H}_x$. We then obtain that $x_n = f^n(x)$ for every $n\geq 1$, and in particular $x_1=f(x)$. Since
\[
(x_1,y_1) = \varphi(x,y) = (y,y^{\frac{1}{d}}-qy-x),
\]
we conclude that $f(x)=y$.
\end{proof}

\section{Existence of two homoclinic orbits}
\label{s:duesol}

In Theorem \ref{t:2gibbs}, we proved by variational methods the existence of a Gibbs measure which is fixed by the involution $\iota_*$ induced by $\iota(i)=-i$ and of a second one which is fixed by the involution $\jmath_*$ induced by $\jmath(i)=1-i$. Now we wish to show how, in the special case of the interaction function $Q$ given by \eqref{e:simpleQ}, two Gibbs measures with the same symmetry can be found by proving the existence of two suitably symmetric homoclinic orbits of the map $\varphi$.  

On the one hand, the dynamical approach gives more precise information on the positive solutions $u\in \ell^p(\Z)$ of \eqref{e:2ndrec}, and hence on the Gibbs measures. On the other hand, in the next sections we shall see how tools from hyperbolic dynamics allow us to use one of the two homoclinic orbits we find here to construct many more homoclinic orbits, and hence many more Gibbs measures.

We need to fix some terminology. By an \textit{arc} $\alpha$ in a surface $M$ (here, $M=\R^2$) we mean the homeomorphic image of an interval $J\subset \R$ with positive length. By $\mathrm{int}(\alpha)$ we mean the image of the interior of $J$ and by the end-points of $\alpha$ the image of the end-points of $J$, when the interval $J$ is compact. Two arcs $\alpha$ and $\beta$ in $M$ are said to have a \textit{topologically transverse} intersection at $p\in \mathrm{int}(\alpha) \cap \mathrm{int}(\beta)$ if there exists a neighborhood $U\subset M$ of $p$ and a homeomorphism $h: U \rightarrow (-1,1)\times (-1,1)$ such that $h(p)=(0,0)$, $h(\alpha \cap U) = (-1,1)\times \{0\}$, and $h(\beta \cap U) = \{0\} \times (-1,1)$. When $\alpha$ and $\beta$ are images of $C^1$-embeddings, topological transversality is implied by \textit{transversality}, that is, the condition that $\alpha$ and $\beta$ have distinct tangent lines at $p$.

Although in general not arcs, the sets $W^s(0)$ and $W^u(0)$ are increasing unions of compact arcs, as in the identities
\[
W^s(0) = \bigcup_{n\in \N} \varphi^{-n}( \mathrm{graph}\, f), \qquad W^u(0) = \bigcup_{n\in \N} \varphi^{n}( S(\mathrm{graph}\, f)).
\] 
We say that $W^s(0)$ and $W^u(0)$ have a topologically transverse intersection at some $z\in W^s(0)\cap W^u(0)$ if, choosing $n$ so large that $z$ belongs to the interior of both $\varphi^{-n}( \mathrm{graph}\, f)$ and  $\varphi^{n}( S(\mathrm{graph}\, f))$, the latter two arcs have a topologically transverse intersection at $z$.

The next theorem gives us the existence of a homoclinic orbit passing through the line of fixed points of the involution $S$. This will give us a Gibbs measure which is fixed by the involution $\jmath_*$.

\begin{thm}
\label{t:omoclina1}
For every real number $q\geq 0$ and integer $d\geq 2$ the map $\varphi$ defined in \eqref{e:mapphi} admits at least one homoclinic orbit $\{z_i\}_{i\in \Z}$ in the quadrant $(0,+\infty)^2$ satisfying the symmetry condition
\begin{equation}
\label{e:1sym}
z_{-i} = S (z_i), \qquad \forall i\in \Z.
\end{equation}
Moreover, writing $z_i=(u(i),u(i+1))$, the positive solution $u\in \ell^p(\Z)$ of \eqref{e:2ndrec}  satisfies 
\[
u(1)=u(0) > u(2)=u(-1) > u(3)=u(-2) > \dots > u(i)=u(1-i) > \dots
\]
and
\begin{equation}
\label{e:orderd}
u(i)^d \leq u(i+1) \leq C \, u(i)^d \qquad \forall i\geq 1,
\end{equation}
for some constant $C$. Furthermore, if $W^s(0)$ does not coincide with $W^u(0)$ then they have a topologically transverse intersection at $z_i$ for every $i\in \Z$.
\end{thm} 

\begin{proof}
We use the notation and results from Section \ref{s:ifss}.
By Addendum \ref{a:more}, we have $g(s)=r>s$, $g(c)>c$, and $g(r)<r$. Therefore, we find $x_* \in (\max\{s,c\},r)$ such that $g(x_*)=x_*$. Since $g$ is decreasing on $[s,r]$, such $x_*$ is unique, and as $g(x)=f^{-1}(x)>x$ on $(0,s]$, $x_*$ is the unique fixed point of $g$ in $(0,r]$. Since the graph of $g$ is contained in $W^u(0)$, the point $z_0:=(x_*,x_*)$ belongs to the intersection of $W^u(0)$ with the line $y=x$ of fixed points of the involution $S$. Therefore, $z_0$ is also in $S(W^u(0))=W^s(0)$ and is hence a homoclinic point of $\varphi$. Since the involution $S$ conjugates $\varphi$ and $\varphi^{-1}$, the  corresponding homoclinic orbit
$z_i := \varphi^i(z_0)$ satisfies
\[
z_{-i} = \varphi^{-i}(z_0) =  \varphi^{-i}\circ S (z_0) = S \circ \varphi^i(z_0) = S(z_i) \qquad \forall i\in \Z,
\]
proving \eqref{e:1sym}. Moreover, 
\[
\varphi(z_0) = (x_*,x_*^{\frac{1}{d}} - q x_* - x_*) =  (x_*,x_*^{\frac{1}{d}} - q x_* - g(x_*)) = (x_*,f(x_*)),
\]
and from the positive invariance of the graph of $f$ we find
\[
z_i = (f^{i-1}(x_*),f^i(x_*)) \qquad \forall i\geq 1,
\]
which together with \eqref{e:1sym} implies that $(z_i)_{i\in \Z}$ is contained in the quadrant $(0,+\infty)^2$. Writing $z_i=(u(i),u(i+1))$ with 
\[
u(1-i) = u(i) = f^{i-1}(x_*) \quad \forall i\geq 1,
\]
we deduce that $u\in \ell^p(\Z)$ is a positive solution of \eqref{e:2ndrec} satisfying 
\[
u(1)=u(0) > u(2)=u(-1) > u(3)=u(-2) > \dots > u(i)=u(1-i) > \dots
\]
Moreover, the identity $u(i+1) = f(u(i))$, holding for every $i\geq 1$, and the bounds 
\[
x^d \leq f(x) \leq C \, x^d \qquad \forall x\in [0,r],
\]
where $C$ is a suitably large number, imply \eqref{e:orderd}. 

There remains to prove that if $W^s(0) \neq W^u(0)$ then $W^s(0)$ and $W^u(0)$ have a topologically transverse intersection at $z_0$, as this implies that their intersection stays topologically transverse along the whole orbit of $z_0$. The analytic arcs $\alpha:= \mathrm{graph}\, g|_{(0,r)}$ and $\beta:= S(\alpha)$ are contained in $W^u(0)$ and $W^s(0)$, respectively, and meet at $z_0$, which is a fixed point of the linear involution $S$. Moreover, both arcs are transverse to the line of fixed points $\{y=x\}$ of $S$ because $g'(x_*)<0$, as $x_*>s$. By analyticity, the assumption $W^s(0) \neq W^u(0)$ implies that the arcs $\alpha$ and $\beta$ do not have common subarcs. These properties imply that $\alpha$ and $\beta$ have a topologically transverse intersection at $z_0$, as shown by the next slightly more general lemma. This concludes the proof of Theorem \ref{t:omoclina1}.
\end{proof}

\begin{lem}
\label{l:transverse}
Let $\psi: \R^2 \rightarrow \R^2$ be an involution fixing a point $z_0\in \R^2$. Assume that $\psi$ is analytic in a neighborhood of $z_0$ and $\mathrm{d}\psi(z_0)$ has eigenvalues $1$ and $-1$. Let $\alpha$ and $\beta=\psi(\alpha)$ be analytic arcs in $\R^2$ with $z_0\in \mathrm{int}(\alpha) \cap \mathrm{int}(\beta)$. Assume that both $\alpha$ and $\beta$ are transverse at $z_0$ to the fixed point set of $\psi$, which by the above assumptions is an analytic arc near $z_0$. If $\alpha$ and $\beta$ do not have a common subarc, then they have a topologically transverse intersection at $z_0$.
\end{lem}

\begin{proof}
The map 
\[
h:= \frac{1}{2} ( \mathrm{id} + \mathrm{d} \psi(z_0) \circ \psi ),
\]
which is an analytic diffeomorphism near $z_0$, conjugates $\psi$ to the linear map $\mathrm{d} \psi(z_0)$. Being a linear involution with eigenvalues $1$ and $-1$, $\mathrm{d}\psi(z_0)$ is a reflection across a line. By a further affine conjugacy, we can then conjugate $\psi$ analytically near $z_0$ to the reflection $(x,y)\mapsto (-x,y)$ across the vertical axis $x=0$. This conjugacy maps $z_0$ to $(0,0)$ and brings the germ at $z_0$ of the fixed point set of $\psi$ to the vertical axis $x=0$, showing that this fixed point set is an analytic arc near $z_0$. Since $\alpha$ and $\beta$ are transverse to the fixed point set of $\psi$ at $z_0$, the conjugacy brings their germs at $z_0$ to graphs of analytic functions of $x$. Using also the fact that $\psi(\alpha)=\beta$, we deduce that the germ of $\alpha$ is mapped to the graph of an analytic function $x \mapsto v(x)$, while the germ of $\beta$ is mapped to the graph of $x \mapsto v(-x)$.

Since the arcs $\alpha$ and $\beta$ do not have a common subarc, the function $w(x) := v(x) - v(-x)$ does not vanish on an interval of positive length. Since $w$ is odd and analytic, we have $w(x) = c \, x^{2k+1} + O(x^{2k+2})$ for some $c\neq 0$ and non-negative integer $k$. Then $w$ changes sign at $0$, so the graphs of $v$ and $x \mapsto v(-x)$ have a topologically transverse intersection at $(0,0)$. The same is then true for $\alpha$ and $\beta$ at $z_0$.
\end{proof}
 
Now we wish to establish the existence of a second homoclinic orbit $\{z_i\}_{i\in \Z}$ in $(0,+\infty)^2$ satisfying the symmetry condition $z_{1-i}= S(z_i)$ for every $i\in \Z$. This will correspond to a Gibbs measure which is fixed by the involution $\iota_*$.

\begin{thm}
\label{t:omoclina2}
For every real number $q\geq 0$ and integer $d\geq 2$ the map $\varphi$ defined in \eqref{e:mapphi} admits at least one homoclinic orbit $\{z_i\}_{i\in \Z}$ in the quadrant $(0,+\infty)^2$ satisfying the symmetry condition
\begin{equation}
\label{e:2sym}
z_{1-i} = S (z_i), \qquad \forall i\in \Z.
\end{equation}
Moreover, writing $z_i=(u(i-1),u(i))$ for every $i\in \Z$, the positive solution $u\in \ell^p(\Z)$ of \eqref{e:2ndrec}  satisfies 
\begin{equation}
\label{e:dove}
u(0) > u(1)=u(-1) > u(2)=u(-2) > \dots > u(i)=u(-i) > \dots
\end{equation}
and
\[
u(i)^d \leq u(i+1) \leq C u(i)^d \qquad \forall i\geq 1,
\]
for some constant $C$. 
Furthermore, if $W^s(0)$ does not coincide with $W^u(0)$ then they have a topologically transverse intersection at $z_i$ for every $i\in \Z$.
\end{thm}

\begin{figure}[h]
\centerline{\scalebox{.4}{\includegraphics{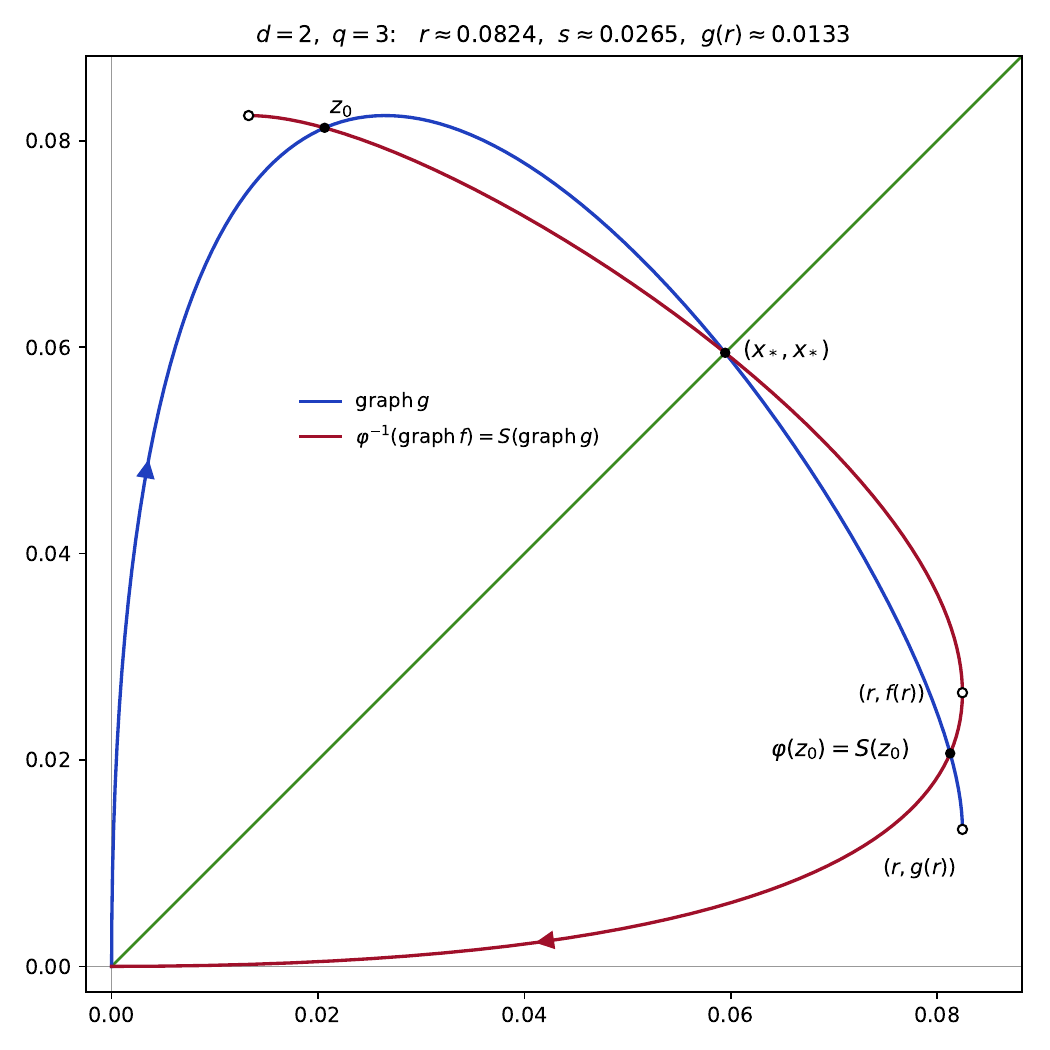} \includegraphics{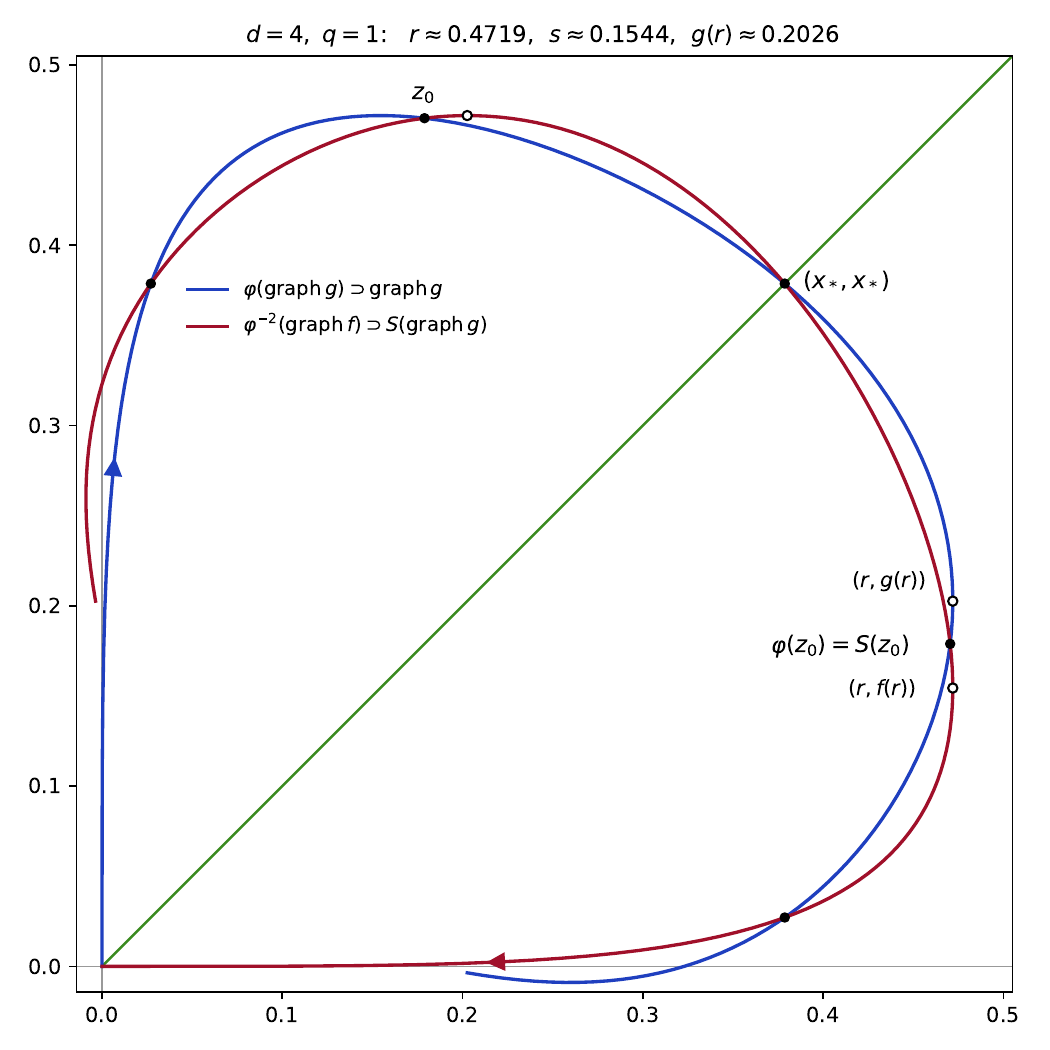}}}
\caption{\footnotesize{Subarcs of $W^s(0)$ and $W^u(0)$ in the case $g(r)<s$ (left) and $g(r)\in (s,r)$ (right).}}
\label{f:intersezioni}
\end{figure} 

\begin{proof}
Since $g(s)=r$ and $g$ decreases on $[s,r]$, we have $g(r)<r$. We distinguish three cases.

\paragraph{Case 1: $g(r)<s$.} This case is exemplified by the picture on the left in Figure \ref{f:intersezioni}. Since $f(s)<r=g(s)$ and $f(r) = s > g(r)$, the compact arcs
\[
\mathrm{graph}\, f|_{[0,r]}  \subset W^s(0) \qquad \mbox{and} \qquad \mathrm{graph}\, g  \subset W^u(0)
\]
have a common interior point $z_1$ of the form $z_1=(x_1,f(x_1))=(x_1,g(x_1))$ with $s<x_1<r$. This intersection is transverse, and a fortiori topologically transverse, because $f'>0$ and $g'<0$ on $(s,r)$. Therefore, $W^s(0)$ and $W^u(0)$ have a topologically transverse intersection at $z_1$ and, by applying the involution $S$ swapping $W^s(0)$ and $W^u(0)$, also at the homoclinic point $z_0 := S(z_1)$. From the identity
\[
\varphi(z_0) = \varphi(f(x_1),x_1) = (x_1, x_1^{\frac{1}{d}} - q x_1 - f(x_1)) = (x_1,g(x_1)) = z_1 = S(z_0),
\]
we deduce that the orbit $z_i:= \varphi^i(z_0)$ of $z_0$ satisfies \eqref{e:2sym}. Indeed, the fact that the involution $S$ conjugates $\varphi$ and $\varphi^{-1}$ implies
\[
z_{1-i} = \varphi^{-i} \circ \varphi(z_0) = \varphi^{-i} \circ S(z_0) = S \circ \varphi^i (z_0) = S(z_i) \qquad \forall i\in \Z.
\]
Writing $z_i=(u(i-1),u(i))$, this symmetry condition translates into $u(-i)=u(i)$ for every $i\in \Z$. Moreover, the identity $z_1=(x_1,f(x_1))$ and the positive invariance of the graph of $f$ imply that $z_{i+1}= (f^i(x_1),f^{i+1}(x_1))$ for every $i\geq 0$, which translates into $u(i)=f^i(x_1)$ for every $i\geq 0$. The fact that $0<f(x)<x$ on $(0,r]$ implies that the orbit of $z_0$ is contained in $(0,+\infty)^2$ and that
\eqref{e:dove} holds.

\paragraph{Case 2: $g(r)=s$.} Since $g(s)=\max g = r$, the point $z_0:=(s,r)$ belongs to the interior of the compact arc
\[
\beta:= \mathrm{graph}\, g \subset W^u(0).
\]
The assumption $g(r)=s$ implies that 
\[
\varphi(z_0)=(r,r^{\frac{1}{d}} - q r - s) = (r,r^{\frac{1}{d}} - q r - g(r)) = (r,f(r)) = (r,s) = S(z_0).
\]
Since $S(z_0)=(r,s)$ belongs to the interior of the compact arc $\varphi^{-1} (  \mathrm{graph}\, f|_{[0,r]} )$, the point $z_0$ belongs to the interior of the compact arc 
\[
\alpha := \varphi^{-2} (  \mathrm{graph}\, f|_{[0,r]} ) \subset W^s(0).
\]
Therefore, $z_0$ is a homoclinic point. Arguing as in Case 1, we easily see that its orbit $z_i = \varphi^i(z_0) = (u(i-1),u(i))$ satisfies \eqref{e:2sym} and \eqref{e:dove}. There remains to prove that the arcs $\alpha$ and $\beta$ have a topologically transverse intersection at $z_0$, provided that $W^s(0)\neq W^u(0)$. Note that $\alpha$ and $\beta$ are both tangent to the line $y=r$ at $z_0$, so this intersection is not transverse.

A first observation is that both arcs $\mathrm{int}(\alpha)$ and $\mathrm{int}(\beta)$ are analytic. For $\mathrm{int}(\beta)$ this is clear, because the function $g$ is analytic on $(0,r)$. The analyticity of $\mathrm{int}(\alpha)$ follows from the analyticity of the arc $\varphi^{-1} (  \mathrm{graph}\, f|_{(0,r)} ) = S (\mathrm{graph}\, g|_{(0,r)})$ and from the fact that this arc lies in the region $x>0$, where $\varphi^{-1}$ is analytic, because in this case $g$ is positive on $(0,r]$.

If $\alpha$ and $\beta$ have a subarc in common, then by analyticity $\alpha\setminus \mathrm{graph}\, f$ is contained in $\beta$ and hence in one branch of $W^u(0)$. Since the latter set is also contained in one branch of $W^s(0)$, we deduce that these two branches coincide. Since the other branches are obtained by applying the map $(x,y) \mapsto (-x,-y)$, we conclude that $W^s(0)=W^u(0)$, which is a contradiction. This shows that $\alpha$ and $\beta$ do not have a common subarc.

The homeomorphism
\[
\psi := S\circ \varphi: \R^2 \rightarrow \R^2, \qquad \psi(x,y) = \bigl(|y|^{\frac{1}{d}-1} y - q y - x, y\bigr),
\]
is an involution. It fixes $z_0$ because $r^{\frac{1}{d}}- q r - s = r^{\frac{1}{d}}- q r - g(r) = f(r) = s$. The map $\psi$ is analytic near $z_0$, and the differential 
\[
\mathrm{d} \psi(z_0) = \left( \begin{array}{cc} -1 & \frac{1}{d} r^{\frac{1}{d}-1} - q \\ 0 & 1 \end{array} \right)
\]
is a linear involution with eigenvalues $\pm 1$. Since
\[
\begin{split}
\psi(\beta) &= S \circ \varphi ( \mathrm{graph} \, g) = S \circ \varphi  ( \varphi\circ S (\mathrm{graph} \, f) ) \\ &= (S \circ \varphi \circ S)^2  (\mathrm{graph} \, f)  = \varphi^{-2}  (\mathrm{graph} \, f)  = \alpha,
\end{split}
\]
the arcs $\alpha$ and $\beta$ are mirror images of each other under $\psi$. Together with the fact that these arcs are both transverse at $z_0$ to the curve of fixed points of $\psi$, namely the set
\[
\Bigl\{(x,y)\in \R^2 \mid x = \frac{1}{2} |y|^{\frac{1}{d}-1} y - \frac{q}{2} y\Bigr\},
\]
this implies that they have a  topologically transverse intersection at $z_0$, thanks to Lemma \ref{l:transverse}. 

\paragraph{Case 3: $g(r)\in (s,r)$.} This case is exemplified by the picture on the right in Figure \ref{f:intersezioni}.  Denote by
\[
\tilde{f} : [g(r),r] \rightarrow [s,r]
\]
the inverse of the strictly decreasing function $g|_{[s,r]}$. Then
\[
\varphi\bigl( \Graph g|_{[s,r]} \bigr) = \{ (g(x), g(x)^{\frac{1}{d}} - q g(x) - x) \mid x\in [s,r]\} = \Graph h,
\]
where $h: [g(r),r] \rightarrow \R$ is the function
\[
h(x) := x^{\frac{1}{d}} - q x - \tilde{f}(x).
\]
Note that 
\begin{equation}
\label{e:archi}
\mathrm{graph}\, \tilde{f} \subset W^s(0), \qquad
\mathrm{graph}\, h \subset W^u(0).
\end{equation}
As in the proof of Theorem \ref{t:omoclina1}, we denote by $x_* \in (s,r)$ the unique fixed point of $g$ in $(0,r]$. Since $g(r)<g(x_*)=x_*$, we have $x_* \in (g(r),r)$ and $\tilde{f}(x_*)=x_*$. The function $h-\tilde{f}$ is continuous on $[x_*,r]$ and analytic on $(x_*,r)$. By the identity $x^{\frac{1}{d}} - q x = g(x) + f(x)$ and the fact that $\tilde{f}(r)=s=f(r)$, we have
\[
\begin{split}
(h-\tilde{f})(x_*) &= g(x_*) + f(x_*) - 2 x_* = f(x_*) - x_* < 0, \\
(h-\tilde{f})(r) &= g(r) + f(r) - 2 s = g(r) - s > 0.
\end{split}
\]
Therefore, $h-\tilde{f}$ changes sign at some isolated zero $x_1\in (x_*,r)$. By \eqref{e:archi}, the point $z_1:= (x_1,h(x_1)) = (x_1,\tilde{f}(x_1))$ is a topologically transverse intersection point of $W^s(0)$ and $W^u(0)$, and the same is then true for $z_0:= \varphi^{-1}(z_1)$. Since
\[
z_0 = \varphi^{-1}(x_1,\tilde{f}(x_1)) = (x_1^{\frac{1}{d}} - q x_1 - \tilde{f}(x_1), x_1) = (h(x_1),x_1) = (\tilde{f}(x_1),x_1) = S(z_1),
\]
we have $\varphi(z_0) = z_1 = S(z_0)$. As in Case 1, this implies that the orbit $z_i:=\varphi^i(z_0)$ satisfies the symmetry condition \eqref{e:2sym}. Writing $z_i = (u(i-1),u(i))$ for every $i\in \Z$, we deduce that $u(-i)=u(i)$ for every $i\in \Z$, where $u(0)=x_1$ and $u(1)=\tilde{f}(x_1)$. Since $\tilde{f}$ is strictly decreasing and fixes $x_*$, we have
\[
s < \tilde{f}(x_1) < x_* < x_1.
\]
By Addendum \ref{a:more}, $\varphi(g(x),x) = (x,f(x))$ for every $x\in [0,r]$. Applying this identity to $x=\tilde{f}(x_1)$, for which $g(x)=x_1$, we find that
\[
z_2 = \varphi(z_1) = \bigl( \tilde{f}(x_1), f(\tilde{f}(x_1)) \bigr)
\]
belongs to the graph of $f$. The positive invariance of this graph then implies that $z_{i+1} = \bigl( f^{i-1}(\tilde{f}(x_1)), f^i(\tilde{f}(x_1)) \bigr)$ for every $i\geq 1$, that is, $u(i) = f^{i-1}(\tilde{f}(x_1))$ for every $i\geq 1$. Since $0<f(x)<x$ on $(0,r]$, we conclude that the orbit of $z_0$ is contained in $(0,+\infty)^2$ and that
\[
u(0) = x_1 > u(1) = \tilde{f}(x_1) > u(2) = f(\tilde{f}(x_1)) > \dots > u(i) = f^{i-1} (\tilde{f}(x_1)) > \dots
\]
Therefore, \eqref{e:dove} holds.

\medskip

In all the three cases, we have $u(i+1)=f(u(i))$ for $i\geq 1$ and hence 
\[
u(i)^d \leq u(i+1) \leq C u(i)^d \qquad \forall i\geq 1
\]
for a suitable positive number $C$. 
\end{proof}

Theorem \ref{t:omocline-main} from the Introduction summarizes the contents of Theorems \ref{t:stable}, \ref{t:omoclina1}, \ref{t:omoclina2} 
and is then completely proven, except for the last assertion about the values of the parameter $q$ such that $W^s(0)=W^u(0)$, which will be proven in Section \ref{s:homoconn}. 

Together with Proposition \ref{p:zachary}, Theorems \ref{t:omoclina1} and \ref{t:omoclina2}
have the following corollary, giving us Gibbs measures with the same symmetries as in Theorem \ref{t:2gibbs}/Theorem \ref{t:ex2-main} and more precise information. 

\begin{cor}
\label{c:info+}
If $Q=q\mathbbm{1}_{\{0\}} + \mathbbm{1}_{\{-1,1\}}$, then there exist Gibbs measures $\mu_1$ and $\mu_2$ in $\mathscr{G}(d,q\mathbbm{1}_{\{0\}} + \mathbbm{1}_{\{-1,1\}})$ such that $\iota_*(\mu_1)=\mu_1$, $\jmath_*(\mu_2)=\mu_2$ with single-site marginals $\pi_1$ and $\pi_2$ which satisfy
\[
\begin{split}
& \pi_1(0) > \pi_1(1)=\pi_1(-1) > \pi_1(2)=\pi_1(-2) > \dots > \pi_1(i)=\pi_1(-i) > \dots \\
& \pi_2(1)=\pi_2(0) > \pi_2(2)=\pi_2(-1) > \pi_2(3)=\pi_2(-2) > \dots > \pi_2(i)=\pi_2(1-i) > \dots
\end{split}
\]
and are infinitesimal of order $d$ and hence decay doubly exponentially for $i\rightarrow \pm \infty$, with
\[
\bigl(2+q^{\frac{d+1}{2}}\bigr)^{-2} \, w^{d-1}_j \, \pi_j(i)^d \leq \pi_j(i+1) \leq C \, \pi_j(i)^d \qquad \forall i\geq 1, \qquad j=1,2,
\]
where $w_j$ is the weight of $\mu_j$ and $C$ is a positive number. 
\end{cor}

\section{Coding of Gibbs measures}
\label{s:coding}

As mentioned in Section \ref{s:ifss}, even if $\varphi$ is not differentiable at the origin, this fixed point should be considered as a hyperbolic fixed point. Indeed, we have the following linearization result, which is proven in the next section.

\begin{thm}
\label{t:linearization}
There is a compact neighborhood $K$ of $0\in \R^2$ and an orientation preserving and $S$-equivariant homeomorphism $h: [-1,1]^2 \rightarrow K$ such that $h(0,0)=(0,0)$ and
\begin{equation}
\label{e:congiu}
\varphi \circ h (x,y) = h \Bigl( \frac{x}{2} , 2y\Bigr) \qquad \forall (x,y) \in [-1,1] \times \Bigl[ - \frac{1}{2}, \frac{1}{2} \Bigr]. 
\end{equation}
Moreover, $h([0,1]^2) \subset \{0\} \cup (0,+\infty)^2$ and
\begin{equation}
\label{e:more}
\begin{split}
h((0,1]\times \{0\} ) &\subset \Graph ( f|_{(0,r]} ), \\
h(\{0\} \times (0,1] ) &\subset S( \Graph ( f|_{(0,r]} )).
\end{split}
\end{equation}
\end{thm}

We denote by $K_+$ the positive sector in $K$, namely the set $K_+:= h([0,1]\times [0,1])$. Assuming that they do not coincide, the stable and unstable sets of the origin have a topologically transverse intersection at $z_0$ which is given by Theorem \ref{t:omoclina2}. This fact and Theorem \ref{t:linearization} will allow us to apply the general Theorem \ref{t:horseshoe} from Appendix \ref{a:th} on the existence of a topological horseshoe. This theorem, which is a refinement of a result of Burns and Weiss \cite{bw95}, gives us a positive integer $N$, disjoint compact sets $Q_0\subset K_+$ and $Q_1\subset (0,+\infty)^2$ with $z_0\in Q_1$ such that the compact $\varphi^N$-invariant set
\[
\Lambda := \bigcap_{j\in \Z} \varphi^{Nj}( Q_0 \cup Q_1) 
\]
is non-empty, and for every bi-infinite sequence $\omega\in \{0,1\}^{\Z}$ there exists at least one point $z\in \Lambda$ such that $\varphi^{Nj}(z)$ belongs to $Q_{\omega(j)}$ for every $j\in \Z$. Moreover, Theorem \ref{t:horseshoe} will also guarantee that, assuming that $\omega$ is not identically zero, the orbit of $z$ by $\varphi$ is contained in $(0,+\infty)\times (0,+\infty)$, and this is a homoclinic orbit whenever $\omega(j)=0$ for all but finitely many $j$. This will lead to the proof of the next result, which is precisely Theorem \ref{t:coding-main} from the Introduction.

\begin{thm}
\label{t:coding}
Let $d\geq 2$ be an integer and assume that $q\in [0,+\infty)$ does not belong to the (possibly empty) finite set of values for which $W^u(0)=W^s(0)$ (see 
Proposition \ref{p:homoconn}). Set $Q:= q\mathbbm{1}_{\{0\}} + \mathbbm{1}_{\{-1,1\}}$. Then there exists a positive integer $N$ and numbers $\theta>1$ and $\eta>0$ such that for every non-empty finite set $A\subset N\Z$ there exists a Gibbs measure $\mu_A$ in $\mathscr{G}(d,Q)$ whose single-site marginal $\pi_A$ satisfies
\begin{equation}
\label{e:favorisce}
\min_{i\in A} \pi_A(i) \geq \theta \, \max_{i\in \Z\setminus A} \pi_A(i),
\end{equation}
and whose weight $w_A$  satisfies $w_A \geq \eta |A|$, where $|A|$ denotes the cardinality of $A$.  If $A'$ is not a translated copy of $A$, then the Gibbs measures $\mu_A$ and $\mu_{A'}$ represent distinct elements in $\mathscr{G}(d,Q)/\Z$, which is then an infinite set.
\end{thm}

\begin{proof}
By Theorem \ref{t:omoclina2}, $W^s(0)$ and $W^u(0)$ have a topologically transverse intersection at a point $z_0\in (0,+\infty)^2$ such that, writing $\varphi^i(z_0) = (u(i-1),u(i))$, we have $u(-i)=u(i)$ for every $i\in \Z$ and the sequence $(u(i))_{i\geq 0}$ is positive and strictly decreasing. Therefore, we can find real numbers $a$ and $b$ such that
\[
\max_{i\in \Z\setminus \{0\}} u(i) = u(1)=u(-1) < a < b < u(0).
\]
We then have
\[
z_0\in U_1:= (0,+\infty) \times (b,+\infty), \qquad \varphi^i(z_0) \in V := (0,+\infty ) \times (0,a) \quad \forall i\in \Z\setminus \{0\}.
\]
Given the neighborhood of the origin $K$ from Theorem \ref{t:linearization} and denoting by $K_+\subset K$ the positive sector, the set
\[
U_0:= K_+ \cap ([0,a)\times [0,a)) \subset (0,+\infty)^2 \cup \{0\}
\]
is a relative neighborhood of the origin in $K_+$. By applying Theorem \ref{t:horseshoe} to the fixed point $0$ of the homeomorphism $\varphi$, the topologically transverse homoclinic point $z_0$, and the sets $U_0$, $U_1$, and $V$, we find a positive integer $N$ and disjoint compact sets 
\[
Q_0 \subset U_0 \subset [0,a) \times [0,a), \qquad Q_1 \subset U_1 = (0,+\infty) \times (b,+\infty),
\]
such that conclusions (i)-(iv) of Theorem \ref{t:horseshoe}
hold for the compact $\varphi^N$-invariant set
\[
\Lambda:= \bigcap_{j\in \Z} \varphi^{Nj} ( Q_0 \cup Q_1).
\]
Given any non-empty finite set $A\subset N\Z$, by (ii) we find $z\in \Lambda$ such that 
\begin{equation}
\label{e:cod1}
\varphi^{Nj}(z) \in Q_1\subset (0,+\infty)\times (b,+\infty) \mbox{ if } j \in \frac{1}{N} A, \qquad \varphi^{Nj}(z) \in Q_0 \mbox{ if } j \notin \frac{1}{N} A.
\end{equation}
By (iv), $z$ is a homoclinic point of $\varphi$. By (iii),
\begin{equation}
\label{e:cod2}
\varphi^i(z) \in (V \cup U_0) \setminus \{0\} \subset (0,+\infty) \times (0,a) \qquad \forall i\in \Z\setminus A.
\end{equation}
We have $\varphi^i(z) = (v(i-1),v(i))$ for every $i\in \Z$, where $v\in \ell^p(\Z)$ is a positive solution of \eqref{e:2ndrec}. By \eqref{e:cod1} and \eqref{e:cod2},
\begin{equation}
\label{e:v}
v(i) > b \quad \forall i\in A, \qquad v(i) < a \quad \forall i\in \Z\setminus A.
\end{equation}
Denoting by $\mu_A\in \mathscr{G}(d,Q)$ the Gibbs measure corresponding to $v$ via Proposition \ref{p:zachary} and by $\pi_A$ its single-site marginal, we have
\[
\pi_A(i) = \frac{v(i)^p}{\|v\|_p^p} \left\{ \begin{array}{ll} > \frac{b^p}{\|v\|_p^p} & \mbox{ if } i\in A, \\ < \frac{a^p}{\|v\|_p^p} & \mbox{ if } i\in \Z\setminus A. \end{array} \right.
\]
Therefore
\[
\frac{\pi_A(j)}{\pi_A(i)} > \left( \frac{b}{a} \right)^p \qquad \forall j\in A, \; \forall i\in \Z\setminus A,
\]
and \eqref{e:favorisce} holds with $\theta:= (\frac{b}{a})^p>1$. Moreover, the weight $w_A$ of $\mu_A$ satisfies
\[
w_A = \bigl(2+q^{\frac{d+1}{2}}\bigr)^{\frac{2}{d-1}} \|v\|_p^p \geq \bigl(2+q^{\frac{d+1}{2}}\bigr)^{\frac{2}{d-1}} |A| b^p,
\]
and the bound $w_A \geq \eta |A|$ holds with $\eta:= (2 + q^{\frac{d+1}{2}})^{\frac{2}{d-1}} b^p$. Finally, if $A'$ is not a translated copy of $A$, then by \eqref{e:v} the solutions $v_A$ and $v_{A'}$ of \eqref{e:2ndrec} corresponding to $\mu_A$ and $\mu_{A'}$ are not related by a translation, and hence $\mu_A$ and $\mu_{A'}$ determine distinct elements in $\mathscr{G}(d,Q)/\Z$.
\end{proof}

\section{Proof of the linearization theorem}

The aim of this section is to prove Theorem \ref{t:linearization}. We will deduce it from a global conjugacy result for a general class of homeomorphisms of the plane.

Given $\theta>0$, we denote by $\Psi_{\theta}$ the set of maps $\psi: \R^2 \rightarrow \R^2$ of the form
\begin{equation}
\label{e:lin1}
\psi(x,y) = \bigl( y, \tau(y) - x \bigr),
\end{equation}
where $\tau:\R \rightarrow \R$ is a homeomorphism whose inverse $\sigma: \R \rightarrow \R$ is $\theta$-Lipschitz continuous and such that the function $x \mapsto \sigma(x) - \theta x$ is bounded.

Note that by \eqref{e:lin1} the map $\psi$ is a homeomorphism with inverse
\begin{equation}
\label{e:lainve}
\psi^{-1}(x,y) = S \circ \psi \circ S(x,y) =   \bigl( \tau(x) - y,x \bigr).
\end{equation}
Here, $S$ denotes as before the involution $(x,y)\mapsto (y,x)$.
Note also that the second assumption on $\sigma=\tau^{-1}$ implies that the function $y\mapsto \theta \,\tau(y) - y$ is bounded.

\begin{thm}
\label{t:lin-gen}
Assume that $\psi_1,\psi_2\in \Psi_{\theta}$ with $0< \theta < \frac{1}{2}$. Then there exists a unique continuous map $h: \R^2 \rightarrow \R^2$ which is a bounded perturbation of the identity (that is, $h-\mathrm{id}$ is a bounded map) and satisfies
\begin{equation}
\label{e:lin2}
\psi_1 \circ h = h \circ \psi_2.
\end{equation}
Moreover, $h$ is an orientation preserving $S$-equivariant homeomorphism.
\end{thm}

\begin{proof}
For $j\in \{1,2\}$, we set
\[
\psi_j(x,y) = \bigl( y, \tau_j(y) - x \bigr), \qquad \sigma_j:= \tau_j^{-1}.
\]
Writing $h$ componentwise as $h=(h_1,h_2)$, \eqref{e:lin2} is equivalent to
\[
h_2 = h_1\circ \psi_2, \qquad \tau_1 \circ h_2 - h_1 = h_2 \circ \psi_2.
\]
and hence to
\begin{eqnarray}
\label{e:lin3}
h_1 &=& h_2\circ \psi_2^{-1}, \\
\label{e:lin4}
\tau_1 \circ h_2 &=& h_2 \circ \psi_2 + h_2 \circ \psi_2^{-1}.
\end{eqnarray}
Writing
\[
h_2(x,y) = y + k_2(x,y),
\]
where $k_2$ belongs to the space $C_b(\R^2)$ of continuous and bounded real valued functions on $\R^2$, \eqref{e:lin4} can be rewritten as
\[
\tau_1(y + k_2(x,y)) = \tau_2(y) - x + k_2 \circ \psi_2(x,y) + x + k_2 \circ \psi_2^{-1}(x,y),
\]
which is equivalent to
\[
k_2(x,y) = \sigma_1 \bigl( \tau_2(y)  + k_2 \circ \psi_2(x,y) + k_2 \circ \psi_2^{-1}(x,y) \bigr) - y.
\]
Therefore, \eqref{e:lin4} is equivalent to the assertion that $k_2$ is a fixed point of the map $F: C_b(\R^2) \rightarrow C_b(\R^2)$ mapping $k\in C_b(\R^2)$ to the function
\[
(x,y)\mapsto \sigma_1 \bigl( \tau_2(y)  + k \circ \psi_2(x,y) + k \circ \psi_2^{-1}(x,y) \bigr) - y.
\]
The map $F$ clearly preserves continuity, and it also preserves also boundedness because the bounds
\[
|\sigma_1(x) - \theta x | \leq c_1, \qquad |\theta \, \tau_2(y) - y | \leq c_2, \qquad \forall x,y\in \R,
\]
imply
\[
|F(k)(x,y)| \leq c_1 + |\theta\, \tau_2(y) + \theta \, k\circ \psi_2(x,y) + \theta \, k \circ \psi_2^{-1}(x,y) - y | \leq c_1+c_2+2\theta\|k\|_\infty,
\]
for every $(x,y)\in \R^2$. This proves that $F$ maps $C_b(\R^2)$ to $C_b(\R^2)$. Moreover, using the fact that $\sigma_1$ is $\theta$-Lipschitz continuous we find
\[
\|F(k) - F(k')\|_{\infty} \leq 2\theta \|k - k\|_{\infty} , \qquad \forall k,k'\in C_b(\R^2),
\]
and the assumption $\theta< \frac{1}{2}$ guarantees that $F$ is a contraction on the Banach space $(C_b(\R^2),\|\cdot\|_{\infty})$. Therefore, $F$ has a unique fixed point $k_2\in C_b(\R^2)$, and hence \eqref{e:lin4} is satisfied by a unique continuous function $h_2: \R^2 \rightarrow \R$ which is a bounded perturbation of the projection $(x,y) \mapsto y$. Together with \eqref{e:lainve}, this implies that \eqref{e:lin3} is uniquely solved by the continuous function
\[
h_1(x,y) = h_2( \tau_2(x)-y,x),
\]
which is a bounded perturbation of the projection $(x,y)\mapsto x$. We conclude that the conjugacy equation \eqref{e:lin2} has a unique solution $h: \R^2 \rightarrow \R^2$ which is continuous and is a bounded perturbation of the identity.

By inverting the roles of $\psi_1$ and $\psi_2$, we find a continuous map $\tilde{h}: \R^2 \rightarrow \R^2$ which is a bounded perturbation of the identity and satisfies
\begin{equation}
\label{e:lin5}
\psi_2 \circ \tilde{h} = \tilde{h} \circ \psi_1.
\end{equation}
From \eqref{e:lin2} and \eqref{e:lin5}, we deduce that the continuous map $h\circ \tilde{h}$, which is a bounded perturbation of the identity, satisfies
\[
\psi_1 \circ h \circ \tilde{h} = h \circ \tilde{h} \circ \psi_1.
\]
By the uniqueness part of the previous statement, we conclude that $h \circ \tilde{h}= \mathrm{id}$. Similarly, $\tilde{h}\circ h = \mathrm{id}$, and hence $h$ is a homeomorphism. Being a bounded perturbation of the identity, $h$ is orientation preserving.

Finally, \eqref{e:lin2} implies the identity $\psi_1^{-1} \circ h = h \circ \psi_2^{-1}$ which, together with \eqref{e:lainve} implies that the map $\widehat{h} := S \circ h \circ S$ satisfies
\[
\psi_1 \circ \widehat{h} = \psi_1 \circ S \circ h \circ S = S\circ \psi_1^{-1} \circ h \circ S = S \circ h \circ \psi_2^{-1} \circ S = S \circ h \circ S \circ \psi_2 = \widehat{h} \circ \psi_2.
\]
Since $\widehat{h}$ is a bounded perturbation of the identity, uniqueness implies that $\widehat{h}=h$ and hence $h\circ S = S \circ h$.
\end{proof}

\begin{rem}
{\rm We shall not need this fact here, but for the sake of completeness we observe that the homeomorphisms in the set $\Psi_{\theta}$ with $0< \theta < \frac{1}{2}$ are all topologically conjugate to the linear automorphism $(x,y) \mapsto (\frac{x}{2},2y)$ by a global conjugacy. Indeed, Theorem \ref{t:lin-gen} gives us a global conjugacy between $\psi\in \Psi_{\theta}$ and the linear automorphism $(x,y) \mapsto (y,\frac{y}{\theta} - x)$, which belongs to $\Psi_{\theta}$ as well. This automorphism is diagonalizable on the reals with eigenvalues $\lambda$ and $\frac{1}{\lambda}$, where $\lambda>1$, and hence it is linearly conjugate to the automorphism $(x,y) \mapsto ( \frac{x}{\lambda}, \lambda y)$. The latter automorphism is topologically conjugate to the  linear automorphism $(x,y) \mapsto (\frac{x}{2},2y)$ by the conjugacy
\[
(x,y) \mapsto (|x|^{\alpha-1} x, |y|^{\alpha-1} y),
\]
where $\alpha:= \log_2 \lambda$. All the conjugacies involved here are $S$-equivariant.
}
\end{rem}

We conclude this section by proving Theorem \ref{t:linearization}. The map $\varphi$ from \eqref{e:mapphi} has the form
\[
\varphi(x,y) = \bigl( y, \zeta(y) - x \bigr),
\]
where $\zeta: \R \rightarrow \R$ is the continuous odd function
\[
\zeta(y) := |y|^{\frac{1}{d}-1} y - q y.
\]
This function satisfies
\[
\zeta'(y) \geq \frac{5}{2} \qquad \forall y\in [-y_0,y_0], \qquad \zeta'(\pm y_0) = \frac{5}{2},
\]
where 
\[
y_0:= \left( d \Bigl( \frac{5}{2} + q \Bigr) \right)^{\frac{d}{1-d}} \in (0,c) \subset (0,r),
\]
see \eqref{c} and Theorem \ref{t:stable}, and we agree that $\zeta'(0)=+\infty> \frac{5}{2}$. Let $\tau: \R \rightarrow \R$ be the unique continuous odd function such that $\tau=\zeta$ on $[-y_0,y_0]$ and $\tau$ is affine linear with slope $\frac{5}{2}$ on $(-\infty,-y_0]$ and on $[y_0,+\infty)$. Then $\tau$ is a homeomorphism and its inverse $\sigma$ is $\frac{2}{5}$-Lipschitz continuous and such that the function $x\mapsto \sigma(x) - \frac{2}{5} x$ is bounded. Therefore, the homeomorphism
\[
\psi: \R^2 \rightarrow \R^2, \qquad \psi(x,y) := \bigl( y , \tau(y) - x \bigr),
\]
belongs to $\Psi_{\frac{2}{5}}$. By Theorem \ref{t:lin-gen}, there exists an orientation preserving homeomorphism $\tilde{h}: \R^2 \rightarrow \R^2$ which is a bounded perturbation of the identity such that
\[
\psi \circ \tilde{h} = \tilde{h} \circ L,
\]
where $L\in \Psi_{\frac{2}{5}}$ is the linear automorphism
\[
L(x,y) := \Bigl( y, \frac{5}{2} y - x \Bigr), \qquad \forall (x,y) \in \R^2.
\]
This linear automorphism of $\R^2$ has eigenvectors $v:= (2,1)$ and $w:=(1,2)$ with
\[
L(v) = \frac{1}{2} v, \qquad L(w) = 2 w.
\]
Therefore, the stable and unstable sets of $0\in \R^2$ with respect to $L$ are the lines $\R v$ and $\R w$, respectively. Let $f:[-r,r]\rightarrow \R$ be the function from Theorem \ref{t:stable}, whose graph generates, by the action of positive powers of $\varphi^{-1}$, the stable set of $0$ with respect to $\varphi$. Note that $0<y_0<r$ and the graph of $f|_{[-y_0,y_0]}$ is contained in $[-y_0,y_0]\times [-y_0,y_0]$. Since $\psi=\varphi$ on $\R\times [-y_0,y_0]$, we have
\[ 
W^s(0;\psi) = \bigcup_{n\geq 0} \psi^{-n} \Bigl( \Graph ( f|_{[-y_0,y_0]} ) \Bigr).
\]
Being a conjugacy, $\tilde{h}$ maps $\R v$ onto $W^s(0;\psi)$ and $(0,+\infty)v$ onto either
\[
W^s_+ (0;\psi) := \bigcup_{n\geq 0} \psi^{-n} \Bigl( \Graph ( f|_{(0,y_0]} ) \Bigr),
\]
or 
\[
W^s_- (0;\psi) := \bigcup_{n\geq 0} \psi^{-n} \Bigl( \Graph ( f|_{[-y_0,0)} ) \Bigr).
\]
We claim that
\begin{equation}
\label{e:lato-giusto}
\tilde{h}\bigl( (0,+\infty) v \bigr) = W^s_+ (0;\psi).
\end{equation}
Indeed, assume for contradiction that
\begin{equation}
\label{e:lato-sbagliato}
\tilde{h}\bigl( (0,+\infty) v \bigr) = W^s_- (0;\psi).
\end{equation}
Since $\tilde{h}$ is a bounded perturbation of the identity, \eqref{e:lato-sbagliato} and the fact that $v$ has positive components imply that $W^s_- (0;\psi)$ contains points in the first quadrant $(0,+\infty)^2$. Since $f<0$ on $[-y_0,0)$,  $W^s_- (0;\psi)$ contains also points in the third quadrant $(-\infty,0)^2$. Therefore, the negative branch $W^s_- (0;\psi)$ meets the line $y=-x$ at some point $(x_0,-x_0)$ with $x_0\neq 0$. Since $\psi$ is an odd map, we have $W^s_+(0;\psi) = - W^s_-(0;\psi)$ and hence $W^s(0;\psi)$ contains also the point $(-x_0,x_0)$. Since $\psi^{-1} = S \circ \psi \circ S$, we have $W^u(0;\psi) = S(W^s(0;\psi))$, and we deduce that the points $(x_0,-x_0)$ and $(-x_0,x_0)$ belong to the intersection $W^s(0;\psi) \cap W^u(0;\psi)$. This is a contradiction because, being the images of $\R v$ and $\R w$ by the homeomorphism $\tilde{h}$, $W^s(0;\psi)$ and $W^u(0;\psi)$ meet only at $0$. This contradiction proves \eqref{e:lato-giusto}.

Using the $S$-equivariance of $\tilde{h}$, we find
\[
\tilde{h}\bigl( (0,+\infty) w \bigr) = W^u_+ (0;\psi) := \bigcup_{n\geq 0} \psi^{n} \Bigl( S(\Graph ( f|_{(0,y_0]} )) \Bigr) = S(W^s_+ (0;\psi) ).
\]
The linear automorphism $G:\R^2 \rightarrow \R^2$ mapping the vector $(1,0)$ to $v$ and the vector $(0,1)$ to $w$ is orientation preserving, $S$-equivariant, and satisfies
\[
L \circ G (x,y) = G\Bigl( \frac{x}{2} , 2y \Bigr), \qquad \forall (x,y)\in \R^2.
\]
The orientation preserving and $S$-equivariant homeomorphism $\widehat{h}:= \tilde{h} \circ G: \R^2 \rightarrow \R^2$ satisfies
\begin{eqnarray}
\label{e:lin6}
\psi \circ \widehat{h} (x,y) = \widehat{h} \Bigl( \frac{x}{2} , 2y \Bigr), \qquad \forall (x,y)\in \R^2, \\
\label{e:lin7}
\widehat{h}((0,+\infty)\times \{0\}) = W^s_+(0;\psi), \\
\label{e:lin8}
\widehat{h}(\{0\} \times (0,+\infty)) = W^u_+(0;\psi).
\end{eqnarray}
Since $\psi=\varphi$ on $\R\times [-y_0,y_0]$,  \eqref{e:lin6} implies
\[
\varphi \circ \widehat{h} (x,y) = \widehat{h} \Bigl( \frac{x}{2} , 2y \Bigr),
\]
for every $(x,y)$ in the set $\widehat{h}^{-1}(\R \times  [-y_0,y_0] )$.
Being a neighborhood of $0$, the latter set contains a set of the form $[-\delta,\delta] \times [-\frac{\delta}{2}, \frac{\delta}{2} ]$, for $\delta>0$ small enough. Setting
\[
h: [-1,1]^2 \rightarrow \R^2, \qquad h(x,y) := \widehat{h}(\delta x, \delta y), \qquad K:= h([-1,1]^2) = \widehat{h}([-\delta,\delta]^2),
\]
\eqref{e:lin6} implies that the homeomorphism $h: [-1,1]^2 \rightarrow K$, which is
orientation preserving and $S$-equivariant,  satisfies \eqref{e:congiu}. 

Moreover, up to the choice of a smaller $\delta$, \eqref{e:more} follows from \eqref{e:lin7} and \eqref{e:lin8}, and we may assume that $K\subset (-y_0,y_0)^2$. There remains to prove that the connected set $D:= h((0,1]^2)$ is contained in $(0,+\infty)^2$. The graph of $f|_{[-y_0,y_0]}$ is contained in $W^s(0;\psi) = \widehat{h}(\R\times \{0\})$, and its image by $S$ is contained in $W^u(0;\psi) = \widehat{h}(\{0\}\times \R)$. Since $\widehat{h}$ is injective, $D$ is disjoint from these two sets, and hence the continuous functions $(x,y)\mapsto x-f(y)$ and $(x,y)\mapsto y-f(x)$ do not vanish on $D$. By \eqref{e:more}, the closure of $D$ contains points of the form $(x,f(x))$ with $x>0$, at which the first function takes the positive value $x-f(f(x))$, and points of the form $(f(y),y)$ with $y>0$, at which the second function takes the positive value $y-f(f(y))$. Being connected, $D$ is then contained in the set
\[
\{(x,y)\in (-y_0,y_0)^2 \mid x>f(y), \; y> f(x) \}.
\]
This set is contained in $(0,+\infty)^2$. Indeed, if one of its points $(x,y)$ satisfies $x\leq 0$, then $f(y)<x\leq 0$ implies $y<0$, and from $f(x)<y<0$ and $f(y)<x\leq 0$ we obtain $|y|<|f(x)|\leq |x|<|f(y)|\leq |y|$, which is impossible. Therefore $x>0$, and $y>0$ by symmetry. This concludes the proof of Theorem \ref{t:linearization}.

\section{Homoclinic connection}
\label{s:homoconn}

We now discuss the scenario
\begin{equation}
\label{e:homoconn}
W^u(0) = W^s(0),
\end{equation}
which is excluded in the assumptions of Theorem \ref{t:coding-main}/Theorem \ref{t:coding}. We wish to show that \eqref{e:homoconn} can hold for at most finitely many values of $q$. 

We use the notation of Sections \ref{s:ifss} and \ref{s:duesol}. Moreover, we denote by
\[
W^s_+(0) := \bigcup_{n\geq 0} \varphi^{-n} \bigl( \Graph f|_{[0,r]} \bigr), \qquad W^u_+(0) := \bigcup_{n\geq 0} \varphi^{n} \bigl( S (\Graph f|_{[0,r]}) \bigr),
\]
the positive branches of the stable and unstable sets. The next lemma is well known in the smooth setting.

\begin{lem}
\label{l:transverse-distinct}
Assume that $W^s_+(0)$ and $W^u_+(0)$ have a topologically transverse intersection at a point $z\neq 0$. Then $W^s(0)\neq W^u(0)$.
\end{lem}

\begin{proof}
We apply Theorem \ref{t:horseshoe} to the homoclinic point $z$ of $\varphi$, with a neighborhood $U_1$ of $z$ not having $0$ in its closure and with $U_0=K_+$, $V=\R^2$.
By (ii), there is $p\in \Lambda$ which is coded by the bi-infinite sequence $\omega_p(j)=0$ for $j\leq 0$, $\omega_p(j)=1$ for $j\geq 1$. The proof of (iv) shows that the backward orbit of $p$ converges to $0$. Since $\overline{U_1}$ is disjoint from $0$, the  forward orbit of $p$ does not converge to $0$. Hence $p\in W^u(0)\setminus W^s(0)$.
\end{proof}

In the next two lemmata and in the following proposition, we make the dependence on $q$ explicit by writing $F_q$, $f_q$, $g_q$, and so on. The convergence radius of the series 
\[
F_q(x) = \sum_{n\geq 0} a_n(q) x^n
\]
from Lemma \ref{l:formal} is denoted by $r(q)$. Since each $a_n$ is a non-decreasing function of $q$, the function $r$ is non-increasing.

\begin{lem}
\label{l:r-continuous}
The function $q\mapsto r(q)$ is continuous on $[0,+\infty)$.
\end{lem}

\begin{proof}
Let $q_0>0$ and set
\begin{equation}
\label{e:r-}
r_-(q_0):=\lim_{q\uparrow q_0} r(q) = \inf_{q\in [0,q_0)} r(q)  \geq r(q_0). 
\end{equation}
For $q\in [0,q_0)$ and  $x\in (0,r_-(q_0))\subset (0,r(q))$ we have $F_q(x)<x$, so by monotone convergence $F_{q_0}(x) = \sup_{q<q_0} F_q(x) \leq x$. Therefore, $F_{q_0}$ converges on $[0,r_-(q_0))$ and hence $r_-(q_0)\leq r(q_0)$. Together with \eqref{e:r-}, this implies that $r_-(q_0)=r(q_0)$ and hence $r$ is left-continuous at $q_0$.

Now let $q_0\in [0,+\infty)$ and denote by $r_+(q_0)$ the right limit of $r$ at $q_0$, so that $r(q)\leq r_+(q_0)\leq r(q_0)$ for every $q>q_0$. Assume by contradiction that $r_+(q_0)<r(q_0)$. Since $g_{q_0}>\mathrm{id}$ on $(0,s(q_0)]$ and $g_{q_0}(s(q_0))=r(q_0)$, we can find $y_1\in [0,r_+(q_0))$ such that $g_{q_0}(y_1)>r_+(q_0)$. For $q>q_0$ close enough to $q_0$ we have $y_1<r(q)$, and since $0\leq a_k(q)\leq a_k(q')$ for $q_0\leq q\leq q'$, dominated convergence gives $F_q(y_1)\rightarrow F_{q_0}(y_1)$, and hence $g_q(y_1)\rightarrow g_{q_0}(y_1)$, for $q\downarrow q_0$. Therefore, $g_q(y_1)>r_+(q_0) \geq r(q)$ for some $q>q_0$, contradicting the fact that $r(q)$ is the maximum of $g_q$ on $[0,r(q)]$ (Addendum \ref{a:more}).
\end{proof}

Recall from the proof of Theorem \ref{t:omoclina1} that $g_q$ has a unique fixed point $x_*(q)$ in $(0,r(q))$.

\begin{lem}
\label{l:analytic}
The functions $q\mapsto x_*(q)$ and $q\mapsto g_q'(x_*(q))$ are analytic on $[0,+\infty)$, meaning that they extend to analytic functions on an open interval containing $[0,+\infty)$.
\end{lem}

\begin{proof}
Fix $q_0\geq 0$. For $q\in \C$ with $|q-q_0|<\rho$ we have $|a_k(q)|\leq a_k(|q|) \leq a_k(q_0+\rho)$, so $(q,x)\mapsto F_q(x)$ is analytic on $\{|q-q_0|<\rho\}\times \{|x|<r(q_0+\rho)\}$. By Lemma \ref{l:r-continuous}, $r(q_0+\rho)>x_*(q_0)$ for $\rho$ small enough. Since $x_*(q_0)>0$, the function $(q,x)\mapsto g_q(x)-x$ is then analytic in a neighborhood of $(q_0,x_*(q_0))$, and its partial derivative in $x$ at this point is $g_{q_0}'(x_*(q_0))-1<0$. By the implicit function theorem, the equation $g_q(x)=x$ has a unique solution $x$ near $x_*(q_0)$, depending analytically on $q$ near $q_0$. For $q\geq 0$ this solution belongs to $(0,r(q))$, by Lemma \ref{l:r-continuous}, and hence it is $x_*(q)$, the unique fixed point of $g_q$ in $(0,r(q)]$. The analyticity of $q\mapsto g_q'(x_*(q))$ follows.
\end{proof}

We can now show that \eqref{e:homoconn} can occur for at most finitely many values of $q$.

\begin{prop}
\label{p:homoconn}
For every $d\geq 2$ the (possibly empty) set of $q\in [0,+\infty)$ for which $W^u(0)=W^s(0)$ is finite and contained in the interval $[0,\frac{d+2}{d-1})$.
\end{prop}

\begin{proof}
By Lemma \ref{l:transverse-distinct}, a necessary condition for \eqref{e:homoconn} to hold is that the intersection of $W^s_+(0)$ and $W^u_+(0)= S(W^s_+(0))$ at $(x_*(q),x_*(q))$ is not transverse, meaning that $g_q'(x_*(q))=-1$, see the proof of Theorem \ref{t:omoclina1}. Since $x_*(q) > c_q = (q+2)^{-\frac{d}{d-1}}$ and $f'_q\geq 0$, we have the bound
\[
g'_q(x_*(q))= \frac{1}{d} x_*(q)^{\frac{1}{d}-1} - q - f'_q(x_*(q)) <  \frac{1}{d} c_q^{\frac{1}{d}-1} - q = \frac{1}{d} (q+2) - q.
\]
We deduce that if $q\geq \frac{d+2}{d-1}$ then $g'_q(x_*(q))<-1$, and hence the analytic function $q\mapsto g'_q(x_*(q))$ can take the value $-1$ only at finitely many points which are contained in the interval $[0,\frac{d+2}{d-1})$. The set of values $q$ for which \eqref{e:homoconn} holds is a subset of this finite set.
\end{proof}

The above proposition concludes the proof of Theorem \ref{t:omocline-main} from the Introduction.

\medskip

We do not know whether \eqref{e:homoconn} can actually hold for some values of $q$. We conclude this section by a brief discussion on this.

One can show that \eqref{e:homoconn} is equivalent to the fact that $g$ restricts to an involution on the interval $[s,r]$, that is, $g(r)=s$ and $g\circ g=\mathrm{id}$ on $[s,r]$. In this case, $W^s(0)\setminus \{0\} = W^u(0)\setminus \{0\}$ consists of two homoclinic connections, contained in $(0,+\infty)^2$ and $(-\infty,0)^2$. Proposition \ref{p:gibbs=homo} gives us a bijection between $\mathscr{G}(d,Q)$ and the first of these two homoclinic connections, which is homeomorphic to $\R$. Therefore, $\mathscr{G}(d,Q)/\Z$ is in bijection with a circle.

 For $d=2$, numerical investigations suggest that there are exactly four values of $q$ for which the intersection of $W^s(0)$ with $W^u(0)$ at $z_0=S(z_0)$ from
Theorem \ref{t:omoclina1} is not transverse ($q_1\simeq 0.105272120$, $q_2 \simeq 0.584414384$, $q_3 \simeq 1.178272946$, $q_4 \simeq 1.801948414$). There are also exactly four values of $q$ for which $g(r)=s$, corresponding to Case 2 in the proof of Theorem \ref{t:omoclina2} in which we have a non-transverse homoclinic intersection at $(s,r)$, and these four values are extremely close to the other ones (difference of order of $10^{-7}$). For such values of $q$, $g$ is very close to be an involution on $[s,r]$ (the supremum norm of $g\circ g - \mathrm{id}$ on that interval is of the order between $10^{-11}$ and $10^{-10}$). Even if the two quadruples of values of $q$ are very close, these numerics suggest that they do not coincide, thus excluding the possibility that \eqref{e:homoconn} can hold for some values of $q$. 

For higher values of $d$ there is a similar scenario, but the number of values of $q$ in which we see tangencies drops to three for $d=3$ and to two for all tested values of $d\geq 4$. The pairs of values of $q$ for which we have tangencies at the homoclinic points coming from either Theorem \ref{t:omoclina1} or \ref{t:omoclina2} separate as $d$ grows, making the case for the absence of homoclinic connections stronger for larger values of $d$.

Interval arithmetic is probably needed in order to confirm (or disprove) the above findings, and to decide whether the homoclinic connection scenario \eqref{e:homoconn} can ever occur.

\section{Interaction functions with bounded support}
\label{s:ifbs}

The equivalence between Gibbs measures in $\mathscr{G}(d,Q)$ and homoclinic points of a (somewhere non-smooth) symplectomorphism which we saw in Section \ref{s:ifss} for $Q$ supported in $\{-1,0,1\}$ generalizes to interaction functions $Q$ with arbitrary bounded support. In order to explain how this goes, we need to recall some facts from the theory of discrete Lagrangian systems. See \cite{ves91} and \cite{mawe01} for the proofs of the facts mentioned below and further references.

Given an integer $n\geq 1$, let $L:\R^{n+1} \rightarrow \R$, $L=L(x_0,x_1,\dots,x_n)$, be a smooth function, which we refer to as \textit{Lagrangian}. The  \textit{formal action} of a bi-infinite sequence $u: \Z \rightarrow \R$ is the expression
\[
\mathbb{A}(u) := \sum_{k\in \Z} L \bigl( u(k),u(k+1),\dots,u(k+n) \bigr).
\]
Although this is just a formal expression, because the above sum need not converge, the difference $\mathbb{A}(v)- \mathbb{A}(u)$ is a well-defined real number whenever $v-u$ has bounded support. A map $u: \Z \rightarrow \R$ is said to be \textit{stationary} for $\mathbb{A}$ if $0$ is a critical point of the real valued function
\[
s\mapsto \mathbb{A}(u+sw) - \mathbb{A}(u)
\]
for all $w: \Z \rightarrow \R$ with bounded support. The map $u$ is stationary for $\mathbb{A}$ if, and only if, it satisfies the following \textit{Euler--Lagrange equation}:
\begin{equation}
\label{e:EL}
\sum_{j=0}^n \partial_{x_j} L\bigl( u(k-j), \dots, u(k-j+n) \bigr) = 0 \qquad \forall k\in \Z.
\end{equation}
Assume that for every $(x_0,x_1,\dots,x_n)\in \R^{n+1}$ the functions
\begin{equation}
\label{e:twist}
t\mapsto \partial_{x_0} L (x_0,\dots,x_{n-1},t) \qquad \mbox{and} \qquad t\mapsto \partial_{x_n} L (t,x_1,\dots,x_n)
\end{equation}
are diffeomorphisms of $\R$. The assumption that the first function in \eqref{e:twist} is a diffeomorphism implies that \eqref{e:EL} can be rewritten as a scalar recursive equation of order $2n$ of the form
\begin{equation}
\label{e:rec}
u(k+n) = F\bigl( u(k-n) \dots, u(k+n-1) \bigr) \qquad \forall k\in \Z,
\end{equation}
for a suitable smooth function $F: \R^{2n} \rightarrow \R$. By the usual equivalence between scalar recursive equations of order $2n$ and vectorial recursive equations of order 1 in $\R^{2n}$, $u$ solves \eqref{e:rec} if, and only if, the function $w: \Z \rightarrow \R^{2n}$ defined as
\begin{equation}
\label{e:theorbit}
w(k) := \bigl( u(k), u(k+1), \dots, u(k+2n-1) \bigr)
\end{equation}
satisfies the recursion
\[
w(k+1) = \varphi(w(k)) \qquad \forall k\in \Z,
\]
where $\varphi: \R^{2n} \rightarrow \R^{2n}$ is the smooth map
\begin{equation}
\label{e:varphi}
\varphi(w_1,\dots,w_{2n}) := \bigl(w_2,\dots,w_{2n},F(w_1,\dots,w_{2n}) \bigr).
\end{equation}
The assumption that the second function in \eqref{e:twist} is a diffeomorphism implies that $\varphi$ is a diffeomorphism. Furthermore, 
the smooth two-form $\omega = \mathrm{d}\lambda$ which is obtained by differentiating the \textit{Poincar\'e--Cartan one-form}
\begin{equation}
\label{e:PC}
\lambda:= \sum_{i=1}^n p_i(w) \, \mathrm{d} w_{n+i}, \qquad p_i(w_1,\dots,w_{2n}) := \sum_{j=i}^{n} \partial_{x_{n+i-j}} L (w_j,\dots,w_{j+n}),
\end{equation}
can be shown to be a symplectic form on $\R^{2n}$, and $\varphi$ is a symplectomorphism on $(\R^{2n},\omega)$. Indeed, $\varphi$ satisfies
\[
\varphi^* \lambda - \lambda = \mathrm{d} \ell, \qquad \mbox{where} \quad \ell(w_1,\dots,w_{2n}) := L(w_{n+1},\dots,w_{2n},F(w_1,\dots,w_{2n})).
\]
We conclude that there is a one-to-one correspondence between solutions of the Euler-Lagrange equation \eqref{e:EL} up to translation and orbits of the symplectomorphism $\varphi$ of $(\R^{2n},\omega)$.

We now fix an interaction function $Q$ which is supported in $\{-n,\dots,n\}$, with $Q(n)=Q(-n)>0$, and
apply the above considerations to the Lagrangian
\begin{equation}
\label{e:lagrangian}
L(x_0,x_1,\dots,x_n) := \frac{1}{2} Q(0) x_0^2 + \sum_{i=1}^n Q(i) x_0 x_i - \frac{d}{d+1} |x_0|^{\frac{d+1}{d}}.
\end{equation}
This Lagrangian is smooth on $\R^{n+1}\setminus \{x_0=0\}$ but just $C^1$ on $\R^{n+1}$, and the above considerations require some adjustments.
Since
\begin{equation}
\label{e:derivatives}
\begin{split}
\partial_{x_0} L(x_0,\dots,x_n) &= \sum_{i=0}^n Q(i) x_i - |x_0|^{\frac{1}{d}-1} x_0, \\
\partial_{x_j} L(x_0,\dots,x_n) &= Q(j) x_0 \quad \mbox{for } 1\leq j \leq n,
\end{split}
\end{equation}
the Euler-Lagrange equation \eqref{e:EL} reads
\[
\sum_{i=0}^n Q(i) u (k+i) + \sum_{j=1}^n Q(j) u(k-j) -  |u(k)|^{\frac{1}{d}-1} u(k) = 0\qquad \forall k\in \Z,
\] 
which, thanks to the fact that $Q$ is even, can be rewritten as
\begin{equation}
\label{e:consign}
Q * u = |u|^{\frac{1}{d}-1} u.
\end{equation}
Therefore, the positive solutions of the Euler-Lagrange equation are precisely the positive solutions of the equation
\begin{equation}
\label{e:nosign}
u = (Q*u)^d.
\end{equation}
The formal action now takes the form
\[
\mathbb{A}(u) = \frac{1}{2} \sum_{k\in \Z} u(k) (Q*u)(k) - \frac{d}{d+1} \sum_{k\in \Z} |u(k)|^{\frac{d+1}{d}},
\]
and its restriction to $\ell^{\frac{d+1}{d}}(\Z)$ is minus the functional $\mathscr{E}$ we considered in Section \ref{s:ex2sol}. 

By \eqref{e:derivatives}, the functions $t\mapsto \partial_{x_0} L(x_0,\dots,x_{n-1},t)$ and $t\mapsto \partial_{x_n} L(t,x_1,\dots,x_n)$ are both affine with positive slope $Q(n)$, and in particular are diffeomorphisms of $\R$. The Euler-Lagrange equation can be rewritten in the form \eqref{e:rec}, where the function $F$ is given by
\[
F(w_1,\dots,w_{2n} ) = \frac{1}{Q(n)} \Bigl( |w_{n+1}|^{\frac{1}{d}-1} w_{n+1} - \sum_{j=1}^{2n} Q(n+1-j) w_j \Bigr).
\]
It is continuous on $\R^{2n}$ and smooth on the complement of the hyperplane $\{w_{n+1}=0\}$, on which it is not differentiable. The map $\varphi: \R^{2n} \rightarrow \R^{2n}$ given by \eqref{e:varphi} is now a homeomorphism of $\R^{2n}$. It restricts to a smooth diffeomorphism from $\{w_{n+1}\neq 0\}$ to $\{w_n\neq 0\}$, but is not differentiable at the points of the hyperplane $\{w_{n+1}=0\}$. 

Note that $\varphi$ fixes the origin and is reversible with respect to the involution
\[
S:\R^{2n} \rightarrow \R^{2n}, \qquad S(w_1,\dots,w_{2n}) := (w_{2n},\dots,w_1),
\]
meaning that $S\circ \varphi = \varphi^{-1}\circ S$.

The functions $p_i$ from \eqref{e:PC} are linear and given by
\[
p_i(w) = \sum_{j=i}^{n} Q(n+i-j)\,w_j, \qquad 1\leq i \leq n,
\]
so the differential of the Poincar\'e-Cartan one-form $\lambda = \sum_{i=1}^n p_i\, \mathrm{d} w_{n+i}$ is the constant symplectic two-form
\[
\omega = \sum_{i=1}^n\sum_{j=n+1}^{n+i}Q(j-i)\,\mathrm{d} w_i\wedge \mathrm{d}w_j.
\]
Note that $S^* \omega = - \omega$. The map $\varphi$ restricts to a symplectomorphism from $\{w_{n+1}\neq 0\}$ to $\{w_n\neq 0\}$ with respect to $\omega$. We conclude that there is a bijective correspondence between the solutions of \eqref{e:consign} up to translations and the orbits of the non-smooth $S$-reversible symplectomorphism $\varphi$. Together with Proposition \ref{p:zachary}, we then obtain the following bijective correspondence between Gibbs measures and homoclinic orbits, generalizing Proposition \ref{p:gibbs=homo}.

\begin{prop}
Assume that the irreducible interaction function $Q$ is supported in $\{-n,\dots,n\}$ and satisfies $Q(n)>0$. Then there is a one-to-one correspondence between $\mathscr{G}(d,Q)/\Z$ and the set of orbits of the non-smooth $S$-reversible symplectomorphism $\varphi: \R^{2n} \rightarrow \R^{2n}$ which are homoclinic to the fixed point at the origin and contained in $(0,+\infty)^{2n}$.
\end{prop} 

\begin{proof}
By Proposition \ref{p:zachary}, every Gibbs measure $\mu\in \mathscr{G}(d,Q)$ corresponds to a positive solution $u$ of \eqref{e:nosign} which is in $\ell^p(\Z)$, where $p=\frac{d+1}{d}$. Since $u(k)\rightarrow 0$ for $k\rightarrow \pm\infty$, the map $w: \Z \rightarrow \R^{2n}$ from \eqref{e:theorbit} defines an orbit of $\varphi$ in $(0,+\infty)^{2n}$ which is homoclinic to $0$.

In order to prove the converse, we must show that all positive solutions of \eqref{e:nosign} which are infinitesimal for $k\rightarrow \pm \infty$ are $p$-summable. The sequence $m(k) := \sup_{j\geq k} u(j)$ is non-increasing and infinitesimal for $k\rightarrow +\infty$. Setting $c:= \sum_{i\in \Z} Q(i)\in (0,+\infty)$, from \eqref{e:nosign} we find
\[
u(j) = \Bigl( \sum_{|i|\leq n} Q(i) u(j-i) \Bigr)^d \leq \bigl(c \, m(j-n)\bigr)^d \qquad \forall j\in \Z,
\]
and taking the supremum over $j\geq k$ we obtain
\[
m(k) \leq \sup_{j\geq k} \bigl(c \, m(j-n)\bigr)^d = \bigl(c \, m(k-n)\bigr)^d \qquad \forall k\in \Z.
\]
The above bound implies that the infinitesimal sequence $m'(k):= m(nk)$ satisfies
\[
m'(k) \leq  \bigl(c \, m'(k-1)\bigr)^d \qquad \forall k\in \Z.
\]
Therefore, $m'$ is infinitesimal for $k\rightarrow +\infty$ of order at least $d\geq 2$, which implies that there exist $C>0$ and $\theta\in (0,1)$ such that
\[
m'(k) \leq C \theta^{d^k} \qquad \forall k\geq 0.
\]
Then
\[
u(j) \leq m(j) \leq m'\bigl(\lfloor {\textstyle \frac{j}{n}} \rfloor \bigr) \leq C \theta^{d^{\lfloor \frac{j}{n} \rfloor}} \leq C \theta^{d^{\frac{j}{n}-1}} \qquad \forall j\geq n,
\]
and in particular $u$ is $p$-summable on $\N$. The argument for the summability on $-\N$ is analogous.
\end{proof}

\begin{rem}
\label{r:cano}
{\rm It is of course possible to identify $(\R^{2n},\omega)$ with $(T^* \R^n,\omega_{\mathrm{can}})$, where $\omega_{\mathrm{can}} = \sum_{i=1}^n \mathrm{d} y_i \wedge \mathrm{d} x_i$ is the canonical symplectic form. Indeed, the linear isomorphism 
\[
\Lambda : \R^{2n} \rightarrow T^* \R^n, \qquad \Lambda(w)=(x,y), \qquad x_i = w_{n+i}, \; y_i := p_i(w),
\]
satisfies $\Lambda^* \omega_{\mathrm{can}} = \omega$. In canonical variables, the symplectomorphism $\varphi$ and the anti-symplectic involution $S$ take the form
\[
\begin{split}
\Lambda \circ \varphi \circ \Lambda^{-1} (x,y) &= \Bigl( x_2,\dots,x_n, x_{n+1} , y_2 + Q(1) x_1, \dots, y_n + Q(n-1) x_1, Q(n) x_1 \Bigr), \\
& \mbox{where }x_{n+1}:= \frac{1}{Q(n)} \bigl( |x_1|^{\frac{1}{d}-1} x_1 - y_1 - \sum_{i=1}^n Q(i-1) x_i \bigr), \\
\Lambda \circ S \circ \Lambda^{-1} (x,y) &= (H^{-1} y, Hx), \qquad \mbox{where } H:= \bigl( Q(i+j-1) \bigr)_{1\leq i,j \leq n}.
\end{split}
\]
Working in canonical coordinates does not seem to bring particular advantages here, but it will be important in the next section.}
\end{rem}

\section{Homoclinic orbits of Lagrangian correspondences}
\label{s:holc}

In this section, we show that in the case of an irreducible interaction function $Q\in \ell^{\frac{d+1}{2}}(\Z)$ with arbitrary support, the elements of $\mathscr{G}(d,Q)$ are in bijection with the set of orbits of a Lagrangian correspondence on the cotangent bundle of a Banach space which are homoclinic to the origin in a suitable sense.

We first need to recall some terminology about the canonical symplectic structure of the cotangent bundle of a Banach space and its Lagrangian correspondences. Let $E$ be a reflexive real Banach space. Its \textit{cotangent bundle} $T^* E := E \times E^*$ carries the \textit{canonical symplectic form}
\[
\omega_{\mathrm{can}} \bigl( (x,y), (x',y') \bigr) := \langle y,x' \rangle - \langle y',x \rangle,
\]
where $\langle \cdot,\cdot \rangle$ denotes the duality pairing between covectors in $E^*$ and vectors in $E$. It is a \textit{strong symplectic form} on the Banach space $T^* E$, meaning that it is a continuous skew-symmetric bilinear form on $T^* E$ whose induced map $T^*E \rightarrow (T^* E)^*= E^* \times E^{**} \cong E^* \times E$, $z\mapsto \omega_{\mathrm{can}}(z,\cdot)$, is an isomorphism. This uses the assumption that $E$ is reflexive.

A \textit{Lagrangian correspondence} on $(T^* E,\omega_{\mathrm{can}})$ is a smooth submanifold $\mathcal{R}\subset T^* E \times T^* E$ which is Lagrangian with respect to the strong symplectic form $\Omega := (-\omega_{\mathrm{can}}) \oplus \omega_{\mathrm{can}}$ on $T^* E \times T^*E$. The latter fact means that for every $\zeta\in \mathcal{R}$ we have
\[
\bigl(T_{\zeta} \mathcal{R} \bigr)^{\perp_{\Omega}} := \{ \xi \in T^* E \times T^* E \mid \Omega (\xi,\eta)=0 \; \forall \eta\in T_{\zeta} \mathcal{R} \} = T_{\zeta} \mathcal{R}.
\]
Here, we are using the standard definition of submanifold $\mathcal{R}$ of a Banach space, ensuring in particular that tangent spaces are complemented. Symplectomorphisms of $(T^* E \times T^*E,\Omega)$ map Lagrangian correspondences to Lagrangian correspondences.

Lagrangian correspondences generalize \textit{symplectomorphisms}.  Indeed, a smooth diffeomorphism $\varphi: T^* E \rightarrow T^* E$  satisfies $\varphi^* \omega_{\mathrm{can}} = \omega_{\mathrm{can}}$ if, and only if, its graph is a Lagrangian submanifold of $(T^* E\times T^* E,\Omega)$, so a symplectomorphism can be identified with the Lagrangian correspondence given by its graph. 
An interesting class of linear Lagrangian correspondences is given by \textit{conormal bundles} of graphs of bounded linear operators: if $W: E \rightarrow E$ is a bounded linear operator, then the closed and complemented linear subspace
\[
\mathcal{R}_W := \Bigl\{ \bigl( (x,y), (x',y') \bigr) \in T^*E \times T^* E \mid x'=Wx, \; y = W^* y' \Bigr\}
\]
is readily seen to be a Lagrangian correspondence on $T^* E$. If $f: E \rightarrow \R$ is a smooth function, then the shear $\Phi(x,y) := (x,y-\di f(x))$ is a symplectomorphism of $(T^* E, \omega_{\mathrm{can}})$ and hence the set
\begin{equation}
\label{e:con+shear}
\scalebox{.93}{$ \mathcal{R}_{W,f} := \bigl\{ \bigl( (x,y), (x',y') \bigr) \in T^*E \times T^* E \mid x'=Wx, \; y = W^* y' - \di f(x) \bigr\} = (\Phi\times \mathrm{id}) (\mathcal{R}_W)$}
\end{equation}
is also a Lagrangian correspondence on $T^* E$.

A bi-infinite sequence $z: \Z\rightarrow T^* E$ is said to be an \textit{orbit} of the Lagrangian correspondence $\mathcal{R}$ if
\[
\bigl( z(k), z(k+1) \bigr)\in \mathcal{R} \qquad \forall k\in \Z.
\]
In the case $\mathcal{R}=\mathrm{graph}\, \varphi$, an orbit of $\mathcal{R}$ is the same thing as an orbit of $\varphi$.
 
Now let $Q\in \ell^{\frac{d+1}{2}}(\Z)$ be an interaction function. Setting as usual $p:= \frac{d+1}{d}$, we wish to identify the solutions $u\in \ell^p(\Z)$ of the equation
\[
Q* u = |u|^{p-2} u,
\]
with suitable orbits of a Lagrangian correspondence on the cotangent bundle of the Banach space
\[
E:= \ell^p( \Z_{\geq 0} ) ,
\]
which is
\[
T^* E = E \times E^* = \ell^p( \Z_{\geq 0} ) \times \ell^{p'}( \Z_{\geq 0} ),
\]
where $p':= d+1$ is the conjugate exponent of $p$. In this infinite dimensional setting, the Lagrangian from \eqref{e:lagrangian} becomes the following function on $E$
\[
L(x):= \frac{1}{2} Q(0) x(0)^2 + \sum_{i\geq 1} Q(i) x(0) x(i) - \frac{1}{p} |x(0)|^p,
\]
which is of class $C^1$. Denote by $\sigma$ the left shift $(\sigma x)(k) = x(k+1)$ on $E$ and by $\sigma^*$ its adjoint, which is the right shift on $E^*$.
As seen in \eqref{e:con+shear}, the shift operator $\sigma$ and the Lagrangian $L$ generate the following Lagrangian correspondence on $T^* E$:
\[
\mathcal{R} = \mathcal{R}_{\sigma,L} := \bigl\{ \bigl( (x,y), (x',y') \bigr) \in T^* E \times T^* E \mid x' = \sigma x, \; y = \sigma^* y' - \di L(x) \bigr\}.
\]
Here, we are using the term Lagrangian correspondence in an extended way, because the function $L$ is not smooth. In particular, $\mathcal{R}$ is only a topological submanifold of $T^*E \times T^* E$. In coordinates, $\mathcal{R}$ 
is the set of tuples $((x,y),(x',y'))$ in $\ell^p(\Z_{\geq 0}) \times \ell^{p'}(\Z_{\geq 0}) \times \ell^p(\Z_{\geq 0}) \times \ell^{p'}(\Z_{\geq 0})$ such that
\[
\begin{split}
x'(k) &= x(k+1) \qquad \forall k\geq 0, \\
y(0) &= - Q(0) x(0) - \sum_{i\geq 1} Q(i) x(i) + |x(0)|^{p-2} x(0), \\
y(k) &= y'(k-1) - Q(k) x(0) \qquad \forall k\geq 1.
\end{split}
\]
Note that the origin $0$ of $T^* E$ is a fixed point of $\mathcal{R}$, meaning that $(0,0)\in \mathcal{R}$.
In order to study the orbits of $\mathcal{R}$, it is useful to introduce the bounded linear operator
\[
\mathcal{Q} : E \rightarrow E^*, \qquad ( \mathcal{Q} x)(i) := \sum_{j\geq 0} Q(i-j) x(j) \quad \forall i\geq 0,
\]
and the $C^1$ action functional
\[
A : E \rightarrow \R, \qquad A(x) := \frac{1}{2} \langle \mathcal{Q} x, x \rangle - \frac{1}{p} \|x\|_p^p = \sum_{i\geq 0} L (\sigma^i x),
\]
which satisfies
\begin{equation}
\label{e:releAL}
A - A\circ \sigma = L.
\end{equation}

Denote by $\mathcal{U}$ the set of bi-infinite sequences $u: \Z \rightarrow \R$ such that $u|_{\N}$ is $p$-summable. Given $u\in \mathcal{U}$, we define a $T^* E$-valued bi-infinite sequence $(x_k,y_k)_{k\in \Z}$ by
\begin{equation}
\label{e:bije}
\begin{split} 
x_k(i) &:= u(k+i) \qquad \forall i \geq 0, \\
y_k &:= - \mathrm{d}A(x_k) = |x_k|^{p-2} x_k - \mathcal{Q} x_k.
\end{split}
\end{equation}
Note that the summability assumption on $u$ implies that
$x_k\rightarrow 0$ in $E$ for $k\rightarrow +\infty$ and, since $\mathrm{d}A$ is continuous with $\mathrm{d}A(0)=0$, $y_k \rightarrow 0$ in $E^*$ for $k\rightarrow +\infty$.

\begin{prop}
\begin{enumerate}[(i)]
\item The map $u\mapsto (x_k,y_k)_{k\in \Z}$ defined in \eqref{e:bije} is a bijection from $\mathcal{U}$ onto the set of orbits of $\mathcal{R}$.
\item This map restricts to a bijection from the set of solutions $u\in \ell^p(\Z)$ of the equation
\begin{equation}
\label{e:theeq}
Q*u = |u|^{p-2} u
\end{equation}
onto the set of orbits $(x_k,y_k)_{k\in \Z}$ of $\mathcal{R}$ such that
\begin{equation}
\label{e:asymptotic}
x_k \rightharpoonup 0 \mbox{ in } E \mbox{ and } y_k \rightarrow 0 \mbox{ in } E^* \mbox{ for } k\rightarrow -\infty,
\end{equation}
where $\rightharpoonup$ denotes weak convergence.
\end{enumerate}
\end{prop}

\begin{proof}
We first check that the bi-infinite sequence $(x_k,y_k)_{k\in \Z}$ defined in \eqref{e:bije} is an orbit of $\mathcal{R}$, that is,
\begin{equation}
\label{e:DD}
x_{k+1} = \sigma x_k, \qquad y_k = \sigma^* y_{k+1} - \di L(x_k) \qquad \forall k\in \Z.
\end{equation}
The first identity in \eqref{e:DD} is clear. Differentiating \eqref{e:releAL} at $x_k$ we find
\begin{equation}
\label{e:diff}
\mathrm{d} A(x_k) - \sigma^* \mathrm{d} A(\sigma x_k) = \mathrm{d} L(x_k),
\end{equation}
and hence
\[
\sigma^* y_{k+1} = - \sigma^* \mathrm{d} A(x_{k+1}) = - \mathrm{d} A(x_k) + \mathrm{d} L(x_k) = y_k +  \mathrm{d} L(x_k),
\]
proving the second identity in \eqref{e:DD}. Now assume that $(x_k,y_k)_{k\in \Z}$ is a $T^* E$-valued bi-infinite sequence satisfying \eqref{e:DD}.
By the first identity in \eqref{e:DD}, there exists a unique $u\in \mathcal{U}$ such that
\[
u(k+i) = x_k(i) \qquad \forall k\in \Z, \; \forall i\geq 0.
\]
Together with \eqref{e:diff}, this implies that the bi-infinite sequence $(\tilde{x}_k,\tilde{y}_k)_{k\in \Z}$ defined by $u$ via \eqref{e:bije} satisfies $\tilde{x}_k=x_k$ and
\[
\tilde{y}_k = - \mathrm{d} A(x_k) = - \mathrm{d} L(x_k) - \sigma^* \mathrm{d} A(x_{k+1}) = - \mathrm{d} L(x_k) + \sigma^* \tilde{y}_{k+1}.
\]
By the second identity in \eqref{e:DD}, we obtain that the bi-infinite sequence $v_k := y_k - \tilde{y}_k$ satisfies $v_k = \sigma^* v_{k+1}$ for every $k\in \Z$. This implies $v_k = (\sigma^*)^j v_{k+j}$ for every $j\geq 1$ and since the image of $(\sigma^*)^j$ consists of the elements in $\ell^{p'}(\Z_{\geq 0})$ whose first $j$ entries vanish, we deduce that $v_k=0$ for every $k\in \Z$. This proves that $\tilde{y}_k = y_k$ and concludes the proof of (i).   

The proof of (ii) is based on the following two claims about the bi-infinite sequence $(x_k,y_k)_{k\in \Z}$ associated to $u\in \mathcal{U}$ via \eqref{e:bije}.

\medskip

\noindent \textit{Claim 1: $u$ belongs to $\ell^p(\Z)$ if, and only if, $x_k \rightharpoonup 0$ in $E$ for $k\rightarrow -\infty$.} Since the quantity
\[
\|x_k\|_p^p = \sum_{j\geq k} |u(j)|^p
\]
is non-increasing in $k$ and converges to $\|u\|_p^p$ for $k\rightarrow -\infty$, $u$ belongs to $\ell^p(\Z)$ if, and only if, $(x_k)$ is bounded. Weakly convergent sequences are bounded. Conversely, assume that $u\in \ell^p(\Z)$ and let $y\in \ell^{p'}(\Z_{\geq 0})$ and $n\geq 0$. Then
\[
|\langle y,x_k \rangle| \leq \left| \sum_{i=0}^n y(i) u(k+i) \right| + \left| \sum_{i>n} y(i) u(k+i) \right| \leq \|y\|_{p'} \| u \mathbbm{1}_{(-\infty,k+n]} \|_p + \|y \mathbbm{1}_{[n+1,+\infty)} \|_{p'} \|u\|_p.
\]
The first term tends to $0$ as $k\rightarrow -\infty$ and the second one is arbitrarily small for $n$ large, so $x_k  \rightharpoonup 0$ in $E$ for $k\rightarrow -\infty$.

\medskip

\noindent \textit{Claim 2: Assume that $u$ is in $\ell^p(\Z)$. Then $u$ solves \eqref{e:theeq} if, and only if, $y_k \rightarrow 0$ in $E^*$ for $k\rightarrow -\infty$.} By \eqref{e:bije}. for every  $i\geq 0$ we have
\[
y_k(i) = |u(k+i)|^{p-2} u(k+i) - \sum_{j\geq k} Q(k+i-j) u(j) = y_k'(i) + v(k+i),
\]
where
\[
y_k'(i) := \sum_{j\geq 1} Q(i+j) u(k-j), \quad \forall i\geq 0, \qquad
v := |u|^{p-2} u - Q* u \in \ell^{p'} (\Z).
\]
The term $y_k'$ is the restriction to $\Z_{\geq k}$, shifted by $k$, of the convolution of $Q$ with $u \mathbbm{1}_{(-\infty,k-1]}$. By Young's inequality
\[
\|y_k'\|_{p'} \leq \|Q\|_{\frac{d+1}{2}} \|  u \mathbbm{1}_{(-\infty,k-1]}\|_p \rightarrow 0
\]
for $k\rightarrow -\infty$. Therefore, $y_k\rightarrow 0$ in $E^*$ for $k\rightarrow -\infty$ if, and only if, $\|v \mathbbm{1}_{[k,+\infty)}\|_{p'}\rightarrow 0$ for $k\rightarrow -\infty$. Since this quantity is non-increasing, this is equivalent to $v=0$, and hence to \eqref{e:theeq}.

\medskip

We can now prove (ii). If $u\in \ell^p(\Z)$ solves \eqref{e:theeq}, then the corresponding orbit $(x_k,y_k)_{k\in \Z}$ satisfies \eqref{e:asymptotic} by Claims 1 and 2. Conversely, let $(x_k,y_k)_{k\in \Z}$ be an orbit of $\mathcal{R}$ satisfying \eqref{e:asymptotic}, and let $u\in \mathcal{U}$ be the corresponding real bi-infinite sequence. Then $u$ belongs to $\ell^p(\Z)$ by Claim 1 and solves \eqref{e:theeq} by Claim 2.
\end{proof}

Since the condition $(x_k,y_k) \rightarrow (0,0)$ in $T^* E$ for $k\rightarrow +\infty$ is automatic for the bi-infinite sequences defined by \eqref{e:bije}, the orbits $(x_k,y_k)_{k\in \Z}$ appearing in (ii) can be considered as homoclinic to the fixed point at the origin in $T^* E$, with the caveat that the convergence of $(x_k)$ to $0$ for $k\rightarrow -\infty$ is only weak (since $\|x_k\|_p \rightarrow \|u\|_p$ for $k\rightarrow -\infty$, convergence in norm holds only in the case of the trivial orbit $(x_k,y_k)=(0,0)$). Thanks to Proposition \ref{p:zachary}, we obtain the following characterization of Gibbs measures as homoclinic orbits of $\mathcal{R}$.

\begin{cor}
Let $Q$ be an irreducible and $\frac{d+1}{2}$-summable interaction function. Then there is a one-to-one correspondence between the elements of $\mathscr{G}(d,Q)$ and the set of orbits $(x_k,y_k)_{k\in \Z}$ of the Lagrangian correspondence $\mathcal{R}$ such that $x_k(i) >0$ for every $k\in \Z$ and $i\geq 0$, $x_k \rightharpoonup 0$ in $E$ and $y_k \rightarrow 0$ in $E^*$ for $k\rightarrow -\infty$.
\end{cor} 

\begin{rem}
{\rm In the case of interaction functions with bounded support, the symplectomorphism $\varphi$ is reversible with respect to an anti-symplectic involution $S$ (see Section \ref{s:ifbs}). In the infinite dimensional setting considered here, reversibility takes the following form. Consider the bounded operator
\[
\mathcal{H}: E \rightarrow E^*, \qquad (\mathcal{H}x)(i) := \sum_{j\geq 0} Q(i+j+1) x(j),
\]
which is the infinite-dimensional version of the matrix $H$ from Remark \ref{r:cano}, and consider the closed and complemented linear subspace
\[
\mathcal{S}:= \bigl\{ \bigl( (x,y), (x',y') \bigr) \in T^* E \times T^* E \mid  y=\mathcal{H}x', \; y' = \mathcal{H} x \bigr\}.
\]
Seen as a submanifold of $T^* E \times T^* E$, $\mathcal{S}$ is Lagrangian with respect to the symplectic form $\omega_{\mathrm{can}} \oplus \omega_{\mathrm{can}}$. This correspondence is involutive, meaning that $\mathcal{S}^{-1} = \mathcal{S}$, where the inverse of a correspondence on $T^* E$ is defined as its image via the involution $(z,z') \mapsto (z',z)$ on $T^* E \times T^* E$. Then $\mathcal{R}$ is $\mathcal{S}$-reversible, meaning that
\[
\mathcal{R} \circ \mathcal{S} = \mathcal{S} \circ \mathcal{R}^{-1},
\]
where the composition of correspondences is defined as usual by
\[
\mathcal{R} \circ \mathcal{S} := \bigl\{ (z,z'') \in T^* E \times T^*E \mid \exists z'\in T^* E \mbox{ such that } (z,z') \in \mathcal{S} \mbox{ and } (z',z'')\in \mathcal{R} \bigr\}.
\]
The verification of these facts is left to the reader.}
\end{rem}

\appendix

\section{Appendix: Proof of Proposition \ref{p:finite-partition-function}}
\label{a:gibbs-measures}

In this appendix, we prove that $d$-summability of the interaction function $Q$ guarantees that the partition function is finite, and that the exponent $d$ is optimal for this to hold. We first need a lemma, in which we use the convention that a product over the empty set has value 1.

\begin{lem}
\label{l:tree-sum}
Let $(\Lambda,E_{\Lambda})$ be a finite tree in which every vertex $x\in \Lambda$ has degree $k_x\leq d+1$, and set $m_x:= d+1-k_x$. For every $x\in \Lambda$ let $f_{x,1},\dots,f_{x,m_x}$ be nonnegative functions in $\ell^d(\Z)$ and set $a_x:= \prod_{h=1}^{m_x} f_{x,h}$. If $Q$ is in $\ell^d(\Z)$, then
\begin{equation}
\label{e:tree-sum}
\sum_{\xi\in \Z^{\Lambda}} \, \prod_{x\in \Lambda} a_x(\xi_x) \prod_{\{x,y\}\in E_{\Lambda}} Q(\xi_x-\xi_y) \leq \|Q\|_d^{|\Lambda|-1} \prod_{x\in \Lambda} \prod_{h=1}^{m_x} \|f_{x,h}\|_d.
\end{equation}
\end{lem}

\begin{proof}
We argue by induction on $|\Lambda|$. All the terms in the sums below are nonnegative, so these sums can be rearranged freely. If $\Lambda=\{x\}$, then $k_x=0$, $m_x=d+1$, and the H\"older inequality with $d+1$ factors, together with the inequality $\|\cdot\|_{d+1}\leq \|\cdot\|_d$, gives us
\[
\sum_{i\in \Z} \prod_{h=1}^{d+1} f_{x,h}(i) \leq \prod_{h=1}^{d+1} \|f_{x,h}\|_{d+1} \leq \prod_{h=1}^{d+1} \|f_{x,h}\|_{d},
\]
which is \eqref{e:tree-sum} in this case.

Assume now that $|\Lambda|\geq 2$ and that the lemma holds for all trees with $|\Lambda|-1$ vertices. Let $x\in \Lambda$ be a leaf of $(\Lambda,E_{\Lambda})$, let $y$ be its unique neighbour, so that $k_x=1$ and $m_x=d$, and let $(\Lambda',E_{\Lambda'})$ be the tree which is obtained by removing the vertex $x$ and the edge $\{x,y\}$. Writing $\xi\in \Z^{\Lambda}$ as $\xi=(\xi',\xi_x)$ with $\xi'\in \Z^{\Lambda'}$, the summands in the left-hand side of \eqref{e:tree-sum} factor as
\[
\Bigl( \prod_{z\in \Lambda'} a_z(\xi_z) \prod_{\{z,z'\}\in E_{\Lambda'}} Q(\xi_z-\xi_{z'}) \Bigr) \cdot a_x(\xi_x)\, Q(\xi_x-\xi_y),
\]
where only the last two factors depend on $\xi_x$. Since $Q$ is even, summing first over $\xi_x\in \Z$ gives
\[
\sum_{\xi_x\in \Z} a_x(\xi_x) Q(\xi_x-\xi_y) = (Q*a_x)(\xi_y) =: g(\xi_y),
\]
and hence the left-hand side of \eqref{e:tree-sum} equals
\begin{equation}
\label{e:reduced-sum}
\sum_{\xi'\in\Z^{\Lambda'}} \, \prod_{z\in \Lambda'} a_z'(\xi_z) \prod_{\{z,z'\}\in E_{\Lambda'}} Q(\xi_z-\xi_{z'}),
\end{equation}
where $a_y':= a_y\, g$ and $a_z':=a_z$ for every $z\in \Lambda'\setminus\{y\}$. This is a sum of the same form over the smaller tree $(\Lambda',E_{\Lambda'})$: there, the vertex $y$ has degree $k_y-1$, and $a_y'$ is the product of the $m_y+1=d+1-(k_y-1)$ functions $f_{y,1},\dots,f_{y,m_y},g$, whereas degrees and functions at the other vertices are unchanged.

In order to apply the inductive hypothesis, we must check that $g$ belongs to $\ell^d(\Z)$. By the H\"older inequality with $d$ factors, $a_x$ is in $\ell^1(\Z)$ with $\|a_x\|_1\leq \prod_{h=1}^d \|f_{x,h}\|_d$, and the Young inequality implies
\begin{equation}
\label{e:leaf}
\|g\|_d \leq \|Q\|_d \|a_x\|_1 \leq \|Q\|_d \prod_{h=1}^d \|f_{x,h}\|_d.
\end{equation}
By the inductive hypothesis, \eqref{e:reduced-sum} is bounded from above by
\[
\|Q\|_d^{|\Lambda|-2} \, \|g\|_d \prod_{z\in \Lambda'} \prod_{h=1}^{m_z} \|f_{z,h}\|_d,
\]
and \eqref{e:tree-sum} follows from \eqref{e:leaf}.
\end{proof}

\begin{proof}[Proof of Proposition \ref{p:finite-partition-function}]
Let $\Lambda_1,\dots,\Lambda_r$ be the sets of vertices of the connected components of the subgraph of $\Gamma_d$ which is induced by $\Lambda$. Each of these components is a finite tree $(\Lambda_s,E_{\Lambda_s})$, and every edge of $\Gamma_d$ with both end-points in $\Lambda$ belongs to exactly one of the sets $E_{\Lambda_s}$. For $x\in \Lambda$ we set
\[
a_x(i) := \prod_{\substack{y\in \partial \Lambda \\ \{x,y\}\in E}} Q(i-\eta_y), \qquad \forall i\in \Z.
\]
Then for every $\xi\in \Z^{\Lambda}$ we have
\[
w_{\Lambda}(\xi,\eta) = \prod_{s=1}^r w_s(\xi|_{\Lambda_s}), \qquad \mbox{where} \quad w_s(\zeta) := \prod_{x\in \Lambda_s} a_x(\zeta_x) \prod_{\{x,y\}\in E_{\Lambda_s}} Q(\zeta_x-\zeta_y), \quad \forall \zeta\in \Z^{\Lambda_s},
\]
and hence
\[
Z_{\Lambda}(\eta) = \prod_{s=1}^r \, \sum_{\zeta\in \Z^{\Lambda_s}} w_s(\zeta).
\]
Let $x\in \Lambda_s$ and denote by $k_x\leq d+1$ its degree in the tree $(\Lambda_s,E_{\Lambda_s})$. Since all the neighbours of $x$ in $\Gamma_d$ which belong to $\Lambda$ are in $\Lambda_s$, the remaining $m_x=d+1-k_x$ neighbours of $x$ belong to $\partial \Lambda$, and hence $a_x$ is the product of $m_x$ translates of $Q$, each of which has $\ell^d$-norm $\|Q\|_d$. By Lemma \ref{l:tree-sum}, we obtain
\[
\sum_{\zeta\in \Z^{\Lambda_s}} w_s(\zeta) \leq \|Q\|_d^{|\Lambda_s|-1 + \sum_{x\in \Lambda_s} m_x}.
\]
Here, $|\Lambda_s|-1$ is the number of edges in $E_{\Lambda_s}$ and $\sum_{x\in \Lambda_s} m_x$ is the number of edges of $\Gamma_d$ joining a vertex in $\Lambda_s$ to a vertex in $\partial \Lambda$. Since every edge of $\Gamma_d$ with at least one end-point in $\Lambda$ is counted exactly once when $s$ ranges from $1$ to $r$, the sum of these exponents over $s$ is $e(\Lambda)$, and the conclusion follows.
\end{proof}

\begin{rem}
{\rm The exponent $d$ in Proposition \ref{p:finite-partition-function} is optimal: if $q>d$ and $\frac{1}{q}<\alpha<\frac{1}{d}$, then the function $Q(i):=(1+|i|)^{-\alpha}$ belongs to $\ell^q(\Z)$, but $Z_{\Lambda}(\eta)=+\infty$ for every $\eta \in \Z^{\partial \Lambda}$ when $\Lambda$ is a connected set whose cardinality $n$ satisfies $\alpha(dn+1)<n$. Indeed, such a set satisfies $e(\Lambda)=dn+1$, so $w_{\Lambda}(\xi,\eta)$ is the product of $dn+1$ factors, each of which is at least $(1+2R)^{-\alpha}$ if $R\geq \max_{y\in \partial \Lambda} |\eta_y|$ and $|\xi_x|\leq R$ for every $x\in \Lambda$. Therefore,
\[
Z_{\Lambda}(\eta) \geq (2R+1)^{n} (1+2R)^{-\alpha (dn+1)} = (2R+1)^{n-\alpha(dn+1)},
\]
which diverges for $R\rightarrow +\infty$.}
\end{rem}

\section{Appendix: Topological horseshoes}
\label{a:th}

In this appendix, we discuss a general result about topological horseshoes for homeomorphisms of surfaces having a fixed point of saddle type whose stable and unstable sets have a topological crossing.

Let $M$ be a surface, $\psi: M \rightarrow M$ a homeomorphism, and $p_0$ a fixed point of $\psi$ of \textit{saddle type}, meaning that there exists a compact neighborhood $K\subset M$ of $p_0$ and a homeomorphism $h: [-1,1]\times [-1,1] \rightarrow K$ mapping $(0,0)$ to $p_0$ and such that
\begin{equation}
\label{e:norm-form}
\psi \circ h(x,y) = h \Bigl( \frac{x}{2}, 2y \Bigr), \qquad \forall (x,y) \in [-1,1] \times \Bigl[ - \frac{1}{2}, \frac{1}{2} \Bigr].
\end{equation}
In the above linear normal form, we chose $\frac{1}{2}$ and $2$ as contraction and expansion factors, but we could have chosen any factors $\lambda,\mu$ with $0<\lambda<1<\mu$, because the linear map $(x,y) \mapsto (\lambda x,\mu y)$ is topologically conjugate to $(x,y) \mapsto (\frac{x}{2},2 y)$. By \eqref{e:norm-form}, the stable and unstable sets of $p_0$ are countable increasing unions of compact arcs:
\[
\begin{split}
W^s(p_0)&:= \{p\in M \mid \psi^n(p) \rightarrow p_0 \mbox{ for } n\rightarrow +\infty\} = \bigcup_{n\geq 0} \psi^{-n} (h([-1,1]\times \{0\})), \\
W^u(p_0)&:= \{p\in M \mid \psi^n(p) \rightarrow p_0 \mbox{ for } n\rightarrow -\infty\} = \bigcup_{n\geq 0} \psi^{n} (h(\{0\} \times [-1,1])).
\end{split}
\]

Having fixed the conjugacy $h$ onto $K$ satisfying \eqref{e:norm-form}, we denote by $K_+\subset K$ the positive sector
\[
K_+ := h\bigl([0,1]\times [0,1] \bigr),
\]
and by
\[
W^s_+(p_0):= \bigcup_{n\geq 0} \psi^{-n} (h([0,1]\times \{0\})), \qquad
W^u_+(p_0):= \bigcup_{n\geq 0} \psi^{n} (h(\{0\} \times [0,1])),
\]
the positive branches of $W^s(p_0)$ and $W^u(p_0)$. 

Let $\alpha$ and $\beta$ be arcs in $M$ (see the beginning of Section \ref{s:duesol} for the terminology used here). We say that $\alpha$ has a \textit{topological crossing} at $q\in \mathrm{int}(\alpha) \cap \mathrm{int}(\beta)$ with $\beta$ if there exists a homeomorphism $f$ from $(-1,1) \times (-1,1)$ to an open subset of $M$ such that $f(0,0)=q$, $f((-1,1) \times \{0\}) \subset \beta$, and every subarc $\alpha'\subset \alpha$ such that $q\in \mathrm{int}(\alpha')$ has a subarc $\alpha''$ with $q\in \mathrm{int}(\alpha'')$ having end-points $q_+$ and $q_-$ with $f^{-1}(q_+) \in (-1,1) \times (0,1)$ and $f^{-1}(q_-) \in (-1,1) \times (-1,0)$. Note that this condition is weaker than the notion of topological transversality obtained in Theorems \ref{t:omoclina1} and \ref{t:omoclina2}, and  does not force the crossing point $q$ to be isolated. 

We say that $W^u_+(p_0)$ has a topological crossing with $W^s_+(p_0)$ at some $p_1\neq p_0$ if 
the compact arc $\psi^{n} (h(\{0\} \times [0,1]))\subset W^u_+(p_0)$ has a topological crossing with the compact arc $\psi^{-n} (h([0,1]\times \{0\})) \subset W^s_+(p_0)$ at $p_1$, where $n\geq 0$ is so large that $p_1$ belongs to the interior of these two compact arcs. Note that in this case $W^u_+(p_0)$ has a topological crossing with $W^s_+(p_0)$ at $\psi^k(p_1)$ for every $k\in \Z$.

The main result of this appendix is the following refinement of \cite[Theorem 2.1]{bw95}.

\begin{thm}
\label{t:horseshoe}
Let $\psi: M \rightarrow M$ be a homeomorphism of a surface $M$ having a fixed point of saddle type at $p_0\in M$. Assume that $W^u_+(p_0)$ has a topological crossing with $W^s_+(p_0)$ at some $p_1\neq p_0$. Let $U_0\subset K_+$ be a relative neighborhood of $p_0$ in $K_+$, $U_1\subset M$ an open neighborhood of $p_1$, and $V\subset M$ an open neighborhood of $\{\psi^i(p_1) \mid i\in \Z \setminus \{0\}\}$. Then there are disjoint compact sets $Q_0\subset U_0$, $Q_1\subset U_1$, and a natural number $N$ such that:
\begin{enumerate}[(i)] 
\item The $\psi^N$-invariant compact set $\Lambda:= \bigcap_{j\in \Z} \psi^{Nj}(Q_0 \cup Q_1)$ is non-empty.
\item The map $\Lambda \rightarrow \{0,1\}^{\Z}$ mapping $p\in \Lambda$ to the bi-infinite sequence $\omega_p \in  \{0,1\}^{\Z}$ such that $\psi^{Nj}(p) \in Q_{\omega_p(j)}$ for every $j\in \Z$ is continuous and surjective.
\item For every $p\in \Lambda$, the set 
\[
\{\psi^i(p) \mid i \in \Z, \mbox{ and if } i=Nj \mbox{ with } j \in \Z \mbox{ then } \omega_p(j)=0 \} 
\]
is contained in $U_0 \cup V$.
\item A point $p\in \Lambda$ belongs to $W^s_+(p_0)\cap W^u_+(p_0)$ if, and only if,  $\omega_p(j)=0$ for all but a finite number of $j\in \Z$.
\end{enumerate}
\end{thm}

In (ii), $\{0,1\}$ is endowed with the discrete topology and $\{0,1\}^{\Z}$ with the corresponding product topology.
The semiconjugacy property (ii) tells us that the full shift on two elements is a \textit{horseshoe factor} of $\psi^N$ and implies that $\psi$ has positive topological entropy (at least $\frac{1}{N} \log 2$, see \cite[Lemma 1.3]{bw95}). Here, our focus is on the implications in terms of existence of homoclinic orbits, and the above formulation is tailored to the application to the map $\varphi$ in Section \ref{s:coding}, where we are interested in homoclinic orbits  in the positive quadrant $(0,+\infty)\times (0,+\infty)$.

The remaining part of this appendix is devoted to the proof of Theorem \ref{t:horseshoe}, building on ideas from \cite{bw95}. Up to the replacement of $K$ by $h([-\delta,\delta]^2)$ for some $\delta\in(0,1)$, $h$ by $(x,y)\mapsto h(\delta x,\delta y)$, and $U_0$ by a smaller set, we may assume that $U_0 = K_+$.
Set
\[
R:= h([0,1]\times [-1,1]) \subset K.
\]
Since $p_1\in W^s_+(p_0)\setminus \{p_0\}$, we can find $n_0\in \N$ such that $\psi^{n_0}(p_1) \in h((0,1)\times \{0\})\subset \mathrm{int}(R)$. The positive branch $W^u_+(p_0)$ has a topological crossing with $W^s_+(p_0)$ at $p_1$, and hence also at $\psi^{n_0}(p_1)$. Since  every sufficiently small compact arc in $W^s_+(p_0)$ containing  $\psi^{n_0}(p_1)$ is contained in $h((0,1)\times \{0\})$, we can find a compact arc $\alpha$ such that:
\begin{enumerate}[(a)]
\item $\alpha\subset W^u_+(p_0) \cap \mathrm{int}(R)$;
\item one end-point of $\alpha$ belongs to $h((0,1)\times (-1,0))$ and the other to $h((0,1) \times (\rho,1))$, for some $\rho\in (0,1)$.
\end{enumerate}
Since $\alpha$ is a compact arc in $W^u_+(p_0)$, there exists $n_1>n_0$ such that
\begin{enumerate}[(a)]
\setcounter{enumi}{2}
\item $\psi^{-n_1}(\alpha) = h(\{0\} \times [z_1,z_2])$ for some $0<z_1 < z_2 \leq 1$. 
\end{enumerate}
Since $U_1$ is a neighborhood of $p_1$ and $V$ is a neighborhood of $\{\psi^i(p_1) \mid i\in \Z \setminus \{0\}\} $, up to the replacement of $\alpha$ by a suitable subarc, and correspondingly to the choice of a smaller positive number $\rho$ in (b), we can also assume
\begin{enumerate}[(a)]
\setcounter{enumi}{3}
\item $\psi^{-n_0}(\alpha) \subset U_1$;
\item $\displaystyle{\bigcup_{i\in \{1,\dots,n_1-1\} \setminus \{n_0\}}} \psi^{-i} (\alpha)\subset V$.
\end{enumerate}

By (d), we can find a set $W\subset R$ which is a neighborhood of $\alpha$ such that 
\begin{equation}
\label{e:daric1}
\psi^{-n_0}(W)\subset U_1.
\end{equation}
We set
\[
S:= h \bigl( [0,1]\times [0,\rho] \bigr),
\]
where $\rho\in (0,1)$ is the number appearing in (b). By (b), the end-points of $\alpha$ belong to distinct components of $\mathrm{int}(R) \setminus S$.

\begin{lem}
\label{l:lemmulo}
There exists $N\geq  n_1$ such that:
\begin{enumerate}[(i)]
\item $\psi^{-N}(\alpha) = h(\{0\} \times [y_1,y_2])$ with $0 < y_1  < y_2 \leq \rho$. 
\item $\psi^N(h([0,1]\times [y_1,y_2])) \subset W$.
\item $\psi^N(h([0,1]\times \{y_1\}))$ and $\psi^N(h([0,1]\times \{y_2\}))$ are contained in distinct components of $\mathrm{int}(R) \setminus S$.
\item For every $i=0,\dots,N-1$, $i\neq N-n_0$, we have $\psi^i(h([0,1]\times [y_1,y_2])) \subset V \cup K_+$.
\end{enumerate}
\end{lem}

\begin{proof} We will show that (i)-(iv) hold for $N\geq n_1$ large enough.
Write $n=n_1+k$ with $k\geq 0$. Then \eqref{e:norm-form} and (c) imply
\[
\psi^{-N}(\alpha) = h(\{0\} \times [2^{-k} z_1,2^{-k} z_2]),
\]
and hence (i) holds with
\[
y_1 := 2^{-k} z_1, \qquad y_2 := 2^{-k} z_2,
\]
provided that $k$ -- or equivalently $N$ -- is large enough. 

Fix a metric on $M$ which is compatible with its topology. If $y\in [y_1,y_2]$ then the diameter of the set
\[
\psi^k(h([0,1]\times \{y\})) = h([0,2^{-k}] \times \{2^k y\} )
\]
does not exceed $\omega(2^{-k})$, where $\omega$ is a continuity modulus for $h$. Therefore, the diameter of the set
\[
\psi^n(h([0,1]\times \{y\})) = \psi^{n_1} ( h([0,2^{-k}] \times \{2^k y\} ))
\]
does not exceed $\omega'\circ \omega(2^{-k})$, where $\omega'$ is a continuity modulus for the restriction of $\psi^{n_1}$ to the compact set $K_+$. Using also the fact that the above set contains the point $\psi^N(h(0,y))$ which belongs to $\alpha\subset \mathrm{int}(W)$, we conclude that (ii) holds if $k$ -- equivalently $N$ -- is so large that  the $\omega'\circ \omega(2^{-k})$-neighborhood of $\alpha$ is contained in $W$.

Similarly, since $\psi^n\circ h$ maps $\{0\} \times \{y_1,y_2\}$ to the end-points of $\alpha$, (iii) holds if $k$ -- equivalently $N$ -- is so large that  the $\omega'\circ \omega(2^{-k})$-neighborhood of these two end-points is contained in the open neighborhood $\mathrm{int}(R) \setminus S$ of the end-points of $\alpha$. This proves (iii).

If $0\leq i \leq k$ then
\[
\psi^i(h([0,1]\times [y_1,y_2])) = h([0, 2^{-i} ]\times [2^i y_1,2^i y_2]) = h([0, 2^{-i} ]\times [2^{i-k} z_1,2^{i-k} z_2]) \subset K_+,
\]
and hence (iv) holds for these indices $i$. Now assume that $k + 1 \leq  i \leq N-1$, or equivalently $1\leq  i-k \leq n_1-1$, and that $i\neq N-n_0$. By (e), we can find $\delta>0$ such that the closed $\delta$-neighborhood of the compact set 
\begin{equation}
\label{e:compactset}
\displaystyle{\bigcup_{j\in \{1,\dots,n_1-1\} \setminus \{n_0\}}} \psi^{-j} (\alpha)
\end{equation}
is contained in $V$. For every $y\in [y_1,y_2]$ the diameter of the set
\begin{equation}
\label{e:the-set}
\psi^i(h([0,1]\times \{y\})) = \psi^{i-k}(h([0,2^{-k}]\times \{2^k y\}))
\end{equation}
does not exceed $\omega''\circ \omega(2^{-k})$, where $\omega''$ is a common continuity modulus for the restrictions to $K_+$ of the maps $\psi, \psi^2,\dots,\psi^{n_1-1}$. The set \eqref{e:the-set} contains the point $\psi^i(h(0,y))$, which is contained in $\psi^{i-N}(\alpha)$ and hence in
the set \eqref{e:compactset} because $1\leq N-i \leq n_1 - 1$ and $N-i\neq n_0$. By our choice of $\delta$, if $k$ -- or equivalently $N$ -- is so large that $\omega''\circ \omega(2^{-k})\leq \delta$ then the set \eqref{e:the-set} is contained in $V$, concluding the proof of (iv).
\end{proof}

Let $n$, $y_1$, $y_2$ be as in the above lemma and set:
\[
H_0 := h\bigl([0,1]\times [0,2^{-N}]\bigr), \qquad H_1 := h \bigl( [0,1] \times [y_1,y_2] \bigr).
\]
Since the sets
\[
\psi^N(h(\{0\}\times [0,2^{-N}])) = h (\{0\} \times [0,1])\subset \partial R \mbox{ and } \psi^N(h(\{0\}\times [y_1,y_2])) = \alpha \subset \mathrm{int}(R)
\]
are disjoint, so are the intervals $[0,2^{-N}]$ and $[y_1,y_2]$, and hence $H_0$ and $H_1$. 

We define the disjoint compact sets $Q_0$ and $Q_1$ as
\[
\begin{split}
Q_0&:= \psi^{N-n_0} (H_0) = h \bigl( [0,2^{n_0-N}] \times [0,2^{-n_0}] \bigr) 
\subset K_+ = U_0, \\ Q_1 &:= \psi^{N-n_0}(H_1) \subset U_1,
\end{split}
\]
where the last inclusion follows from Lemma \ref{l:lemmulo} (ii) and \eqref{e:daric1}.

We must prove that the set
\[
\Lambda:= \bigcap_{j\in \Z} \psi^{Nj}(Q_0\cup Q_1)
\]
which is clearly $\psi^N$-invariant and compact, satisfies statements (i)-(iv) from Theorem \ref{t:horseshoe}. Statement (i) follows from the surjectivity of the coding map in (ii), and this crucial fact is proven below in Lemma \ref{l:coding}. The continuity of the coding map in (ii) follows immediately from the continuity of $\psi^N$ and the fact that $Q_0$ and $Q_1$ have disjoint neighborhoods.

In order to prove (iii), we must show that if $p\in \Lambda$ and $k\in \Z$ is not of the form $k=Nj$ with $\omega_p(j)=1$ then
\begin{equation}
\label{e:daric2}
\psi^k(p) \in U_0 \cup V.
\end{equation}
We write
 \begin{equation}
\label{e:daric3}
k-n_0 = jN + h
\end{equation}
with $j\in \Z$ and $h\in \{0,1,\dots,N-1\}$. If $h=N-n_0$ then $k=(j+1)N$ and hence the assumption on $k$ implies that $\omega_p(j+1)=0$. Therefore,
\[
\psi^k(p) = \psi^{(j+1)N} (p) \in Q_0 \subset U_0,
\]
and claim \eqref{e:daric2} holds. We can then assume that $h\neq N-n_0$. Since $p$ belongs to
\[
\Lambda = \psi^{-n_0} \Bigl( \bigcap_{i\in \Z} \psi^{Ni}(H_0 \cup H_1) \Bigr),
\]
we have
\[
\psi^{Ni+n_0}(p) \in H_0 \cup H_1 \qquad \forall i\in \Z,
\]
and in particular, by \eqref{e:daric3}, 
\[
\psi^{k-h}(p) = \psi^{Nj+n_0}(p) \in H_0 \cup H_1.
\]
If $\psi^{k-h}(p)\in H_0$ then 
\[
\psi^k(p) \in \psi^h(H_0) = h ([0,2^{-h}] \times [0,2^{h-N}]) \subset K_+ = U_0,
\]
and \eqref{e:daric2} holds in this case. If $\psi^{k-h}(p)\in H_1$ then
\[
\psi^k(p) \in \psi^h(H_1) \subset K_+ \cup V = U_0 \cup V,
\]
by Lemma \ref{l:lemmulo} (iv) and the fact that $h\neq N-n_0$. This concludes the proof of (iii). 

We now prove (iv). Assume that the coding $\omega_p\in \{0,1\}^{\Z}$ of a point $p\in \Lambda$ satisfies $\omega_p(j)=0$ for all but a finite number of $j\in \Z$. Then $\psi^{Nj}(p)$ belongs to $Q_0= \psi^{N-n_0}(H_0)$ for $|j|$ large enough. Equivalently, $\psi^{Nj+n_0}(p)\in H_0$ for $|j|$ large enough. Together with the inclusion
\[
\psi^i(H_0) = h\bigl([0,2^{-i}] \times [0,2^{i-N}] \bigr) \subset h\bigl([0,1]\times [0,\tfrac{1}{2}]\bigr) \qquad \forall i\in \{0,1,\dots,N-1\},
\]
this implies that $\psi^k(p) = h(x_k,y_k)$ with $(x_k,y_k)\in [0,1]\times [0,\frac{1}{2}]$ for $|k|$ large enough. By \eqref{e:norm-form}, $x_{k+1} = x_k/2$ and $y_{k+1} = 2 y_k$ for $|k|$ large, and hence $y_k=0$ for $k$ large and $x_k = 0$ for $-k$ large. This implies
\[
\psi^k(p) \in h([0,1]\times \{0\}) \subset W^s_+(p_0), \qquad \psi^{-k}(p) \in h(\{0\} \times [0,1]) \subset W^u_+(p_0)
\]
for $k$ large enough, and hence $p\in W^s_+(p_0)\cap W^u_+(p_0)$. 

Conversely, let $p\in \Lambda \cap W^s_+(p_0)\cap W^u_+(p_0)$. By the definition of $W^s_+(p_0)$ and $W^u_+(p_0)$, there is $i\geq 0$ such that $\psi^i(p) = h(x,0)$ and $\psi^{-i}(p) = h(0,y)$ for some $x,y\in [0,1]$. Then \eqref{e:norm-form} gives
\[
\psi^{k}(p) = h(2^{i-k} x,0), \qquad \psi^{-k}(p) = h(0,2^{i-k} y) \qquad \forall k\geq i,
\]
and hence $\psi^k(p)$ belongs to $h([0,2^{n_0-N}]\times \{0\}) \cup h(\{0\}\times [0,2^{-n_0}]) \subset Q_0$ for $|k|$ large enough. In particular, the coding $\omega_p$ satisfies $\omega_p(j)=0$ for all but a finite number of $j\in \Z$.

There remains to prove the surjectivity of the coding map in (ii), that is, the assertion that for every $\omega\in \{0,1\}^{\Z}$ the set
\[
\bigcap_{j\in \Z} \psi^{-Nj}(Q_{\omega(j)})
\]
is non-empty. Since
\[
\psi^{-Nj} (Q_{\omega(j)}) = \psi^{-n_0}(\psi^{-N(j-1)}(H_{\omega(j)})),
\]
the previous set equals
\[
\psi^{-n_0} \left( \bigcap_{j\in \Z} \psi^{-Nj}(H_{\omega(j+1)}) \right),
\]
and it is enough to show that the set 
\[
\bigcap_{j\in \Z} \psi^{-Nj}(H_{\omega(j)})
\]
is non-empty for every $\omega\in \{0,1\}^{\Z}$. This follows from the next lemma in which, given a rectangular region of the form $H = h([0,1]\times [a,b])$ with $-1\leq a < b \leq 1$, we are using the notation
\[
\partial^+ H := h([0,1]\times \{b\}), \qquad \partial^- H := h([0,1]\times \{a\}),
\]
to denote its upper and lower side.

\begin{lem}
\label{l:coding}
Assume that a homeomorphism $\phi : M \rightarrow M$ satisfies:
\begin{enumerate}[(a')]
\item $\phi(H_0)\subset R$, $\phi(H_1)\subset R$.
\item $\phi(\partial^+ H_0)$ and $\phi(\partial^- H_0)$ are contained in distinct components of
\[
\overline{R\setminus S} = h([0,1] \times [-1,0]) \cup h([0,1]\times [\rho,1]),
\]
and the same holds for $\phi(\partial^+ H_1)$ and $\phi(\partial^- H_1)$.
\end{enumerate}
Then for every $\omega\in \{0,1\}^{\Z}$ the set
\[
\bigcap_{j\in \Z} \phi^{-j}(H_{\omega(j)})
\]
is non-empty.
\end{lem}

Indeed, the homeomorphism $\phi=\psi^N$ satisfies (a') and (b') because of Lemma \ref{l:lemmulo} (ii) and (iii), together with the identities
\[
\begin{split}
\psi^N(H_0) &= h([0,2^{-N}] \times [0,1]), \\
\psi^N(\partial^+ H_0) &= h ([0,2^{-N}] \times \{1\} ), \\
\psi^N(\partial^- H_0) &=  h ([0,2^{-N}] \times \{0\} ).
\end{split}
\]

\begin{proof}
We say that a subset $A\subset R$ \textit{crosses} $S$ if it contains a compact arc with end-points in distinct components of $\overline{R\setminus S}$. 

\paragraph{Claim 1.} If $A\subset R$ crosses $S$, then the sets $\phi(H_0\cap A)$ and $\phi(H_1\cap A)$ are contained in $R$ and cross $S$.

\medskip

Fix $i\in \{0,1\}$. By (a'), $\phi(H_i \cap A)$ is contained in $R$. Let $\alpha$ be a compact arc in $A$ with end-points in distinct components of $\overline{R\setminus S}$. Since $H_i = h([0,1]\times[a,b])$ with $0\le a<b\le\rho$, the intermediate value theorem applied to the second component of $h^{-1}$ along $\alpha$ gives a compact subarc $\alpha'$ which is contained in $H_i \cap A$ and has one end-point in $\partial^+ H_i$ and the other in $\partial^- H_i$. By (b'), $\phi(\alpha')$, which is a compact arc in $\phi(H_i \cap A)$, has end-points in distinct components of $\overline{R\setminus S}$.

\paragraph{Claim 2.} For every $k\in \N$ and every function $\eta: \{1,\dots,k\} \rightarrow \{0,1\}$ the set
\[
\phi(H_{\eta(1)}) \cap \phi^2(H_{\eta(2)}) \cap \dots \cap \phi^k(H_{\eta(k)})
\]
is contained in $R$ and crosses $S$. In particular, it is non-empty.

\medskip

This is proven by induction on $k$. For $k=1$, the claim follows from (a') and (b'). Assume that the claim is true for some $k\in \N$ and choose any function $\eta: \{1,\dots,k+1\} \rightarrow \{0,1\}$. By the inductive assumption, the set
\[
A:= \phi(H_{\eta(2)}) \cap \phi^2(H_{\eta(3)}) \cap \dots \cap \phi^k(H_{\eta(k+1)})
\]
is contained in $R$ and crosses $S$. By Claim 1, the set
\[
\phi( H_{\eta(1)} \cap A ) = \phi(H_{\eta(1)} )\cap \phi(A) =  \phi(H_{\eta(1)}) \cap \phi^2(H_{\eta(2)}) \cap \dots \cap \phi^{k+1}(H_{\eta(k+1)})
\]
is contained in $R$ and crosses $S$. This proves the induction step and concludes the proof of Claim 2.

\medskip

We can now prove the lemma. Since
\[
\bigcap_{j\in \Z} \phi^{-j}(H_{\omega(j)}) = \bigcap_{k\in \N} \Lambda_k,
\]
where 
\[
\Lambda_k := \bigcap_{|j|\leq k} \phi^{-j} (H_{\omega(j)})
\]
is a decreasing sequence of compact sets, it is enough to show that each $\Lambda_k$ is non-empty. This set is indeed non-empty, because the identity
\[
\Lambda_k = \phi^{-k-1} \left( \bigcap_{|j|\leq k} \phi^{k+1-j}(H_{\omega(j)}) \right) = \phi^{-k-1}  \bigl( \phi(H_{\omega(k)}) \cap \phi^2 ( H_{\omega(k-1)} ) \cap\dots \cap \phi^{2k+1} (H_{\omega(-k)}) \bigr)
\]
shows that $\Lambda_k$ is the image by $\phi^{-k-1}$ of a set which by Claim 2 is non-empty.
\end{proof}


\providecommand{\bysame}{\leavevmode\hbox to3em{\hrulefill}\thinspace}
\providecommand{\MR}{\relax\ifhmode\unskip\space\fi MR }
\providecommand{\MRhref}[2]{%
  \href{http://www.ams.org/mathscinet-getitem?mr=#1}{#2}
}
\providecommand{\href}[2]{#2}

\end{document}